\documentclass[11pt, reqno]{amsart}
\usepackage[numeric]{amsrefs}
\DeclareRobustCommand{\VAN}[3]{#2}
\usepackage[utf8]{inputenc}
\usepackage{tikz-cd}
\usepackage{hyperref}
\usepackage[OT2,T1]{fontenc}
\DeclareSymbolFont{cyrletters}{OT2}{wncyr}{m}{n}
\DeclareMathSymbol{\Sha}{\mathalpha}{cyrletters}{"58}

\usepackage[all,cmtip]{xy}
\usepackage{amsmath}
\usepackage{mathtools} 
\usepackage{amsfonts}
\usepackage{amssymb} 
\usepackage{amsthm}
\usepackage{mathrsfs}
\usepackage{enumitem}

\calclayout

\newcommand\restr[2]{{
  \left.\kern-\nulldelimiterspace 
  #1 
  \vphantom{\big|} 
  \right|_{#2} 
  }}

\mathtoolsset{showonlyrefs}
\numberwithin{equation}{subsection}
\newtheorem{Thm}[subsubsection]{Theorem}
\newtheorem{Lem}[subsubsection]{Lemma}
\newtheorem{Prop}[subsubsection]{Proposition}
\newtheorem{Cor}[subsubsection]{Corollary}

\newtheorem{mainThm}{Theorem}

\theoremstyle{definition}
\newtheorem{Def}[subsubsection]{Definition}
\newtheorem{Hyp}[subsubsection]{Hypothesis}

\newtheorem{eg}[subsubsection]{Example}
\newtheorem{Rem}[subsubsection]{Remark}

\newcommand{\qpbr}{\breve{\mathbb{Q}}_p}
\newcommand{\qpur}{\qp^{\mathrm{ur}}}
\newcommand{\zpbr}{\breve{\mathbb{Z}}_p}
\newcommand{\zpur}{\zp^{\mathrm{ur}}}
\newcommand{\qp}{\mathbb{Q}_p}
\newcommand{\zp}{\mathbb{Z}_{p}}
\newcommand{\fp}{\mathbb{F}_p}
\newcommand{\fpbar}{\overline{\mathbb{F}}_p}
\newcommand{\zlocp}{\mathbb{Z}_{(p)}}
\newcommand{\ql}{\mathbb{Q}_{\ell}}
\newcommand{\gal}{\operatorname{Gal}}
\newcommand{\afp}{\mathbb{A}_f^p}
\newcommand{\af}{\mathbb{A}_f}
\newcommand{\ad}{\mathbb{A}}
\newcommand{\qbar}{\overline{\mathbb{Q}}}
\newcommand{\qpbar}{\overline{\mathbb{Q}}_p}
\newcommand{\spec}{\operatorname{Spec}}
\newcommand{\spf}{\operatorname{Spf}}
\newcommand{\spa}{\operatorname{Spa}}
\newcommand{\spd}{\operatorname{Spd}}
\newcommand{\perf}{\operatorname{Perf}}
\newcommand{\ovfp}{\overline{\mathbb{F}}_p}
\newcommand{\grg}{\operatorname{Gr}_{\mathcal{G}}}
\newcommand{\grgmu}{\operatorname{Gr}_{G, \mu}}

\newcommand{\shtg}{\operatorname{Sht}_{\mathcal{G}}}

\newcommand{\gad}{\mathsf{G}^{\mathrm{ad}}}
\newcommand{\calgad}{\mathcal{G}^{\mathrm{ad}}}
\newcommand{\calgadcirc}{\mathcal{G}^{\mathrm{ad}, \circ}}
\newcommand{\calgcirc}{\mathcal{G}^{\circ}}
\newcommand{\calhad}{\mathcal{H}^{\mathrm{ad}}}
\newcommand{\calhadcirc}{\mathcal{H}^{\mathrm{ad}, \circ}}

\newcommand{\calhocirc}{\mathcal{H}_1^{\circ}}

\newcommand{\shtgmu}{\operatorname{Sht}_{\mathcal{G},\mu}}
\newcommand{\shtgmuone}{\mathrm{Sht}_{\mathcal{G},\mu, \delta=1}}
\newcommand{\shtgcircmu}{\operatorname{Sht}_{\calgcirc,\mu}}

\newcommand{\shtgadcircmu}{\operatorname{Sht}_{\calgadcirc,\mu}}

\newcommand{\shthadcircmu}{\operatorname{Sht}_{\calhadcirc,\mu}}
\newcommand{\shtgp}{\operatorname{Sht}_{\mathcal{G}'}}
\newcommand{\shtgpmup}{\operatorname{Sht}_{\mathcal{G}',\mu'}}

\newcommand{\shthmu}{\operatorname{Sht}_{\mathcal{H},\mu}}
\newcommand{\shthomu}{\operatorname{Sht}_{\mathcal{H}_1,\mu}}
\newcommand{\shthomuone}{\operatorname{Sht}_{\mathcal{H}_1,\mu,\delta=1}}
\newcommand{\shthopmu}{\operatorname{Sht}_{\mathcal{H}_1',\mu}}

\newcommand{\gder}{G^{\mathrm{der}}}

\newcommand{\mintgb}{\mathcal{M}^{\mathrm{int}}_{\mathcal{G},b,\mu}}

\newcommand{\mintgadcircbxnaught}{\mathcal{M}^{\mathrm{int}}_{\calgadcirc,b_x,\mu,/x_0}}
\newcommand{\minthadcircbxnaught}
{\mathcal{M}^{\mathrm{int}}_{\calhadcirc,b_x,\mu,/x_0}}

\newcommand{\minthb}{\mathcal{M}^{\mathrm{int}}_{\mathcal{H},b,\mu}}

\newcommand{\mintgbxnaught}{\mathcal{M}^{\mathrm{int}}_{\mathcal{G},b_x,\mu,/x_0}}
\newcommand{\mintgbpxnaught}{\mathcal{M}^{\mathrm{int}}_{\mathcal{G},b',\mu,/x_0}}

\newcommand{\perfpairs}{\mathcal{D}^{\circ}}
\newcommand{\perfpairsrat}{\mathcal{D}}
\newcommand{\perfstacks}{\operatorname{Stk}_{\operatorname{Perf}}}
\newcommand{\bdtimes}{\buildrel{\boldsymbol{.}}\over\times}
\newcommand{\frob}{\mathrm{Frob}}
\newcommand{\bun}{\operatorname{Bun}}
\newcommand{\bungmu}{\operatorname{Bun}_{G, \mu^{-1}}}
\newcommand{\bunhomu}{\operatorname{Bun}_{H_1,\mu^{-1}}}

\newcommand{\shtrp}{\operatorname{ShTrp}}
\newcommand{\shtpr}{\operatorname{ShtPr}}
\newcommand{\igs}{\operatorname{Igs}}
\newcommand{\scrs}{\mathscr{S}}
\newcommand{\bgmu}{B(G,\mu^{-1})}
\newcommand{\locmodgmu}{\mathbb{M}_{\mathcal{G},\mu}}

\newcommand{\locmodgmuv}{\mathbb{M}_{\mathcal{G},\mu}^{\mathrm{v}}}
\newcommand{\locmodhmuv}{\mathbb{M}_{\mathcal{H},\mu}^{\mathrm{v}}}
\newcommand{\locmodgolpmuv}{\mathbb{M}_{\mathcal{G}_{\mathcal{O}_{L'}},\mu}^{\mathrm{v}}}
\newcommand{\locmodhmu}{\mathbb{M}_{\mathcal{H},\mu}}
\newcommand{\locmodhomu}{\mathbb{M}_{\mathcal{H}_1,\mu}}

\newcommand{\locmodgvfmu}{\mathbb{M}_{\mathcal{G}_{V,F},\mu}}
\newcommand{\locmodholpmuv}{\mathbb{M}_{\mathcal{H}_{\mathcal{O}_{L'}},\mu}^{\mathrm{v}}}
\newcommand{\locmodhmuGamma}{\mathbb{M}_{\mathcal{H},\mu}^{\Gamma}}
\newcommand{\locmodhmuvGamma}{\mathbb{M}_{\mathcal{H},\mu}^{\mathrm{v},\Gamma}}
\newcommand{\locmodgadmu}{\mathbb{M}_{\calgad,\mu}}
\newcommand{\locmodgadmuv}{\mathbb{M}_{\calgad,\mu}^{\mathrm{v}}}
\newcommand{\locmodhadmu}{\mathbb{M}_{\calhad,\mu}}

\newcommand{\locmodgmuStack}{\left[\locmodgmu / \mathcal{G}\right]}

\newcommand{\locmodhmuStackGamma}{\left[\locmodhmu / \mathcal{H} \right]^{h \Gamma}}

\newcommand{\locmodgadmuStack}{\left[\locmodgadmu / \calgad \right]}

\newcommand{\locmodhadmuStack}{\left[ \locmodhadmu / \calhad\right]}

\newcommand{\locmodhadmuStackGamma}{\left[\locmodhadmu / \calhad \right]^{h \Gamma}}

\newcommand{\g}{\mathsf{G}}
\newcommand{\tor}{\mathsf{T}}
\newcommand{\h}{\mathsf{H}}
\newcommand{\ho}{\mathsf{H}_{1}}
\newcommand{\gp}{\mathsf{G}'}

\newcommand{\hafp}{\mathsf{H}(\afp)}
\newcommand{\haf}{\mathsf{H}(\af)}
\newcommand{\hy}{{(\mathsf{H}, \mathsf{Y})}}

\newcommand{\hoaf}{\mathsf{H}_1(\af)}
\newcommand{\hyo}{{(\mathsf{H}_1, \mathsf{Y}_1)}}

\newcommand{\gafp}{\mathsf{G}(\afp)}
\newcommand{\gaf}{\mathsf{G}(\af)}
\newcommand{\gx}{{(\mathsf{G}, \mathsf{X})}}
\newcommand{\gxg}{{(\mathsf{G}, \mathsf{X}, \mathcal{G})}}

\newcommand{\gv}{\mathsf{G}_{V}}
\newcommand{\gw}{\mathsf{G}_{W}}
\newcommand{\gwx}{(\mathsf{G}_{W},\mathsf{H}_{W})}
\newcommand{\gvf}{\mathsf{G}_{V,\f}}
\newcommand{\gvfx}{(\gvf, \mathsf{H}_{V, \f})}

\newcommand{\gvx}{(\mathsf{G}_{V},\mathsf{H}_{V})}
\newcommand{\gxp}{(\mathsf{G}', \mathsf{X}')}
\newcommand{\calhcirc}{\mathcal{H}^{\circ}}
\newcommand{\f}{\mathsf{F}}
\newcommand{\rH}{\mathrm{H}}

\thanks{J.S. was partially supported by funding from the European Research Council (ERC) under the European Union's Horizon 2020 research and innovation program (grant agreement No. 884596).}

\title[On the Piatetski-Shapiro construction]{On the Piatetski-Shapiro construction for integral models of Shimura varieties}
\author{Pol van Hoften and Jack Sempliner}

\begin{document}

\begin{abstract}
We study the Piatetski-Shapiro construction, which takes a totally real field $\f$ and a Shimura datum $\gx$ and produces a new Shimura datum $\hy$. If $\f$ is Galois, then the Galois group $\Gamma$ of $\f$ acts on $\hy$, and we show that the $\Gamma$-fixed points of the Shimura varieties for $\hy$ recover the Shimura varieties for $\gx$ under some hypotheses. For Shimura varieties of Hodge type with parahoric level, we show that the same is true for the $p$-adic integral models constructed by Pappas--Rapoport. Along the way, we study the $\Gamma$-fixed points of the Igusa stacks of \cites{ZhangThesis, DvHKZIgusaStacks} for $\hy$ and prove optimal results.
\end{abstract}
\maketitle
\setcounter{tocdepth}{2}
\tableofcontents

\section{Introduction}
Let $\gx$ be a Shimura datum of Hodge type and let $\f$ be a Galois totally real field with Galois group $\Gamma = \text{Gal}(\f/\mathbb{Q})$. Define $\h:=\operatorname{Res}_{\f/\mathbb{Q}} \g_{\f}$ and let $\mathsf{Y}$ be the Shimura datum for $\h$ induced by $\mathsf{X}$; we call $\hy$ the \emph{Piatetski-Shapiro construction} for $\gx$ associated to $\f$. This is \emph{not} of Hodge type if $[\f:\mathbb{Q}]>1$, but there is a modified Shimura datum $\hyo \xhookrightarrow{} \hy$ that is of Hodge type, see Section \ref{Sec:FixedPointsHodge}. Then $\Gamma$ acts on $\hyo$ and there is a morphism $\gx \to \hyo$ that is $\Gamma$-equivariant. Thus if $K \subset \hoaf$ is a $\Gamma$-stable compact open subgroup, then there is a morphism of Shimura varieties $\mathbf{Sh}_{K^{\Gamma}}\gx \to \mathbf{Sh}_{K}\hyo$, which induces a map (where the superscript $\Gamma$ denotes taking $\Gamma$-fixed points)
\begin{align}
    \mathbf{Sh}_{K^{\Gamma}}\gx \to \mathbf{Sh}_{K}\hyo^{\Gamma}.
\end{align}
This morphism can be arranged to be an isomorphism under minor hypotheses, see Theorem \ref{Thm:FixedPointsHodge}. In this paper, we will prove extensions of this result to $p$-adic integral models. We hope that this link between the geometry of the integral model associated to $\g$ and that associated to $\ho$ will prove to be a useful general tool in the arithmetic geometry of Shimura varieties. As a first step, our main results will be used in joint work in progress of the authors \cite{SSGI2} to construct new exotic Hecke correspondences between the special fibers of different Shimura varieties.

\subsection{Main results} Let $\gx$ be a Shimura datum of Hodge type with reflex field $\mathsf{E}$. Fix a prime $p$ and a prime $v$ of $\mathsf{E}$ above $p$, let $E$ be the completion of $\mathsf{E}$ at $v$ and let $\mathcal{O}_{E}$ be its ring of integers. Let $G=\mathsf{G} \otimes \qp$ and let $\mathcal{G}$ be a parahoric model of $G$ over $\zp$. 

Let $\f$ and $\Gamma$ be as above, let $F=\f \otimes \qp$ and let $\mathcal{O}_F$ be $\mathcal{O}_{\f} \otimes \zp$. Let $\mathcal{H}_1$ be the corresponding (quasi-)parahoric model of $H_1$ (see Section \ref{sub:ProofIgusa} for the precise construction). Let $K_p=\mathcal{H}(\zp)$ and $K_{1,p}=\mathcal{H}_1(\zp)$. We will write $\scrs_{K_{1,p}}\hyo$ for the integral model over $\mathcal{O}_E$ of $\mathbf{Sh}_{K_{1,p}}\hyo$ constructed by \cite[Theorem 4.2.3]{DanielsVHKimZhangII}, cf. \cite[Theorem 4.5.2]{PappasRapoportShtukas}. \smallskip 

Let $\scrs_{K_{1,p}}\gx^{\lozenge/}$ be the v-sheaf associated to $\scrs_{K_{1,p}}\gx$, see Section \ref{Sec:DiamondFunctors}. This is characterized by the existence of a morphism (functorial in the tuple $(\h_{1},\mathsf{Y},\mathcal{H}_{1})$) of v-stacks
\begin{align}
    \scrs_{K_{1,p}}\gx^{\lozenge/} \to \shthomu.
\end{align}
In particular, this morphism is $\Gamma$-equivariant for the natural $\Gamma$-action on the source and target. We similarly have $\scrs_{K_{p}}\gx^{\lozenge/} \to \shtgmu$ and the obvious square is $2$-commutative.
\begin{mainThm}[Theorem \ref{Thm:FixedPointsIntegralHodgeRamified}] \label{Thm:FixedPointsVsheavesIntroduction}
If $\Sha^1(\mathbb{Q}, \mathsf{G}) \to \Sha^1(\f, \g)$ is injective, 
then the following diagram of v-stacks is $2$-Cartesian
\begin{equation}
    \begin{tikzcd}
    \scrs_{K_p^{\Gamma}}\gx^{\lozenge/} \arrow{r} \arrow{d} & \scrs_{K_{p,1}}\hyo^{ \Gamma, \lozenge/} \arrow{d} \\
    \shtgmu \arrow{r} & \shthomu^{h \Gamma}.
    \end{tikzcd}
\end{equation}
Here the superscript $h\Gamma$ denotes the stacky (or homotopy) fixed points of a stack and the superscript $\Gamma$ denotes the usual fixed points of a sheaf.
\end{mainThm}
The assumption that $\Sha^1(\mathbb{Q}, \mathsf{G}) \to \Sha^1(\f, \g)$ is injective is necessary, see Section \ref{Sec:Tori} for an exploration of the theorem in the case of tori. \smallskip

\subsubsection{} The statement of Theorem \ref{Thm:FixedPointsVsheavesIntroduction} is essentially optimal, but it does not directly address when the map of schemes $\scrs_{K_p}\gx \to \scrs_{K_{p,1}}\hyo^{\Gamma}$ is an isomorphism. We address this more concrete question in the following theorem, phrased for finite level Shimura varieties. For a $\Gamma$-stable and neat compact open subgroup $K^p \subset \hafp$ we write $K_1^p=K_1^p \cap H_1(\afp)$ and $K_1=K_{1}^p K_{1,p}$.
\begin{mainThm}[Theorem \ref{Thm:FixedPointsCanonicalIntegralModelsHodge}] \label{Thm:FixedpointIntegralIntroduction}
Assume that $\f$ is tamely ramified over $\mathbb{Q}$, that $\Sha^1(\mathbb{Q}, \mathsf{G}) \to \Sha^1(\f, \g)$ is injective, that $p>2$ and that $p$ is unramified in $\f$. The following results hold. \begin{enumerate}
    \item There is a cofinal collection of $\Gamma$-stable compact open subgroups $K^p \subset \hafp$ such that the natural map
\begin{align} \label{Eq:IntroNaturalMap}
    \scrs_{K^{\Gamma}}\gx \to \scrs_{K_1}\hyo^{\Gamma} 
\end{align}
is a universal homeomorphism. 
\item If $p$ is coprime to $ |\Gamma| \cdot | \pi_1(\gder)|$, and if $G$ splits over a tamely ramified extension, then there is a cofinal collection of $K^p$ as above such that the natural map \eqref{Eq:IntroNaturalMap} is an isomorphism. \item If $K_p$ is hyperspecial, then there is a cofinal collection of $K^p$ as above such that the natural map \eqref{Eq:IntroNaturalMap} is an isomorphism.
\end{enumerate}
\end{mainThm}
The assumption that $G$ splits over a tamely ramified extension in part (2) can be weakened, see Hypothesis \ref{Hyp:InducedRamification}. 
\begin{eg}
If $\gx=(\operatorname{GL}_2,\mathsf{H}^{\pm})$, then $\hyo$ is the subgroup of $\operatorname{Res}_{\f/\mathbb{Q}} \operatorname{GL}_2$ consisting of those matrices with determinant in $\mathbb{G}_m \subset \operatorname{Res}_{\f/\mathbb{Q}} \mathbb{G}_m$. The $p$-adic integral models for the Shimura varieties for $\hyo$ have a moduli interpretation in terms of (weakly polarized) abelian varieties $A$ of dimension $[\f:\mathbb{Q}]$ up to prime-to-$p$ isogeny, equipped with an action $i:\mathcal{O}_{\f,(p)} \to \operatorname{End}(A)$. The action of an element $\gamma \in \Gamma$ is then by precomposing $i$ with $\gamma:\mathcal{O}_{\f,(p)} \to \mathcal{O}_{\f,(p)}$. The morphism from the modular curve can be thought of as taking an elliptic curve up to prime-to-$p$ isogeny $E$, and forming the abelian variety up to isogeny $E \otimes_{\mathbb{Z}_{(p)}} \mathcal{O}_{\f,(p)}$ together with its tautological $\mathcal{O}_{\f,(p)}$ action. The statement of Theorem \ref{Thm:FixedpointIntegralIntroduction}, up to keeping track of level structures, is then essentially an instance of \'etale Galois descent of modules for the \'etale cover $\spec \mathcal{O}_{\f,(p)} \to \spec \mathbb{Z}_{(p)}$.\footnote{Indeed, one applies this to $A$ considered as a sheaf on the category of schemes over $\zp$, with target the category of $\mathcal{O}_{\f,(p)}$-modules. Descent then tells us that a $\Gamma$-descent datum on $A$ is induced from unique sheaf $E$ of $\zlocp$-modules on the category of schemes over $\zp$, by $A= E \otimes_{\zlocp} \mathcal{O}_{\f,(p)}$.} If $p$ is \emph{not} unramified in $\mathsf{F}$, then $\spec \mathcal{O}_{\f,(p)} \to \spec \mathbb{Z}_{(p)}$ is not a Galois cover with Galois group $\Gamma$, and so this argument does not work. 
\end{eg}

\subsubsection{Igusa stacks} Igusa stacks are certain $p$-adic analytic objects (Artin v-stacks) associated to a Shimura datum $\gx$; they were conjectured to exist by Scholze. They were recently constructed by Zhang \cite{ZhangThesis} in the PEL type case, and in \cite{DvHKZIgusaStacks} in the Hodge type case. 

Let $\gx$ be a Shimura datum of Hodge type and a place $v$ above $p$ of the reflex field $\mathsf{E}$ of $\gx$. Let $\f, \Gamma$ and $\hyo$ be as above. Then there is a v-sheaf $\igs \hyo$ equipped with a $\Gamma$-equivariant map $\igs \hyo \to \bun_{H_1}$ to the stack $\bun_{H_1}$ of $H_1$-bundles on the Fargues--Fontaine curve, see \cite[Theorem I]{DvHKZIgusaStacks}. This map moreover factors through the open substack $\bunhomu \subset \bun_{H_1}$ corresponding to the set $B(H_1, \mu^{-1})\subset B(H_1)$ of $\mu^{-1}$-admissible $\sigma$-conjugacy classes, where $\mu$ is the $H_1(\qpbar)$ conjugacy class of cocharacters of $H_1$ induced by the Hodge cocharacter and the place $v$. We have similar objects for $\gx$.
\begin{mainThm}[Theorem \ref{Thm:FixedPointsIgusa}] \label{Thm:IntroIgusa}
    If $\Sha^1(\mathbb{Q}, \mathsf{G}) \to \Sha^1(\f, \g)$ is injective, then the natural map
    \begin{align}
        \igs \gx \to \igs \hyo^{\Gamma} \times_{\left(\bunhomu\right)^{h \Gamma}} \bungmu
    \end{align}
    is an isomorphism.
\end{mainThm}
Although we believe Theorem \ref{Thm:IntroIgusa} to be of significant independent interest, it is also the crucial input in proving Theorem \ref{Thm:FixedPointsVsheavesIntroduction}. Indeed, the proof of Theorem \ref{Thm:FixedPointsVsheavesIntroduction} roughly proceeds first by proving that the fixed points of the v-sheaf associated to the rigid generic fiber of the Shimura variety are as expected, then using this input to prove Theorem \ref{Thm:IntroIgusa}, then finally using this to deduce Theorem \ref{Thm:FixedPointsVsheavesIntroduction}. In this procedure the Igusa stack is the key linchpin which allows us to pass from information about fixed points on the generic fiber to information about the fixed points of the integral model. See Section \ref{sub:proofsIntro} for more details.  

\subsection{Motivation} Our motivation for proving Theorems \ref{Thm:FixedpointIntegralIntroduction} and \ref{Thm:FixedPointsVsheavesIntroduction} is that they can be used to reduce questions about the integral models for $\gx$ to $\Gamma$-equivariant questions about the integral models for $\hy$ or $\hyo$. This is useful, because given $\gx$ and a prime $p$, one can always choose a Galois totally real field $\f$ that is unramified at $p$ such that $H \otimes \qp $ and $H_1 \otimes \qp$ are quasi-split. \smallskip

For example, the Langlands--Rapoport conjecture for $\gx$ is out of reach when $G \otimes \qp$ is not quasi-split, because of the non-existence of CM lifts of isogeny classes. We expect that it is possible to establish a weak version of the Langlands--Rapoport $\tau$ conjecture of \cite{KisinShinZhu} at non quasi-split primes, by reducing to the quasi-split case via Theorems \ref{Thm:FixedpointIntegralIntroduction}. \smallskip 

A second example: In forthcoming work, Xiao and Zhu \cite{XiaoZhu, XiaoZhu2} will construct exotic Hecke correspondences between the mod $p$ fibers of different Shimura varieties of Hodge type at unramified and quasi-split primes. In their work, the Shimura varieties correspond to certain Shimura data $\gx$ and $\gxp$ such that $\g$ and $\gp$ are pure inner forms and such that $\g \otimes \af \simeq \gp \otimes \af$. Their proof relies heavily on the fact that both $\g \otimes \qp$ and $\g' \otimes \qp$ are quasi-split, through the existence of CM lifts of isogeny classes.

These correspondences conjecturally exist for certain pairs $\gx$ and $\gxp$, where $\g$ and $\gp$ are pure inner forms, under the more general condition that $\g \otimes \afp \simeq \gp \otimes \afp$. In work in preparation of the authors, see \cite{SSGI2}, we construct these more general exotic Hecke correspondences for primes $p$ where $G$ splits over an unramified extension (but is not necessarily quasi-split), by reducing to the quasi-split case via Theorem \ref{Thm:FixedpointIntegralIntroduction}. 

\subsection{Proofs of the main theorems} \label{sub:proofsIntro} The first step in the proofs of Theorems \ref{Thm:FixedPointsVsheavesIntroduction}, \ref{Thm:FixedpointIntegralIntroduction} and \ref{Thm:IntroIgusa} is to compute the $\Gamma$-fixed points of Shimura varieties over $\mathbb{C}$. 

To compute these fixed points, we argue on the level of $\mathbb{C}$-points and reduce to a concrete question about fixed points of adelic double quotients. To tackle the resulting questions, we use the methods of non-abelian group cohomology. These general methods tell us that the natural morphism of action groupoids\footnote{The group $\hy$ does typically not satisfy Milne's axiom SV5, see Lemma \ref{Lem:SV5}, and so the action of $\h(\mathbb{Q})$ on $\mathsf{Y} \times \haf / K$ is typically not free even for sufficiently small $K$.}
\begin{align}
     \left[\g(\mathbb{Q}) \backslash \mathsf{X} \times \gaf/K^{\Gamma}\right] \to \left[\h(\mathbb{Q}) \backslash \mathsf{Y} \times \haf / K\right]^{h\Gamma}
\end{align}
is an equivalence if $\Sha^1(\mathbb{Q}, \mathsf{G}) \to \Sha^1(\f, \g)$ is injective and $\rH^1(\Gamma,K)=\{1\}$. If $\f$ is tamely ramified over $\mathbb{Q}$, then we prove that there is a cofinal collection of $\Gamma$-stable $K \subset \haf$ with $\rH^1(\Gamma, K)$ trivial, see Section \ref{Sec:GoodSubgroups}. Using some tricks, we then use this to compute the fixed points of $\mathbf{Sh}_{K_1}\hyo$, where $K_1 = K \cap \ho(\af)$, see Section \ref{Sec:FixedPointsHodge}.

\subsubsection{} \label{Sec:ArgumentsIntegral} We now explain how to prove Theorem \ref{Thm:FixedpointIntegralIntroduction}. Since we know what happens on the generic fiber, we observe that it suffices to prove that $\scrs_{K_1}\hyo^{\Gamma}$ is normal and flat over $\spec \mathcal{O}_{E}$. If the order of $\Gamma$ is prime-to-$p$, then we will prove this by taking $\Gamma$-fixed points of the local model diagram for $\scrs_{K_1}\hyo$ constructed by \cite{KisinPappas}, and using the fact that fixed points of smooth morphisms are again smooth. The problem then reduces to showing that $\Gamma$-fixed points of the local models for $\hyo$ give the local models for $\gx$. We deduce this from unramified base change for local models. Here we crucially use the assumption that $p$ is unramified in $\f$. It would be interesting to understand the $\Gamma$-fixed points of the local models for $\hyo$ when $\f$ is tamely ramified at $p$. \smallskip

When $p$ divides the order of $\Gamma$, we can only show that the weak normalization of $\scrs_{K_1}\hyo^{\Gamma}$ is normal and flat over $\spec \mathcal{O}_{E}$. For this, we may argue on the level of the corresponding v-sheaves, where we reduce it to proving that the $\Gamma$-fixed points of the integral local Shimura varieties for $H$ give the integral local Shimura varieties for $G$. Here we crucially use the assumption that $p$ is unramified in $\f$, as in this setting it is straightforward to calculate the fixed points of the relevant integral local Shimura varieties via the fixed points of stacks of local shtukas. On the other hand, using Theorem \ref{Thm:FixedPointsVsheavesIntroduction}, if this local problem can be solved in greater generality one could obtain a concomitant strengthening of Theorem \ref{Thm:FixedpointIntegralIntroduction}.

\subsubsection{} \label{Sec:ArgumentsIntegralII} We now explain how to prove Theorem \ref{Thm:IntroIgusa}; let us assume for simplicity that the Shimura varieties in question are proper. The Igusa stack for $\hyo$ sits by \cite[Thm I]{DvHKZIgusaStacks} in a $2$-Cartesian diagram
\begin{equation}
    \begin{tikzcd}
        \mathbf{Sh}\hyo_{E}^{\lozenge} \arrow{r}{} \arrow{d} & \operatorname{Gr}_{H_{1}, \mu^{-1}} \arrow{d} \\
        \igs \hyo \arrow{r} & \bunhomu.
    \end{tikzcd}
\end{equation}
We now take $\Gamma$-homotopy-fixed points of the four objects in the square (the resulting square is again $2$-Cartesian by Lemma \ref{Lem:FiberProducts}). The top row of the fixed point diagram is given by $\mathbf{Sh}\gx_{E}^{\lozenge} \to \operatorname{Gr}_{G,\mu^{-1}}$; this uses our fixed point result for the generic fiber of the Shimura variety and is straightforward for the target of the Hodge--Tate period map. If we base change via $\bungmu \to \bunhomu$, we get the $2$-Cartesian diagram
\begin{equation}
    \begin{tikzcd}
        \mathbf{Sh}\gx_{E}^{\lozenge} \arrow{r}{} \arrow{d} & \operatorname{Gr}_{G, \mu^{-1}} \arrow{d} \\
        (\igs \hyo)^{\Gamma} \times_{\bunhomu^{h \Gamma}} \bungmu \arrow{r} & \bungmu.
    \end{tikzcd}
\end{equation}
The Igusa stack for $\gx$ sits in the same Cartesian diagram via the natural map $\igs \gx \to \igs \hyo$. To see that this natural map is an isomorphism, we use the fact that $\operatorname{Gr}_{G, \mu^{-1}} \to \bungmu$ is a covering map to reduce it to checking that $\mathbf{Sh}\gx_{E}^{\lozenge} \to \mathbf{Sh}\hyo_{E}^{\Gamma,\lozenge}$ is an isomorphism. This latter map is an isomorphism by the argument of Section \ref{Sec:ArgumentsIntegral}. 

\subsubsection{} In the proper case, Theorem \ref{Thm:FixedPointsVsheavesIntroduction} follows from Theorem \ref{Thm:IntroIgusa} and the compatible $2$-Cartesian diagrams of \cite[Thm VII, Remark 6.0.2]{DvHKZIgusaStacks}
\begin{equation} \label{eq:IntroFiberProductDiagram}
\begin{tikzcd}
        \scrs_{K_{p}}\gx^{\lozenge/} \arrow{r} \arrow{d} & \shtgmu \arrow{d}{} \\
        \igs \gx \arrow{r} & \bungmu
    \end{tikzcd}
     \begin{tikzcd}
        \scrs_{K_{p,1}}\hyo^{\lozenge/} \arrow{r} \arrow{d} & \shthomuone \arrow{d}{} \\
        \igs \hyo \arrow{r} & \bunhomu.
    \end{tikzcd}
\end{equation}
The subscript $\delta=1$ in $\shthomuone$ deals with the technical issue that $\mathcal{H}_{1}$ could genuinely be a quasi-parahoric and not a parahoric. Thus the only property of integral models going into our proof of Theorem \ref{Thm:FixedpointIntegralIntroduction} is the fiber product diagram of \eqref{eq:IntroFiberProductDiagram}.

\subsubsection{} In order to facilitate the arguments outlined above, it is important for us to verify that many constructions in the theory of integral models of Shimura varieties (shtukas, local models, etc.) are functorial in the triple $\gxg$ in a $2$-categorical sense. For example, we show that the morphisms $\mathscr{S}_K\hy^{\lozenge/} \to \shthmu$ from the (modified) diamond associated to an integral model of a Shimura variety, see Section \ref{Sec:DiamondFunctors}, to the stack of $\mathcal{G}$-shtukas of type $\mu$ of Pappas--Rapoport \cite{PappasRapoportShtukas}, can be upgraded to a weak natural transformation of weak functors (or pseudo-functors), see Proposition \ref{Prop:FunctorialityShimuraVarietiesToShtukas} and Corollary \ref{Cor:FunctorialityShimuraVarietiesToShtukasII}. This is necessary because this implies that there is an induced map
\begin{align}
    \mathscr{S}_K\hy^{\lozenge/,\Gamma} \to \left(\shthmu\right)^{h \Gamma},
\end{align}
which features in the statement of Theorem \ref{Thm:FixedPointsVsheavesIntroduction}. In the proofs of Theorems \ref{Thm:FixedPointsVsheavesIntroduction} and \ref{Thm:IntroIgusa}, we make use of the $2$-categorical theory of non-abelian Galois cohomology theory for the $(2,1)$-category of stacks on a site. We develop this theory from scratch in Appendix \ref{Appendix:HFP}.

\subsection{Outline of the paper} In Section \ref{Sec:Prelim} we discuss 
preliminaries on perfectoid geometry. We also recall local models and local shtukas and compute the fixed points of local models and integral local Shimura varieties. In Section \ref{Sec:Shimura} we prove Theorem \ref{Thm:FixedPointsVsheavesIntroduction} on the generic fiber. In Section \ref{Sec:IntegralModels} we discuss integral models of Shimura varieties following Pappas--Rapoport \cite{PappasRapoportShtukas}, and construct a $2$-categorical enhancement of their shtukas. In Section \ref{Sec:LocalModelDiagrams}, we discuss local model diagrams for Shimura varieties of Hodge type following \cite{KisinZhou}, and prove $\Gamma$-equivariance. In Section \ref{Sec:MainTheoremI} we prove Theorem \ref{Thm:FixedpointIntegralIntroduction}. In Section \ref{Sec:IgusaStacks} we introduce Igusa stacks, prove Theorem \ref{Thm:IntroIgusa} and deduce Theorem \ref{Thm:FixedPointsVsheavesIntroduction}. 

In Appendix \ref{Appendix:HFP} we develop some 2-category theory to be used throughout the paper. The main question we answers is when taking (2-categorical) fixed points commutes with taking (2-categorical) quotients.

\subsection{Acknowledgements} This project started while both authors were attending the 2022 IHES summer school on the Langlands Program, and we would like to thank the IHES for providing excellent working conditions. We would also like to thank Ana Caraiani, Brian Conrad, Sean Cotner, Toby Gee, Richard Taylor, Lie Qian, Rong Zhou and Xinwen Zhu for helpful discussions. In addition, large portions of this work were written up while the first author was a visitor at Imperial College London, and while the second author was a visitor at Stanford University. The authors would like to thank both of these institutions for their hospitality.

\section{Preliminaries} \label{Sec:Prelim}

\subsection{\texorpdfstring{$2$}{2}-category theory} At many points in this article, we will be taking fixed points for the action of a finite group $\Gamma$ on a stack $\mathcal{X}$. The most reasonable way of doing this seems to be taking the $2$-categorical or homotopy fixed points. For example, if $\f$ is a finite Galois extension of $\mathbb{Q}$ with Galois group $\Gamma$, then $\Gamma$ acts on the category of $\f$-vector spaces by twisting the $\f$-action, and the $\Gamma$-homotopy fixed points can be identified with the category of $\mathbb{Q}$-vector spaces; this is a restatement of Galois descent along $\spec \f \to \spec \mathbb{Q}$.

We will refer the reader to Appendix \ref{Appendix:HFP} for the definition of a stack (or category) with $\Gamma$-action and the notion of a $\Gamma$-equivariant morphism of stacks (or categories) with $\Gamma$-action, see Definition \ref{Def:GammaObject}. In particular, we note that it is \emph{not} a property of a morphism to be $\Gamma$-equivariant, but rather an extra structure. To emphasize this, we will sometimes call such morphisms \emph{$\Gamma$-equivariant in the $2$-categorical sense}.

\subsection{Non-abelian cohomology} Given a group $\Gamma$ acting on a (possibly non-abelian) group $H$, we denote by $\rH^1(\Gamma,H)$ the pointed set given by $1$-cocycles $\sigma:\Gamma \to H$ up to $H$-conjugacy. We begin with the following simple non-abelian Shapiro's lemma, which appears in \cite[Section 5.8.(b)]{SerreGaloisCohomology} as an exercise left to the reader. 
\begin{Lem} \label{Lem:Shapiro}
Let $\Gamma$ be a group and let $H$ be a group with an action of a subgroup $\Gamma'$ of $\Gamma$. Then if $X := \text{Map}_{\Gamma'}(\Gamma, H)$ with the natural left action $\gamma_0((h_{\gamma})_{\gamma \in \Gamma}) = (h_{\gamma \cdot \gamma_0})_{\gamma \in \Gamma}$ we have a natural isomorphism $\psi: \rH^1(\Gamma, X) \xrightarrow{\sim} \rH^1(\Gamma', H)$ . 
\end{Lem}
\subsection{Background on perfectoid geometry} In this section we recall some background on perfectoid spaces. We refer the reader to \cite{PappasRapoportShtukas} and \cite{ScholzeWeinsteinBerkeley} for details.

Let $k$ be a perfect field of characteristic $p$ and write $\perf_{k}$ for the category of perfectoid spaces over $k$. If $k=\fp$ we write $\perf=\perf_{\fp}$. For any perfectoid space $S$ over $k$, we write $S \bdtimes \zp$ for the analytic adic space defined in \cite[Proposition 11.2.1]{ScholzeWeinsteinBerkeley}. In particular, when $S = \spa(R,R^+)$ is affinoid perfectoid, $S \bdtimes \zp$ is given by 
\begin{align}
    S \bdtimes \zp = \spa(W(R^{+})) \setminus \{[\varpi]=0\},
\end{align}
where $W(R^+)$ denotes the ring of $p$-typical Witt vectors of the perfect ring $R^+$, and where $[\varpi]$ denotes the Teichm\"ulller lift to $W(R^+)$ of a pseudo-uniformiser $\varpi$ in $R^+$. The Frobenius for $W(R^+)$ restricts to give a Frobenius operator $\frob_S$ on $S \bdtimes \zp$. By \cite[Proposition 11.3.1]{ScholzeWeinsteinBerkeley}, any untilt $S^\sharp$ of $S$ determines a closed Cartier divisor $S^\sharp \hookrightarrow S \bdtimes \spa \zp$. For $S$ in $\perf$ we define also $\mathcal{Y}(S) = S \bdtimes \zp \setminus \{p = 0\}$.

For any $S$ in $\perf$, the relative adic Fargues--Fontaine curve over $S$ is the quotient
\begin{align}
    X_S = \mathcal{Y}(S) / \varphi^\mathbb{Z}.
\end{align}
By \cite[Proposition II.1.16]{FarguesScholze}, the action of $\frob_S$ on $\mathcal{Y}(R,R^+)$ is free and totally discontinuous, which means that the quotient is well-defined. Let $G$ be a reductive group over $\qp$. Following \cite{FarguesScholze}, we denote by $\bun_G(S)$ the groupoid of $G$-torsors on $X_S$. By \cite[Proposition III.1.3]{FarguesScholze}, $\bun_G$ is a small v-stack on $\perf$. A morphism $f:G \to G'$ induces a $1$-morphism $\bun_G \to \bun_{G'}$ by pushing out torsors, and the following lemma follows from Lemma \ref{Lem:FunctorialityClassifyingStack}.
\begin{Lem} \label{Lem:WeakFunctorBunG}
The construction of the v-stack $\bun_G$ for $G$ a reductive group scheme
  over $\qp$ extends to a weak functor
  \[
    \lbrace \text{reductive groups over } \qp \rbrace \to \lbrace v\text{-stacks
    on } \perf \rbrace; \quad G \mapsto \bun_G.
  \]
\end{Lem}

\subsubsection{} Let $B(G)$ be the set of $\sigma$-conjugacy classes in $G(\qpbr)$, equipped with the topology coming from the \emph{opposite} of the partial order defined in \cite[Section 2.3]{RapoportRichartz}. By \cite[Theorem 1]{Viehmann}, there is a homeomorphism
\begin{align}
    |\bun_G| \to B(G). 
\end{align}
If $\mu$ is a $G(\qpbar)$-conjugacy class of minuscule cocharacters, we let $\bgmu \subset B(G)$ be the set of $\mu^{-1}$-admissible elements, as defined in \cite[Section 1.1.5]{KMPS}; note that this set is closed in the partial order and thus defines an open substack
$$\bungmu \subset \operatorname{Bun}_G,$$ via \cite[Proposition 12.9]{EtCohDiam}. 

\subsubsection{} \label{Sec:DiamondFunctors}  If $X$ is an adic space over $\spa(\zp)$, let $X^\lozenge$ denote the set-valued functor on $\perf$ given by
\begin{align}
    X^\lozenge(S) = \{(S^\sharp, f)\} / \sim
\end{align}
for any $S$ in $\perf$, where $S^\sharp$ is an untilt of $S$ and $f:S^\sharp \to X$ is a morphism of adic spaces. This determines a v-sheaf on $\perf$ by \cite[Lemma 18.1.1]{ScholzeWeinsteinBerkeley}. For a Huber pair $(A,A^+)$ over $\zp$, we write $\spd(A,A^+)$ in place of $\spa(A,A^+)^\lozenge$, and when $A^+ = A^\circ$ we write $\spd(A)$ instead of $\spd(A,A^+)$. In particular, we see that $\spd(\zp)$ parametrizes isomorphism classes of untilts, cf. \cite[Definition 10.1.3]{ScholzeWeinsteinBerkeley}. 

If $\mathcal{X}$ is a formal scheme over $\spf \zp$, then we can consider it as an adic space over $\spa(\zp)$ using \cite[Proposition 2.2.1]{ScholzeWeinsteinModuli} and use that adic space to define $\mathcal{X}^{\lozenge}$. This defines a functor from formal schemes over $\zp$ to v-sheaves over $\spd \zp$. This functor is fully faithful when restricted to absolutely weakly normal formal schemes that are flat, separated and formally of finite type over $\spf \zp$, see \cite[Theorem 2.16]{AGLR}. 

If $X$ is a scheme over $\spec \zp$, then there are three possible v-sheaves that can be associated to $X$: There is $X^{\lozenge}$, there is $X^{\lozenge/}$ and there is $X^{\diamond}$, see \cite[Definition 2.10]{AGLR} and \cite[Definition 2.1.9]{PappasRapoportShtukas}. By \cite[Corollary 2.1.8]{PappasRapoportShtukas}, the functor $X \mapsto X^{\lozenge/}$ is fully faithful when restricted to schemes that are flat normal and separated locally of finite type over $\zp$. 

For schemes $X$ that are locally of finite type over $\spec \zp$, these admit the following descriptions, see \cite[Remark 2.11]{AGLR}:
\begin{itemize}
    \item The v-sheaf $X^{\diamond}$ can be identified with $\left(\widehat{X}\right)^{\lozenge}$, where $\widehat{X}$ is the $p$-adic formal scheme over $\spf \zp$ given by the completion of $X$ in its special fiber. In \cite{PappasRapoportShtukas}, they write $X^{\blacklozenge}$ for $X^{\diamond}$.
    
    \item The v-sheaf $X^{\lozenge}$ can be identified with $(X^{\mathrm{an}})^{\lozenge}$, where $X^{\mathrm{an}}$ is the `analytification' of $X$, see \cite[Remark 2.11]{AGLR}.

    \item The v-sheaf $X^{\lozenge/}$ is created by gluing $X^{\diamond}$ to $X^{\lozenge}_{\qp}$ along the open immersion $X^{\diamond}_{\qp} \to X^{\lozenge}_{\qp}$.

\end{itemize}

\subsection{Local models}
In this section we let $G$ be a connected reductive group over a finite extension $L$ of $\qp$ and let $\mu$ be a $G(\overline{L})$-conjugacy class of minuscule cocharacters of $G$ with reflex field $E \subset \overline{L}$. We fix a quasi-parahoric\footnote{In the sense of \cite[Section 2.2]{PappasRapoportRZ}} model $\mathcal{G}$ of $G$ over $\mathcal{O}_L$. We let $\grg \to \spd \mathcal{O}_{L}$ be the Beilinson--Drinfeld affine flag variety of $\mathcal{G}$, see \cite[Section 4.1]{AGLR}, considered as a v-sheaf on $\perf$. There is a closed subfunctor
\begin{align}
    \grgmu \subset \grg \times_{\spd \mathcal{O}_{L}} \spd E,
\end{align}
see \cite[Corollary 4.6]{AGLR}, whose closure in $\grg \times_{\spd \mathcal{O}_L} \spd \mathcal{O}_E$ defines a closed subfunctor
\begin{align}
    \locmodgmuv \subset \grg \times_{\spd \mathcal{O}_L} \spd \mathcal{O}_{E},
\end{align}
called the v-sheaf local model attached to $(\mathcal{G},\mu)$. 

\subsubsection{} Assume from now on that $\mathcal{G}$ is a parahoric and let $\calgad$ be the corresponding parahoric model for $\gad$. The formation of $\locmodgmuv$ is functorial for morphisms of pairs $(\mathcal{G}_1, \mu_1) \to (\mathcal{G}_2, \mu_2)$ where $\mathcal{G}_1, \mathcal{G}_2$ are parahoric, and preserves closed embeddings by \cite[Proposition 4.16]{AGLR}. Moreover, the morphism 
\begin{align}
    \locmodgmuv \to \locmodgadmuv
\end{align}
induced by $(\mathcal{G},\mu) \to (\calgad,\mu)$ is an isomorphism. 

\subsubsection{} Let $L' \subset \overline{L}$ be a finite extension of $L$ and let $E' \supset L'$ be the reflex field of $\mu$ considered as a $G(\overline{L})$-conjugacy class of cocharacters of $G_{L'}$. 
\begin{Lem} \label{Lem:UnramifiedBaseChangeLocalModel}
There is a natural isomorphism
\begin{align}
    \locmodgmuv \times_{\spd \mathcal{O}_E} \spd \mathcal{O}_{E'} \to \locmodgolpmuv.
\end{align}
\begin{proof}
The analogous result for $$ \grg \times_{\spd \mathcal{O}_{L'}} \spd \mathcal{O}_{L'} \to \operatorname{Gr}_{\mathcal{G}_{\mathcal{O}_{L'}}}$$ follows directly from the definition, see \cite[Section 4.1]{AGLR}. It also follows directly from the definition that this induces a natural isomorphism
\begin{align}
    \grgmu \times_{\spa E} \spa E'  \to \operatorname{Gr}_{G_{L'},\mu}.
\end{align}
Since the map $\spec \mathcal{O}_{E'} \to \mathcal{O}_{E}$ is finite flat, it is in particular proper and open. Thus the map $\spd \mathcal{O}_{E'} \to \spd \mathcal{O}_{E}$ is partially proper and open. It then follows from \cite[Corollary 2.9]{GleasonLourenco}, that the formation of v-sheaf closures commutes with basechanging from $\spd \mathcal{O}_E$ to $\spd \mathcal{O}_{E'}$. Thus the natural map 
\begin{align}
    \locmodgmuv \times_{\spd \mathcal{O}_E} \spd \mathcal{O}_{E'} \to \locmodgolpmuv.
\end{align}
is an isomorphism.
\end{proof}
\end{Lem}
Now suppose that $L'/L$ is Galois and let $\Gamma$ be its Galois group. Define $\mathcal{H}:=\operatorname{Res}_{\mathcal{O}_{L'}/\mathcal{O}_{L}} \mathcal{G}$ and let $\mu$ be the induced conjugacy class of cocharacters of $H:=\mathcal{H}_{L}$. The natural map $(\mathcal{G},\mu) \to (\mathcal{H}, \mu)$ induces a natural map
\begin{align} \label{Eq:NaturalMapLocalModel}
    \locmodgmuv \to \locmodhmuv,
\end{align}
\begin{Lem} \label{Lem:IdentificationLocalModelsI}
If $L'/L$ is unramified, then there is a $\Gamma$-equivariant isomorphism 
\begin{align}
    \locmodhmuv \times_{\spd \mathcal{O}_E} \spd \mathcal{O}_{E'} \to \prod_{\gamma \in \Gamma} \locmodgmuv \times_{\spd \mathcal{O}_E} \spd \mathcal{O}_{E'},
\end{align}
under which the natural map from equation \ref{Eq:NaturalMapLocalModel} corresponds to the inclusion of the diagonal.
\end{Lem}
\begin{proof}
By Lemma \ref{Lem:UnramifiedBaseChangeLocalModel}, there is a natural (in particular $\Gamma$-equivariant) isomorphism
\begin{align}
    \locmodhmuv \times_{\spd \mathcal{O}_E} \spd \mathcal{O}_{E'} \to \locmodholpmuv.
\end{align}
There is a $\Gamma$-equivariant isomorphism $\mathcal{H}_{\mathcal{O}_{L'}} \to \prod_{\gamma \in \Gamma} \mathcal{G}_{\mathcal{O}_{L}'}$ since $L'/L$ is unramified. Since the formation of local models commutes with direct products, see \cite[Proposition 4.16]{AGLR}, this induces a $\Gamma$-equivariant isomorphism 
\begin{align}
    \locmodholpmuv \to \prod_{\gamma \in \Gamma} \locmodgolpmuv.
\end{align}
Using Lemma \ref{Lem:UnramifiedBaseChangeLocalModel} again, we again identify the right-hand side with
\begin{align}
    \prod_{\gamma \in \Gamma} \locmodgmuv \times_{\spd \mathcal{O}_E} \spd \mathcal{O}_{E'}.
\end{align}
Moreover, under these identifications the natural map 
\begin{align}
    \locmodgmuv \to \locmodhmuv
\end{align}
corresponds to the inclusion of the diagonal into the product.
\end{proof}
\subsubsection{} Now let $\spec \mathcal{O}_{L'} \to \spec \mathcal{O}_L$ be a finite \'etale Galois cover with Galois group $\Gamma$ and generic fiber $\spec L' \to \spec L$. Define $\mathcal{H}:=\operatorname{Res}_{\mathcal{O}_{L'}/\mathcal{O}_{L}} \mathcal{G}$ and let $\mu$ be the induced conjugacy class of cocharacters of $H:=\mathcal{H}_{L}$. Then $\Gamma$-acts on $\mathcal{H}$ in a way that preserves $\mu$ and thus acts on $\locmodhmuv$. The natural map of \eqref{Eq:NaturalMapLocalModel} is $\Gamma$-equivariant for the trivial $\Gamma$-action on the source. 

Let $\mathfrak{p}$ be a maximal ideal of $\mathcal{O}_{L'}$, let $\mathcal{O}_{L''}$ be the local ring of $\mathcal{O}_{L'}$ at $\mathfrak{p}$ with fraction field $L''$. Let $\Gamma'' \subset \Gamma$ be the stabilizer of $\mathfrak{p}$, which is also the Galois group of $L''/L$. Choose an embedding $L'' \to \overline{L}$ and let $E''$ be the reflex field of $\mu$ considered as an $G_{L''}(\overline{L})$-conjugacy class of cocharacters of $G_{L''}$.
\begin{Lem} \label{Lem:IdentificationLocalModelsII}
There is a $\Gamma$-equivariant isomorphism 
\begin{align}
    \locmodhmuv \times_{\spd \mathcal{O}_E} \spd \mathcal{O}_{E''} \to \prod_{\gamma \in \Gamma} \locmodgmuv \times_{\spd \mathcal{O}_E} \spd \mathcal{O}_{E'},
\end{align}
under which the natural map from equation \ref{Eq:NaturalMapLocalModel} corresponds to the inclusion of the diagonal.
\end{Lem}
\begin{proof}
Since $\spec \mathcal{O}_{L'} \to \spec \mathcal{O}_L$ is Galois there is a $\Gamma$-equivariant isomorphism
\begin{align}
    \spec \mathcal{O}_{L'} \to \hom_{\Gamma''}(\Gamma, \spec \mathcal{O}_{L''}).
\end{align}
Similarly there is a $\Gamma$-equivariant isomorphism
\begin{align}
    \mathcal{H} \to \hom_{\Gamma''}(\Gamma, \operatorname{Res}_{\mathcal{O}_{L''}/\zp} \mathcal{G}_{\mathcal{O}_{L''}}) \simeq \prod_{\Gamma'' \backslash \Gamma} \operatorname{Res}_{\mathcal{O}_{L''}/\zp} \mathcal{G}_{\mathcal{O}_{L''}}.
\end{align}
Since the formation of local models commutes with products by \cite[Corollary 2.9]{GleasonLourenco}, the lemma reduces to Lemma \ref{Lem:IdentificationLocalModelsI}.
\end{proof}

\smallskip

\subsubsection{} \label{Sec:LocalModels} Recall that there are unique (up to unique isomorphism) flat and (absolutely) weakly normal schemes $\locmodgmu$ over $\mathcal{O}_{E}$ with associated v-sheaf isomorphic to $\locmodgmuv$, by \cite[Theorem 1.11]{AGLR}. 
\begin{Prop} \label{Prop:FixedPointsLocalModels}
The natural map of equation \eqref{Eq:NaturalMapLocalModel} induces an isomorphism
\begin{align}
    \locmodgmu \to \locmodhmuGamma.
\end{align}
\end{Prop}
\begin{proof}
By \cite[Proposition 2.18]{AGLR}, the functor sending a proper flat absolutely weakly normal scheme over $\spec \mathcal{O}_L$ to its associated v-sheaf is fully faithful. This functor moreover commutes with base change along $\spec \mathcal{O}_{E''} \to \spec \mathcal{O}_E$ and with fiber products. Thus by Lemma \ref{Lem:IdentificationLocalModelsII} there is a $\Gamma$-equivariant isomorphism 
\begin{align}
    \locmodhmu \times_{\spec \mathcal{O}_E} \spec \mathcal{O}_{E''} \to \prod_{\gamma \in \Gamma} \locmodgmu \times_{\spec \mathcal{O}_E} \spec \mathcal{O}_{E''}
\end{align}
under which the natural map $\locmodgmu \to \locmodhmu$ corresponds to the inclusion of the diagonal. The proposition follows immediately from this description. 
\end{proof}
\begin{Cor} \label{Cor:HomotopyFixedPointsLocalModelQuotients}
The natural map 
\begin{align}
    \locmodgmuStack \to \locmodhmuStackGamma
\end{align}
is an isomorphism. 
\end{Cor}
\begin{proof}
After basechanging to $\spa(\mathcal{O}_L)$ we have an isomorphism $\mathcal{H} \simeq \hom(\Gamma, \mathcal{G})$ and thus it follows from Lemma \ref{Lem:Shapiro} that $\underline{H}^1(\Gamma, \mathcal{H})$ vanishes. The corollary now follows from Proposition \ref{Prop:FixedPointsLocalModels} and Corollary \ref{Cor:NoCohomologyIsomorphism}.
\end{proof}

\subsection{Shtukas} \label{Sec:Shtukas} In this section we let $G$ be a connected reductive group over $\qp$ and $\mu$ a $G(\qpbar)$-conjugacy class of minuscule cocharacters of $G$ with reflex field $ E \subset \qpbar$. We fix a quasi-parahoric model $\mathcal{G}$ of $G$ over $\zp$. 

Recall from \cite[Definition 2.4.3]{PappasRapoportShtukas} that a $\mathcal{G}$-shtuka over a perfectoid space $S$ with leg at an untilt $S^{\sharp}$ is defined to be a quadruple: $(\mathcal{Q}, \mathcal{P}, \phi_{\mathcal{P}}, \kappa)$ where $\mathcal{Q}$ and $\mathcal{P}$ are $\mathcal{G}$-torsors over the analytic adic space $S \bdtimes \zp$, where $\kappa:\mathcal{Q} \to \left(\operatorname{Frob}_S \times 1\right)^{\ast}\mathcal{P}$ is an isomorphism of $\mathcal{G}$-torsors and where 
\begin{align}
    \phi_{\mathcal{P}}:\restr{\mathcal{Q}}{S \bdtimes \zp \setminus S^{\sharp}} \to \restr{\mathcal{P}}{S \bdtimes \zp \setminus S^{\sharp}}
\end{align}
is an isomorphism of $\mathcal{G}$-torsors over $S \bdtimes \zp \setminus S^{\sharp}$. Here $S^{\sharp} \subset S \bdtimes \zp$ is the Cartier divisor coming from the untilt $S^{\sharp}$. To be precise, here we consider $\left(\operatorname{Frob}_S \times 1\right)^{\ast}\mathcal{P}$ as a $\mathcal{G}$-torsor via the isomorphism $\left(\operatorname{Frob}_S \times 1\right)^{\ast} \mathcal{G} \to \mathcal{G}$ coming from the fact that $\mathcal{G}$ is defined over $\zp$. Note that the data of $\mathcal{Q}$ and $\kappa$ is superfluous in our definition, but it will be useful for us later to keep track of this information. In Section \ref{Sec:IntegralModels} we will often omit $\mathcal{Q}$ and $\kappa$ from the notation

\subsubsection{} Let $\perf_{\zp}$ be the category of perfectoid spaces $S$ of characteristic $p$ equipped with a map $S^{\lozenge} \to \spd \zp$, equipped with its v-topology. Let $\operatorname{Sht}_{\mathcal{G}}$ be the stack of $\mathcal{G}$-shtukas over $\perf_{\zp}$, considered as a category fibered in groupoids over $\perf_{\zp}$. Explicitly, this means that it is the category whose objects are quintuples $(S \to \spd \zp, \mathcal{Q}, \mathcal{P}, \phi_{\mathcal{P}}, \kappa)$, where $S \to \spd \zp$ is an object of $\perf_{\zp}$ and $(\mathcal{Q}, \mathcal{P}, \phi_{\mathcal{P}}, \kappa)$ is a $\mathcal{G}$-shtuka over $S$. A morphism $f:(S \to \spd \zp, \mathcal{Q}, \mathcal{P}, \phi_{\mathcal{P}}, \kappa) \to (S' \to \spd \zp, \mathcal{Q}', \mathcal{P}', \phi_{\mathcal{P}'}, \kappa')$ is a triple $(f,f_{\mathcal{P}},f_{\mathcal{Q}})$ consisting of a morphism $f:S \to S'$ and $\mathcal{G}$-equivariant morphisms $f_{\mathcal{P}}:\mathcal{P} \to \mathcal{P}', f_{\mathcal{Q}}:\mathcal{Q} \to \mathcal{Q}'$ fitting in a pair of Cartesian diagrams
\begin{equation}
    \begin{tikzcd}
    \mathcal{P} \arrow{r}{f_{\mathcal{P}}} \arrow{d} & \mathcal{P}' \arrow{d} \\
    S \bdtimes \zp \arrow{r}{f \times 1} & S' \dot{\times} \zp
    \end{tikzcd} \qquad
    \begin{tikzcd}
    \mathcal{Q} \arrow{r}{f_{\mathcal{Q}}} \arrow{d} & \mathcal{Q}' \arrow{d} \\
    S \bdtimes \zp \arrow{r}{f \times 1} & S' \dot{\times} \zp
    \end{tikzcd}
\end{equation}
such that the following diagrams commute
\begin{equation}
    \begin{tikzcd}
        \restr{\mathcal{Q}}{S \bdtimes \zp \setminus S^{\sharp}} \arrow{r}{\phi_{\mathcal{P}}} \arrow{d}{f_{\mathcal{P}}} & \restr{\mathcal{P}}{S \bdtimes \zp \setminus S^{\sharp}} \arrow{d}{f_{\mathcal{Q}}} \\
       \restr{\mathcal{P}}{S' \times \zp \setminus S^{\sharp}} \arrow{r}{\phi_{\mathcal{P}'}}  & \restr{\mathcal{Q'}}{S' \times \zp \setminus S^{\sharp}}
    \end{tikzcd}
    \qquad 
    \begin{tikzcd}
    \mathcal{Q} \arrow{d}{f_{\mathcal{Q}}} \arrow{r}{\kappa} & \left(\operatorname{Frob}_S \times 1\right)^{\ast}\mathcal{P}  \arrow{d}\\
    \mathcal{Q'} \arrow{r}{\kappa'} & \left(\operatorname{Frob}_{S'} \times 1\right)^{\ast}\mathcal{P}',
    \end{tikzcd}
\end{equation}
where the right vertical arrow in the second diagram is the arrow induced by $f_{\mathcal{P}}$. Since $\mu$ is defined over $E$, there is a closed substack\footnote{Here by $\operatorname{Sht}_{\mathcal{G}} \otimes_{\spd \zp} \spd \mathcal{O}_{E}$ we mean the category fibered in groupoids over $\perf_{\mathcal{O}_E}$ given by the fiber product construction of \cite[Lemma 0040]{stacks-project}.}
\begin{align}
    \shtg \subset \operatorname{Sht}_{\mathcal{G}} \otimes_{\spd \zp} \spd \mathcal{O}_{E},
\end{align}
defined as those shtukas where $\phi_{\mathcal{P}}$ has relative position bounded by the v-sheaf local model
\begin{align}
    \locmodgmuv \subset \operatorname{Gr}_{\mathcal{G}} \otimes_{\spd \zp} \spd \mathcal{O}_{E},
\end{align}
see \cite[Definition 2.4.4]{PappasRapoportShtukas}. It is clear that a homomorphism $f:\mathcal{G} \to \mathcal{G}'$ of quasi-parahoric group schemes induces a morphism.
\begin{align}
    f:\operatorname{Sht}_{\mathcal{G}} \to \operatorname{Sht}_{\mathcal{G}'}
\end{align}
by pushing out torsors. Noting that the reflex field $E'$ of $\mu'$ is contained in $E$, this restricts to a morphism 
\begin{align}
    \shtg \times_{\spd \mathcal{O}_{E}} \spd \mathcal{O}_{E'} \to \operatorname{Sht}_{\mathcal{G}',\mu'},
\end{align} 
if $f(\mu)=\mu'$. This upgrades to a $2$-categorical statement as follows: Fix a field $L \subset \qpbar$ and let $\shtpr_{L}$ be the category of pairs $(\mathcal{G},\mu)$ such that the reflex field $E(\mu)$ of $\mu$ is contained in $L$, and where the morphisms are the obvious morphisms of pairs. We will often write $E$ for $E(\mu)$ if $\mu$ is clear from the context. We let $\perfpairs_L$ be the strict $(2,1)$-category of categories fibered in groupoids over $\perf_{\mathcal{O}_L}$ and $\perfpairsrat_L$ the full subcategory of categories fibered in groupoids over $\perf_{L}$.
\begin{Lem} \label{Lem:FunctorialityShtukas}
There is a weak functor $\operatorname{Sht}:\shtpr_{L} \to \perfpairs_L$, which on objects sends $(\mathcal{G},\mu)$ to $\shtgmu \times_{\spd \mathcal{O}_E} \spd \mathcal{O}_L$ and which sends a morphism $(\mathcal{G}, \mu) \to (\mathcal{G}', \mu')$ to the induced morphism $\shtgmu \times_{\spd \mathcal{O}_E} \spd \mathcal{O}_L \to \shtgmu \times_{\spd \mathcal{O}_{E'}} \spd \mathcal{O}_L$ induced by pushing out torsors along $\mathcal{G} \to \mathcal{G}'$.
\end{Lem}
\begin{proof}
A morphism $f:\mathcal{G} \to \mathcal{G}'$ induces a morphism $\shtg \to \shtgp$ by sending $(S \to \spd \zp, \mathcal{Q}, \mathcal{P}, \phi_{\mathcal{P}}, \kappa)$ to $(S \to \spd \zp, \mathcal{Q} \times^{\mathcal{G}} \mathcal{G}', \mathcal{P} \times^{\mathcal{G}} \mathcal{G}', \phi_{\mathcal{P} \times^{\mathcal{G}} \mathcal{G}'}, \kappa'$). Here $\kappa'$ is the unique isomorphism induced by $\kappa$, using the fact that $f:\mathcal{G} \to \mathcal{G}'$ is defined over $\zp$ and thus commutes with $\frob_S$. To turn this into a weak functor, we use the coherence data for pushing out torsors coming from the proof of Lemma \ref{Lem:FunctorialityClassifyingStack}. 

As discussed above, this morphism restricts to a morphism between $\shtgmu$ and $\shtgpmup$ and thus induces a morphism $\shtgmu \times_{\spd \mathcal{O}_E} \spd \mathcal{O}_L \to \shtgpmup \times_{\spd \mathcal{O}_{E'}} \spd \mathcal{O}_L$.
\end{proof}
We recall the open and closed substack $\shtgmuone \subset \shtgmu$ introduced in \cite[Section 3.3.8]{DanielsVHKimZhangII}. If $\mathcal{G}$ has connected special fiber (so $\mathcal{G}$ is parahoric instead of just quasi-parahoric), then $\shtgmuone=\shtgmu$.

\subsection{Integral local Shimura varieties} In this section we let $G$ be a connected reductive group over $\qp$ and $\mu$ a $G(\qpbar)$-conjugacy class of minuscule cocharacters of $G$ with reflex field $E \subset \qpbar$. We fix a quasi-parahoric model $\mathcal{G}$ of $G$ over $\zp$. It follows from Lemma \ref{Lem:FunctorialityClassifyingStack} that there is a weak functor $\operatorname{Bun}$ from the category of reductive groups over $\qp$ to the strict $(2,1)$-category of v-stacks on $\perf$. Recall that there is a faithful morphism $\operatorname{BL}^{\circ}:\shtgmu \to \operatorname{Bun}_G$, as explained in \cite[Section 2.2.2]{PappasRapoportShtukas}, see \cite[Section 2.3.3]{DvHKZIgusaStacks}.
\begin{Lem} \label{Lem:FunctorialityShtukasToBunG}
There is a weak natural transformation $\operatorname{Sht} \to \operatorname{Bun}$ of weak functors $\shtpr_L \to \perfstacks$, giving $\operatorname{BL}^{\circ}:\shtgmu \to \operatorname{Bun}_G$ on objects.
\end{Lem}
\begin{proof}
This comes down to checking that pushing out torsors is compatible with descending them from $\mathcal{Y}_{[r,\infty)}(S)$ (notation as in \cite[Section 2.2.2]{PappasRapoportShtukas}) to $X_S$, which is straightforward. 
\end{proof}
For $b: \spd \ovfp \to \shtgmu$ we define $\mintgb$ as the $2$-fiber product
\begin{equation}
    \begin{tikzcd}
    \mintgb \arrow{r} \arrow{d} & \shtgmu \arrow{d} \\
    \spd \ovfp \arrow{r}{b} & \operatorname{Bun}_G.
    \end{tikzcd}
\end{equation}
The universal property of the fiber product induces a natural map $x_0:\spd \ovfp \to \mintgb$. Note that a morphism of triples $(\mathcal{G},\mu,b) \to (\mathcal{G}',\mu',b')$ induces a morphism $\mintgb \to \mathcal{M}^{\mathrm{int}}_{\mathcal{G}',b',\mu'}$ taking $x_0$ to $x_0'$. It is straightforward to check that this gives a functor from the category of triples $(\mathcal{G},\mu, b)$ with reflex field contained in $L$, to $\perfpairs_{\mathcal{O}_L}$, sending
\begin{align}
    (\mathcal{G},\mu,b) \mapsto \mintgb \otimes_{\spd \mathcal{O}_{E}} \spd \mathcal{O}_L.
\end{align}
\subsubsection{Fixed points of integral local Shimura varieties}
Let $\spec \mathcal{O}_L \to \spec \zp$ be a finite \'etale Galois cover with Galois group $\Gamma$ and generic fiber $\spec L \to \spec \qp$. Let $\mathcal{G}$ be a quasi-parahoric model of $G$ as before and define $\mathcal{H}:=\operatorname{Res}_{\mathcal{O}_{L}/\zp} \mathcal{G}_{\mathcal{O}_L}$ and let $\mu$ be the induced conjugacy class of cocharacters of $H:=\mathcal{H}_{\qp}$. Then $\Gamma$-acts on $\mathcal{H}$ in a way that preserves $\mu$ and thus acts on $\shthmu \to \perf_{\mathcal{O}_E}$ by Lemma \ref{Lem:FunctorialityShtukas}. 
\begin{Lem} \label{Lem:TorsorsTrivialAdic}
Let $X$ be an adic space over $\spa(\zp,\zp)$, then cocycles $\sigma:\Gamma \to \mathcal{H}(X)$ are trivial \'etale locally on $X$. 
\end{Lem}
\begin{proof}
After basechanging to $\spa(\mathcal{O}_L,\mathcal{O}_L)$, we have an isomorphism $\mathcal{H} \simeq \hom(\Gamma, \mathcal{G})$ and the result then follows from Lemma \ref{Lem:Shapiro}.
\end{proof}
\begin{Cor} \label{Cor:TorsorsInvariantsAdicSpaces}
Consider the stack $\mathbb{B} \mathcal{H}$ on the category of adic spaces over $\spa(\zp,\zp)$ equipped with the \'etale topology. Then the natural map
\begin{align}
    \mathbb{B} \mathcal{G} \to \left(\mathbb{B} \mathcal{H}\right)^{h \Gamma}
\end{align}
is an equivalence.
\end{Cor}
\begin{proof}
This is a direct consequence of Lemma \ref{Lem:TorsorsTrivialAdic} and Proposition \ref{Prop:HomotopyFixedPointsClassifyingStack}.
\end{proof}
\subsubsection{Homotopy fixed points of shtukas} \label{Sec:HomotopyFixedPointsShtukas} There is a natural map $\shtg \to \operatorname{Sht}_{\mathcal{H}}^{h \Gamma}$ of categories fibered in groupoids over $\perf_{\zp}$; this follows from Lemma \ref{Lem:FunctorialityShtukas} and Lemma \ref{Lem:FunctorialityHomotopyFixedPoints}.
\begin{Lem} \label{Lem:HomotopyFixedPointsShtukas}
The natural map
\begin{align}
    \operatorname{Sht}_{\mathcal{G}} \to \operatorname{Sht}_{\mathcal{H}}^{h \Gamma}
\end{align}
is an isomorphism.
\end{Lem}
\begin{proof}
To prove that this is an equivalence, it suffices to do this fiberwise over $\spd \zp$; so fix $S \to \spd \zp$ corresponding to an untilt $S^{\sharp}$ of $S$. The category of $\mathcal{H}$-shtukas over $S$ can be described as the $2$-fiber product
\begin{equation}
    \begin{tikzcd}
    \operatorname{Sht}_{\mathcal{H}}(S) \arrow{r} \arrow{d} & \mathbb{B} \mathcal{H}(S \bdtimes \zp) \arrow{d}{\Gamma_{\operatorname{Frob}_S}} \\
    \mathbb{B} \mathcal{H}(S \bdtimes \zp) \arrow{r}{\Delta} & \mathbb{B} \mathcal{H}(S \bdtimes \zp \setminus S^{\sharp}) \times \mathbb{B} \mathcal{H}(S \bdtimes \zp),
    \end{tikzcd}
\end{equation}
where $\Delta$ is the diagonal and where $\Gamma_{\operatorname{Frob}_S}$ is the graph of pullback along $\operatorname{Frob}_S$. This diagram is $\Gamma$-equivariant because the right vertical and bottom horizontal maps are equivariant by Lemma \ref{Lem:FunctorialityClassifyingStack}. Lemma \ref{Lem:FiberProducts} tells us that the homotopy fixed points of the diagram give a $2$-Cartesian diagram
\begin{equation}
    \begin{tikzcd}
    \operatorname{Sht}_{\mathcal{H}}(S)^{h \Gamma} \arrow{r} \arrow{d} & \mathbb{B} \mathcal{H}(S \bdtimes \zp)^{h \Gamma} \arrow{d}{\Gamma_{\operatorname{Frob}_S}} \\
    \mathbb{B} \mathcal{H}(S \bdtimes \zp)^{h \Gamma} \arrow{r}{\Delta} & \mathbb{B} \mathcal{H}(S \bdtimes \zp \setminus S^{\sharp})^{h \Gamma} \times \mathbb{B} \mathcal{H}(S \bdtimes \zp )^{h \Gamma},
    \end{tikzcd}
\end{equation}
which after repeatedly applying Corollary \ref{Cor:TorsorsInvariantsAdicSpaces} can be identified with the $2$-Cartesian diagram
\begin{equation}
    \begin{tikzcd}
    \operatorname{Sht}_{\mathcal{H}}(S)^{h \Gamma} \arrow{r} \arrow{d} & \mathbb{B} \mathcal{G}(S \bdtimes \zp) \arrow{d}{\Gamma_{\operatorname{Frob}_S}} \\
    \mathbb{B} \mathcal{G}(S \bdtimes \zp) \arrow{r}{\Delta} & \mathbb{B} \mathcal{G}(S \bdtimes \zp \setminus S^{\sharp}) \times \mathbb{B} \mathcal{G}(S \bdtimes \zp).
    \end{tikzcd}
\end{equation}
By construction, the map $\shtg$ to $\operatorname{Sht}_{\mathcal{H}}(S)^{h \Gamma}$ compatible is compatible with the $2$-fiber product description, and thus an isomorphism.
\end{proof}
\subsubsection{} Recall that we also use $\mu$ to denote the induced $H(\qpbar)$-conjugacy class of cocharacters of $H$, and we will later do this for all other groups. 
\begin{Lem} \label{Lem:HomotopyFixedPointsShtukasII}
The natural map
\begin{align}
    \shtgmu \to \shthmu^{h \Gamma}
\end{align}
is an isomorphism.
\end{Lem}
\begin{proof}
This follows from Lemma \ref{Lem:HomotopyFixedPointsShtukas} as soon as we can prove that a $\mathcal{G}$-shtuka whose induced $\mathcal{H}$-shtuka is bounded by $\mu$, is itself bounded by $\mu$. This follows from the fact that the natural isomorphism
\begin{align}
    \grg \to \operatorname{Gr}_{\mathcal{H}}^{\Gamma}
\end{align}
induces an isomorphism $\locmodgmuv \to \locmodhmuvGamma$, see Proposition \ref{Prop:FixedPointsLocalModels}.
\end{proof}
\begin{Lem} \label{Lem:HomotopyFixedPointsBunG}
Let $\spec L \to \spec \qp$ be a finite \'etale Galois cover (not necessarily unramified) with Galois group $\Gamma$ and let $H= \operatorname{Res}_{L/\mathbb{Q}} G_{L}$. Then the natural map
\begin{align}
    \operatorname{Bun}_G \to \operatorname{Bun}_\rH^{h \Gamma}
\end{align}
is an isomorphism.
\end{Lem}
\begin{proof}
It follows from the proof of Corollary \ref{Cor:TorsorsInvariantsAdicSpaces} that 
\begin{align}
    \mathbb{B} G \to \left(\mathbb{B} H \right)^{h \Gamma}
\end{align}
is an equivalence on the category of adic spaces over $\spa(\qp,\zp)$. Indeed, the point is that we can use Lemma \ref{Lem:Shapiro} after the finite \'etale cover $\spa(L,\mathcal{O}_L) \to \spa(\qp,\zp)$ and then conclude using Corollary \ref{Cor:NoCohomologyIsomorphism}.
\end{proof}
Now we return to the assumption that $\spec \mathcal{O}_L \to \spec \zp$ is a finite \'etale cover with Galois group $\Gamma$. If $b:\spd \ovfp \to \shthmu$ is induced by $b:\spd \ovfp \to \shtgmu$, then there is an action of $\Gamma$ on $\minthb$.
\begin{Prop} \label{Prop:InvariantsRZ}
The natural map $\mintgb \to \left(\minthb\right)^{\Gamma}$ is an isomorphism.
\end{Prop}
\begin{proof}
The Cartesian diagram defining $\minthb$ is $\Gamma$-equivariant by Lemma \ref{Lem:FunctorialityShtukasToBunG} and the fact that $b$ is induced by $b:\spd \ovfp \to \shtgmu$. The proposition then follows by taking $\Gamma$-homotopy fixed points of this diagram, and using Lemma \ref{Lem:FiberProducts} in combination with Lemmas \ref{Lem:HomotopyFixedPointsShtukasII} and \ref{Lem:HomotopyFixedPointsBunG}.
\end{proof}

\section{Fixed points of Shimura varieties} \label{Sec:Shimura}
Let $\gx$ be a Shimura datum and let $\f$ be a totally real Galois extension of $\mathbb{Q}$ with Galois group $\Gamma$. Let $H:=\operatorname{Res}_{\f/\mathbb{Q}}G_F $ equipped with its natural action of $\Gamma$. For $h:\mathbb{S} \to \g_{\mathbb{R}}$ in $\mathsf{X}$ we let $\mathsf{Y}$ be the $\h(\mathbb{R})$-conjugacy class of the composition $h:\mathbb{S} \to G_{\mathbb{R}} \to \h_{\mathbb{R}}$, this does not depend on the choice of $h$. Then $\hy$ is a Shimura datum and the action of $\Gamma$ on $\h$ is by automorphisms of Shimura data. Following a suggestion of Rapoport, we call $\hy$ the \emph{Piatetski-Shapiro construction} for $\gx$ associated to $\f$.\footnote{This construction was used by Borovoi \cite{Borovoi} and Milne \cite{MilneLanglands} in their proof of the conjecture of Langlands on conjugation of Shimura varieties, and they credit it to Piatetski-Shapiro.} The natural closed immersion $\g \to \h$ defines a morphism of Shimura data. The goal of this section is to investigate the fixed points of the action of $\Gamma$ on the Shimura varieties for $\hy$, and to compare them to the Shimura varieties for $\gx$.

\subsection{The case of tori} \label{Sec:Tori}
As a toy example, we study the case when $\g=\tor$ is a torus with $\mathbb{R}$-split rank zero. Then the $\mathbb{R}$-split rank of
$\h=\operatorname{Res}_{\f/\mathbb{Q}} \tor_F$ is also zero, and thus for neat $K \subset \haf$ the natural map
\begin{align}
    \h(\mathbb{Q}) \to \haf/K
\end{align}
is injective by Lemma \ref{Lem:FreeActionShimura}. The cokernel of this injection defines the set of $\mathbb{C}$-points of the Shimura variety $\mathbf{Sh}_{K}\hy$. If $K$ is $\Gamma$-stable, then $\Gamma$ acts on $\mathbf{Sh}_{K}\hy$ and there is a long exact sequence of abelian groups
\begin{align}
    0 \to K^{\Gamma} \to \gafp \to \left(\haf/K)\right)^{\Gamma} \to \rH^1(\Gamma, K) \to \rH^1(\Gamma, \haf) \to \cdots
\end{align}
We will see in Proposition \ref{Prop:ExistenceCofinalGoodCompactOpens} that it is possible to find arbitrary small $\Gamma$-stable $K$ with $\rH^1(\Gamma,K)=0$, at least if $\f$ is tamely ramified over $\mathbb{Q}$. So from now on we will assume that $\rH^1(\Gamma, K)=0$, which implies that the natural map
\begin{align}
   \gaf/K^{\Gamma} \to \left(\haf/K)\right)^{\Gamma}.
\end{align}
is an isomorphism and that the natural map
\begin{align}
    \rH^1(\Gamma, \haf/K) \to \rH^1(\Gamma, \haf)
\end{align}
is injective. It moreover implies that there is a long exact sequence
\begin{align}
    0 \to \g(\mathbb{Q}) \to \gaf/K^{\Gamma} \to \mathbf{Sh}_K\hy^{\Gamma}(\mathbb{C}) \to \rH^1(\Gamma, \h(\mathbb{Q})) \to \rH^1(\Gamma, \haf/K).
\end{align}
Thus the obstruction to the natural map
\begin{align} \label{Eq:ShimuraToriFixedPoints}
    \mathbf{Sh}_{K^{\Gamma}}\gx(\mathbb{C}) \to \mathbf{Sh}_{K}\hy^{\Gamma}(\mathbb{C})
\end{align}
being a bijection is given by the kernel of
\begin{align} \label{Eq:ShaGFTorus}
    \rH^1(\Gamma, \h(\mathbb{Q})) \to \rH^1(\Gamma, \haf).
\end{align}
We will later prove, see Lemma \ref{Lem:KernelAndSha}, that this can be identified with the kernel of
\begin{align}
    \Sha^1(\mathbb{Q}, \mathsf{G}) \to \Sha^1(\f, \g),
\end{align}
in particular, the obstruction is finite. In the rest of the section, we will prove that the same result, suitably interpreted, holds for arbitrary Shimura varieties.

\subsection{Shimura varieties and Shimura stacks} Let the notation be as in the start of Section \ref{Sec:Shimura} and let $K \subset \hafp$ be a compact open subgroup. It turns out that it is easier to analyze the $\Gamma$-homotopy fixed points of the stack quotient/action groupoid
\begin{align} \label{Eq:ShimuraStackExample}
    \left[\h(\mathbb{Q}) \backslash \left(\mathsf{Y} \times \hafp/K\right)\right],
\end{align}
than it is to analyze the $\Gamma$-fixed points of the Shimura variety $\mathbf{Sh}_{K}\hy$ with $\mathbb{C}$-points
\begin{align}
    \mathbf{Sh}_{K}\hy(\mathbb{C})=\h(\mathbb{Q}) \backslash \left(\mathsf{Y} \times \hafp/K\right).
\end{align}
Thus it is important for us to understand when the Shimura variety for $\hy$ is equal to the stack quotient in \eqref{Eq:ShimuraStackExample} (for sufficiently small $K$). \smallskip

Recall Milne's axiom SV5 for a Shimura datum $\gx$ from \cite[p. 64]{Milne}; it asks that $Z(\mathbb{Q}) \subset Z(\af)$ is discrete, where $Z=Z_{\g}$ is the center of $\g$. By \cite[Lemma 1.5.5]{KisinShinZhu} this happens if and only if the $\mathbb{Q}$-split rank of $Z$ is equal to the $\mathbb{R}$-split rank of $Z$. The following lemma is well known, see \cite[Proposition 3.1, Lemma 5.13]{Milne}.
\begin{Lem} \label{Lem:FreeActionShimura}
If axiom SV5 holds for $\gx$, then for neat $K$ the group $\g(\mathbb{Q})$ acts freely on $\mathsf{X} \times \gafp/K$.
\end{Lem}

Unfortunately if SV5 holds for $\gx$ then it often does not hold for $\hy$. For example if $Z_{\g}=\mathbb{G}_m$, then $Z_{\h}=\operatorname{Res}_{\f/\mathbb{Q}} \mathbb{G}_m$, which has $\mathbb{Q}$-split rank one but $\mathbb{R}$-split rank equal to $[\f:\mathbb{Q}]$. We make this observation precise in the following lemma.
\begin{Lem} \label{Lem:SV5}
Suppose that $\gx$ satisfies SV5 and that $[\f:\mathbb{Q}]>1$. Then $\hy$ satisfies SV5 if and only if the $\mathbb{R}$-split rank of $Z_{\mathsf{G}}$ is zero. 
\end{Lem}
\begin{proof}
By \cite[Lemma 1.5.5]{KisinShinZhu}, there is an isogeny $Z_{\g} \sim \mathsf{T}_1\times \mathsf{T}_2$ where $\mathsf{T}_1$ is $\mathbb{Q}$-split and $\mathsf{T}_2$ has $\mathbb{R}$-rank zero. This induces an isogeny
\begin{align}
    Z_{\h} \simeq \operatorname{Res}_{\f/\mathbb{Q}} \mathsf{T}_{1, \f} \times \operatorname{Res}_{\f/\mathbb{Q}} \mathsf{T}_{2, \f}.
\end{align}
The torus $\operatorname{Res}_{\f/\mathbb{Q}} \mathsf{T}_{2, \f}$ still has $\mathbb{R}$-rank zero. The torus $\operatorname{Res}_{\f/\mathbb{Q}} \mathsf{T}_{1, \f}$ has $\mathbb{Q}$-rank equal to the $\mathbb{Q}$-split rank $d$ of $Z_{\mathsf{G}}$, and has $\mathbb{R}$-split rank equal to $d \cdot [\f:\mathbb{Q}]$. Thus the $\mathbb{Q}$-split rank of $Z_{\h}$ is equal to the $\mathbb{R}$-split rank of $Z_{\h}$ if and only if $d=0$.
\end{proof}

\subsection{Fixed points of adelic quotients}
Let the notation be as in the start of Section \ref{Sec:Shimura}. Before we investigate the homotopy fixed points of the groupoid \eqref{Eq:ShimuraStackExample}, we first investigate the fixed points of $\h(\mathbb{Q}) \backslash \mathsf{Y} \times \haf$. To understand the $\Gamma$-fixed points of the quotient $\h(\mathbb{Q})\backslash \h(\mathbb{A})$, we need to understand the map $\rH^1(\Gamma, \h(\mathbb{Q})) \to \rH^1(\Gamma, \h(\mathbb{A}))$. For this we will use the inflation maps
\begin{align}
    \rH^1(\Gamma, \h(\mathbb{Q}))&=\rH^1(\Gamma, \mathsf{G}(\f)) \to \rH^1(\operatorname{Gal}_{\mathbb{Q}}, \g(\qbar))=:\rH^1(\mathbb{Q}, \g) \\
    \rH^1(\Gamma, \h(\mathbb{A}))&=\rH^1(\Gamma, \g(\mathbb{A}_{\f})) \to \rH^1(\operatorname{Gal}_{\mathbb{Q}}, \g(\overline{\mathbb{A}})),
\end{align}
where $\overline{\mathbb{A}}=\mathbb{A} \otimes_{\mathbb{Q}} \qbar$ and $\mathbb{A}_{\f}=\mathbb{A} \otimes_{\mathbb{Q}} \f$. As explained on \cite[p. 298]{PR}, there is an injective map
\begin{align}
    \rH^1(\operatorname{Gal}_{\mathbb{Q}}, \g(\overline{\mathbb{A}})) \to \prod_v \rH^1(\operatorname{Gal}_{\mathbb{Q}_v}, \g(\overline{\mathbb{Q}}_v)),
\end{align}
such that the induced map $\rH^1(\operatorname{Gal}_{\mathbb{Q}}, \g(\qbar)) \to \prod_v \rH^1(\operatorname{Gal}_{\mathbb{Q}_v}, \g(\overline{\mathbb{Q}}_v))$ is the natural map.
\begin{Lem} \label{Lem:KernelAndSha}
The inflation map induces a bijection
\begin{align}
    \ker \left(\rH^1(\Gamma, \h(\mathbb{Q})) \to \rH^1(\Gamma, \h(\mathbb{A})) \right) \to \ker \left(\Sha^1(\mathbb{Q}, \mathsf{G}) \to \Sha^1(\f, \g) \right).
\end{align}
\end{Lem}
\begin{proof}
The inflation map $\rH^1(\Gamma, \h(\mathbb{Q})) \to \rH^1(\mathbb{Q}, \mathsf{G})$ is injective by \cite[Section 5.8.(a)]{SerreGaloisCohomology} and lands $\Sha^1(\mathbb{Q}, \mathsf{G})$ by the discussion above, and so we are done.
\end{proof}
Since $\f$ is totally real, it follows that $\h(\mathbb{R})=\operatorname{Hom}(\Gamma,\g(\mathbb{R}))$ and therefore by Lemma \ref{Lem:Shapiro} that $\rH^1(\Gamma, \h(\mathbb{R}))=\{1\}$. Thus we observe that
\begin{align}
    \ker \left(\rH^1(\Gamma, \h(\mathbb{Q})) \to \rH^1(\Gamma, \haf) \right)=\ker \left(\rH^1(\Gamma, \h(\mathbb{Q})) \to \rH^1(\Gamma, \h(\ad)) \right).
\end{align}
\begin{Lem} \label{Lem:QuotientExactAdelicSpace}
If $\Sha^1(\mathbb{Q}, \mathsf{G}) \to \Sha^1(\f, \g)$ is injective, then the natural map
\begin{align}
    \g(\mathbb{Q}) \backslash \mathsf{X} \times \gaf \to \left(\h(\mathbb{Q}) \backslash \mathsf{Y} \times \haf\right)^{\Gamma} 
\end{align}
is a bijection. 
\end{Lem}
\begin{proof}
By \cite[Corollary 1 on page 50]{SerreGaloisCohomology}, there is a short exact sequence of pointed sets
\begin{align}
    1 \to \h(\mathbb{Q})^{\Gamma} \to \h(\mathbb{A})^{\Gamma} \to \left(\h(\mathbb{Q}) \backslash \h(\mathbb{A})\right)^{\Gamma} \to \rH^1(\Gamma, \h(\mathbb{Q})) \to \rH^1(\Gamma,\h(\mathbb{A})) \to \cdots 
\end{align}
It moreover follows from loc. cit. that if $\rH^1(\Gamma, \h(\mathbb{Q})) \to \rH^1(\Gamma,\h(\mathbb{A}))$ has trivial kernel then $\h(\mathbb{A})^{\Gamma} \to \left(\h(\mathbb{Q}) \backslash \h(\mathbb{A})\right)^{\Gamma}$ is surjective. Thus by Lemma \ref{Lem:KernelAndSha} and our assumption, the natural map
\begin{align}
    \g(\mathbb{Q}) \backslash \g(\mathbb{A}) \to \left(\h(\mathbb{Q}) \backslash \h(\mathbb{A})\right)^{\Gamma} 
\end{align}
is a bijection. Let $\mathsf{K}_{\mathsf{X}} \subset \g(\mathbb{R})$ be the stabilizer of some point $x \in \mathsf{X}$ and let $\mathsf{K}_{\mathsf{Y}}$ be its stabilizer inside of $\h(\mathbb{R})$. Since $\f$ is totally real it follows that $\h(\mathbb{R})=\operatorname{Hom}(\Gamma,\g(\mathbb{R}))$ and also that $\mathsf{K}_{\mathsf{Y}}=\operatorname{Hom}(\Gamma, \mathsf{K}_{\mathsf{X}})$; thus $\rH^1(\Gamma, \mathsf{K}_{\mathsf{Y}})=\{1\}$ and $\mathsf{K}_{\mathsf{X}}=\mathsf{K}_{\mathsf{Y}}^{\Gamma}$. 

We can identify the natural map of the lemma with
\begin{align}
    \g(\mathbb{Q}) \backslash \g(\mathbb{A})/\mathsf{K}_{\mathsf{X}} \to \left(\h(\mathbb{Q}) \backslash \h(\mathbb{A})/\mathsf{K}_{\mathsf{Y}}\right)^{\Gamma}.
\end{align}
The lemma now follows from the fact that $\rH^1(\Gamma, \mathsf{K}_{\mathsf{Y}})=\{1\}$ in combination with Corollary \ref{Cor:NoCohomologyIsomorphism}.
\end{proof}
\subsection{Constructing good compact open subgroups} \label{Sec:GoodSubgroups}
In order to apply Lemma \ref{Lem:QuotientExactAdelicSpace} to Shimura varieties (or stacks), we need to understand the $\Gamma$-cohomology of compact open subgroups of $\haf$. In this section, we are going to show that there exist many compact open subgroups with vanishing $\Gamma$-cohomology, at least if $\f$ is tamely ramified over $\mathbb{Q}$.

Let $p$ be a prime number and write $\mathcal{O}_{F}:=\mathcal{O}_{\f} \otimes \zp$. It is a finite flat algebra over $\zp$ equipped with an action of $\Gamma$. If $\mathcal{G}$ is a smooth affine group scheme over $\zp$ with generic fiber isomorphic to $G := \mathsf{G} \otimes \qp$, we define $\mathcal{H}:=\operatorname{Res}_{\mathcal{O}_{F}/\zp} \mathcal{G}_{\mathcal{O}_{F}}$. It is a smooth affine group scheme over $\zp$ with generic fiber isomorphic to $\operatorname{Res}_{F / \qp}  G_{F}$, where $F=\f \otimes_{\mathbb{Q}} \qp$. Moreover the Galois group $\Gamma$ acts on $\mathcal{H}$. 

\begin{Lem} \label{Lem:Cofinal}
As $\mathcal{G}$ runs over all smooth affine group schemes over $\zp$ with connected special fiber and with generic fiber isomorphic to $G$, the groups $\mathcal{H}(\zp)$ form a cofinal collection of compact open subgroups of $H(\qp)$. 
\end{Lem}
\begin{proof}
By the results \cite[Remark A.5.14, Lemma A.5.15]{KalethaPrasad} we can construct a sequence of smooth group schemes $\mathcal{G}_1 \leftarrow \mathcal{G}_2 \leftarrow \mathcal{G}_3 \leftarrow \cdots$ over $\zp$, such that: The induced maps $\mathcal{G}_{i+1,\qp} \to \mathcal{G}_{i,\qp}$ are isomorphisms, there is an isomorphism $\mathcal{G}_{1,\qp} \simeq G$ and by \cite{KalethaPrasad}*{Lemma A.5.13} we have $$\bigcap_{i = 1}^\infty (\operatorname{Res}_{\mathcal{O}_{F}/\zp} \mathcal{G}_{i,\mathcal{O}_{F}})(\mathbb{Z}_p) = \{1\}.$$ The lemma follows by replacing each $\mathcal{G}_i$ with the open subgroup whose special fiber is the identity component of $\mathcal{G}_{i,\fp}$.
\end{proof}
We now show that if $p$ is unramified in $\f$, then the cohomology of $\mathcal{H}(\zp)$ is trivial.
\begin{Lem} \label{Lem:TrivialCohomologyUnramifiedPrime}
If $p$ is unramified in $\f$, then for each smooth affine group scheme $\mathcal{G}$ over $\zp$ with connected special fiber and with generic fiber isomorphic to $G$, the group $\mathcal{H}(\zp)=\mathcal{G}(\mathcal{O}_{F})$ satisfies $$\rH^1(\Gamma,\mathcal{H}(\zp))=\{1\}.$$ 
\end{Lem}
\begin{proof}
Choose a prime $\mathfrak{p}$ of $\f$ above $p$ and let $\Gamma' \subset \Gamma$ be its stabilizer. Then there is a $\Gamma$-equivariant isomorphism of rings
\begin{align}
    \mathcal{O}_{F} \simeq \operatorname{Map}_{\Gamma'}(\Gamma, \mathcal{O}_{F,\mathfrak{p}}),
\end{align}
which induces a $\Gamma$-equivariant isomorphism of groups
\begin{align}
    \mathcal{H}(\zp)&=\mathcal{G}(\mathcal{O}_{F}) \\
    &=\operatorname{Map}_{\Gamma'}(\Gamma, \mathcal{G}(\mathcal{O}_{F,\mathfrak{p}})).
\end{align}
By Lemma \ref{Lem:Shapiro}, it then suffices to show that
\begin{align}
    \rH^1(\Gamma', \mathcal{G}(\mathcal{O}_{F,\mathfrak{p}}))=\{1\}.
\end{align}
Consider the inflation map (which is injective by \cite[Section 5.8.(a)]{SerreGaloisCohomology})
\begin{align}
    \rH^1(\Gamma', \mathcal{G}(\mathcal{O}_{F,\mathfrak{p}})) \to \rH^1(\gal(\qpur/\qp, \mathcal{G}(\zpur)).
\end{align}
Because $\mathcal{H}$ is smooth with connected special fiber, Lang's lemma tells us that $$\rH^1(\gal(\qpur/\qp), \mathcal{G}(\zpur))=\rH^1(\gal(\ovfp/\mathbb{F}_p), \mathcal{G}(\ovfp))=\{1\}$$ and so we are done.
\end{proof}
If $p$ is coprime to the order of $\Gamma$, then the cohomology of $\mathcal{H}(\zp)$ is trivial if it is a pro-$p$ group.
\begin{Lem} \label{Lem:CohomologyPrimeToPFinite}
Let $\Gamma$ be a finite group acting on a finite $p$-group $M$. If the order of $\Gamma$ is coprime to $p$, then $\rH^1(\Gamma,M)=\{1\}$.
\end{Lem}
\begin{proof}
We prove this by strong induction on $k$ where $p^k$ is the order of $M$. If $k=1$ then the result is true because then $M \simeq \fp$ is abelian and $\Gamma$ is a group of order prime-to-$p$.

To prove the induction step, we use the well-known fact that the center $Z \subset M$ of $M$ is nontrivial because $M$ is a $p$-group. Since $\Gamma$ acts via group homomorphisms, we see that $Z$ is $\Gamma$-stable. We now study the long exact sequence in $\Gamma$-cohomology for
\begin{align} \label{Eq:ShortExactSequenceCohomology}
    1 \to Z \to M \to M/Z \to 1.
\end{align}
The induction hypothesis tells us that $\rH^1(\Gamma, M/Z)=\{1\}$ and because $Z$ is abelian and of $p$-power order it follows that $\rH^1(\Gamma,Z)=\{1\}$. The induction step then follows from the long exact sequence in $\Gamma$-cohomology for \eqref{Eq:ShortExactSequenceCohomology}.
\end{proof}
\begin{Lem} \label{Lem:CohomologyTrivialPrimeToP}
If $p$ is coprime to the order of $\Gamma$ and $\mathcal{H}(\zp)$ is pro-$p$, then $\rH^1(\Gamma,\mathcal{H}(\zp))=\{1\}$.
\end{Lem}
\begin{proof}
There is a $\Gamma$-equivariant identification $\mathcal{H}(\zp)=\varprojlim_n \mathcal{H}(\mathbb{Z}/p^n \mathbb{Z})$. Let $\sigma:\Gamma \to \mathcal{H}(\zp)$ be a cocycle and let $P$ be the (possibly empty) set of elements $h \in \mathcal{H}(\zp)$ such that for all $\gamma \in \Gamma$ we have $1=h \cdot \sigma(\gamma) \cdot \alpha_{\gamma}(\rH^{-1})$, where $\alpha_{\gamma}:\mathcal{H}(\zp) \to \mathcal{H}(\zp)$ denotes the action of $\gamma$. Then $\mathcal{H}(\zp)^{\Gamma}$ acts on $P$ by left multiplication on $h$ and this action is simply transitive if $P$ is nonempty. The lemma is asserting that $P$ is nonempty, which we will now prove.

For a positive integer $n$, we let $P_n$ be the set of elements $h_n \in \mathcal{H}(\mathbb{Z}/p^n \mathbb{Z})$ such that for all $\gamma \in \Gamma$ we have $1=h_n \cdot \sigma_n(\gamma) \cdot \alpha_{\gamma}(h_n^{-1})$, where $\sigma_n$ is the composition of $\sigma$ with $\mathcal{H}(\zp) \to \mathcal{H}(\mathbb{Z}/p^n \mathbb{Z})$. Then $P_n$ is nonempty by Lemma \ref{Lem:CohomologyPrimeToPFinite} and $\mathcal{H}(\mathbb{Z}/p^n \mathbb{Z})^{\Gamma}$ acts simply transitively on $P_n$ by left multiplication on $h_n$. There are maps $P_{n+1} \to P_{n}$ which are 
$\mathcal{H}(\mathbb{Z}/p^{n+1} \mathbb{Z})^{\Gamma}$-equivariant via the natural map $\mathcal{H}(\mathbb{Z}/p^{n+1}\mathbb{Z})^{\Gamma} \to \mathcal{H}(\mathbb{Z}/p^n \mathbb{Z})^{\Gamma}$, and it follows from the definitions that the natural map
\begin{align}
    P \to \varprojlim_n P_n
\end{align}
is a bijection. To show that $P$ is nonempty, it is therefore enough to show that the transition maps $P_{n+1} \to P_n$ are surjective. For this, we consider the $\Gamma$-equivariant short exact sequence
\begin{align}
    1 \to Q \to \mathcal{H}(\mathbb{Z}/p^{n+1} \mathbb{Z}) \to \mathcal{H}(\mathbb{Z}/p^n \mathbb{Z}) \to 1
\end{align}
defining $Q$. Note that $Q$ is again a $p$-group, and thus it follows from Lemma \ref{Lem:CohomologyPrimeToPFinite} and the long exact sequence in $\Gamma$-cohomology that $\mathcal{H}(\mathbb{Z}/p^{n+1} \mathbb{Z})^{\Gamma} \to \mathcal{H}(\mathbb{Z}/p^n \mathbb{Z})^{\Gamma}$ is surjective. We deduce that $P_{n+1} \to P_n$ is surjective, which concludes the proof.
\end{proof}
\begin{Lem} \label{Lem:NoCohomologyTame}
Suppose that $p$ is tamely ramified in $\f$. Let $\mathcal{G}$ be a smooth affine group scheme over $\zp$ with connected special fiber and with generic fiber isomorphic to $G$. If $\mathcal{H}(\zp)=\mathcal{G}(\mathcal{O}_{F})$ is a pro-$p$ group then $$\rH^1(\Gamma,\mathcal{H}(\zp))=\{1\}.$$ 
\end{Lem}
\begin{proof}
Choose a prime $\mathfrak{p}$ of $\f$ above $p$ and let $L$ be the completion of $\f$ at $\mathfrak{p}$. Let $\gal(L/\mathbb{Q}_p) \subset \Gamma$ be the stabilizer of $\mathfrak{p}$. As in the proof of Lemma \ref{Lem:TrivialCohomologyUnramifiedPrime}, we can use Lemma \ref{Lem:Shapiro} to reduce the lemma to proving that
\begin{align}
    \rH^1(\gal(L/\mathbb{Q}_p), \mathcal{G}(\mathcal{O}_L))=\{1\}.
\end{align}
Let $L_0 \subset L$ be the maximal unramified subfield of $L$ and consider the short exact sequence of Galois groups
\begin{align}
    1 \to I_{L} \to \gal(L/\qp) \to \gal(L_0/\qp) \to 1.
\end{align}
Then we get an inflation restriction exact sequence of pointed sets
\begin{align}
    \cdots \to \rH^1(\gal(L_0/\qp), \mathcal{G}(\mathcal{O}_L)^{I_{L}}) \to \rH^1(\gal(L/\mathbb{Q}_p), \mathcal{G}(\mathcal{O}_L)) \to \rH^1(I_{L}, \mathcal{G}(\mathcal{O}_L)).
\end{align}
The natural map $\mathcal{G}(\mathcal{O}_{L_0}) \to \mathcal{G}(\mathcal{O}_L)^{I_{L}}$ is an isomorphism and so the first term is trivial by Lemma \ref{Lem:TrivialCohomologyUnramifiedPrime}. The assumption that $p$ is tamely ramified in $\f$ means that $I_{L}$ has order prime to $p$ and thus $\rH^1(I_{L}, \mathcal{G}(\mathcal{O}_L))$ vanishes by Lemma \ref{Lem:CohomologyTrivialPrimeToP} and the fact that $\mathcal{H}(\zp)$ is pro-$p$. It now follows from \cite[Corollary 1 on p. 52]{SerreGaloisCohomology} that $\rH^1(\gal(L/\mathbb{Q}_p), \mathcal{G}(\mathcal{O}_L))=\{1\}$.
\end{proof}

\subsubsection{}
The results proved above provide many compact open subgroups $K \subset \hafp$ that are $\Gamma$-stable and have $\rH^1(\Gamma, K)=\{1\}$, at least when $\f$ is tamely ramified over $\mathbb{Q}$.

\begin{Prop} \label{Prop:ExistenceCofinalGoodCompactOpens}
Suppose that $\f$ is tamely ramified over $\mathbb{Q}$. Then the collection of compact open subgroups $K \subset \haf$ that are $\Gamma$-stable and satisfy $\rH^1(\Gamma,K)=\{1\}$ is cofinal in the set of all compact open subgroups of $\haf$. 
\end{Prop}
We will call such compact open subgroups \emph{good}, and we will use the same terminology for $\Gamma$-stable compact open subgroups $K \subset G(\mathbb{A}_f^S)$, for some finite set of places $S$ of $\mathbb{Q}$, that satisfy $\rH^1(\Gamma,K)=\{1\}$.
\begin{proof}[Proof of Proposition \ref{Prop:ExistenceCofinalGoodCompactOpens}]
We can choose groups $K$ of the form
\begin{align}
    K=\prod_{p} K_p,
\end{align}
with $K_p = \mathcal{H}(\mathbb{Z}_p)$ for $H=\operatorname{Res}_{\mathcal{O}_{F}/\zp} \mathcal{G}_{\mathcal{O}_{F}}$ for some smooth affine group scheme $\mathcal{G}$ over $\zp$ with connected special fiber and generic fiber $\g \otimes \qp$, such that $\mathcal{G}$ is a reductive model of $G$ for all but finitely many $p$. This collection of compact open subgroups $K \subset \haf$ is cofinal by Lemma \ref{Lem:Cofinal}. We can moreover assume that either $p$ is unramified in $\f$ or that $K_p$ is a pro-$p$ group, without affecting co-finality. \smallskip 

Then for primes $p$ unramified in $\f$ we have
\begin{align}
    \rH^1(\Gamma, K_p)=\{1\}
\end{align}
by Lemma \ref{Lem:TrivialCohomologyUnramifiedPrime}, and for primes $p$ ramified in $\f$ we have $\rH^1(\Gamma, K_p)=\{1\}$ by Lemma \ref{Lem:NoCohomologyTame}. Thus we find that 
\begin{align}
    \rH^1(\Gamma,K)=\prod_p \rH^1(\Gamma, K_p)=\{1\}
\end{align}
and the result is proved.
\end{proof}
\subsection{Homotopy fixed points of Shimura stacks} Let the notation be as in the beginning of Section \ref{Sec:Shimura}. If $K \subset \haf$ is a $\Gamma$-stable compact open subgroup then there is a natural morphism of groupoids (see Section \ref{Sec:AppendixQuotientStackFixedPoints})
\begin{align} \label{Eq:ShimuraGroupoid}
    \left[\g(\mathbb{Q}) \backslash \left(\mathsf{X} \times \gaf/K^{\Gamma}\right)\right] \to \left[\h(\mathbb{Q}) \backslash \left(\mathsf{Y} \times \haf/K\right)\right]^{h \Gamma}.
\end{align}
In this section we will investigate when this morphism is an equivalence.  
\begin{Thm} \label{Thm:HomotopyFixedPointsShimura}
If $K \subset \haf$ is a good compact open subgroup and if $\Sha^1(\mathbb{Q}, \mathsf{G}) \to \Sha^1(\f, \g)$ is injective, then the natural functor
\begin{align}
    \left[\g(\mathbb{Q}) \backslash \left(\mathsf{X} \times \gaf/K^{\Gamma}\right)\right] \to \left[\h(\mathbb{Q}) \backslash \left(\mathsf{Y} \times \haf/K\right)\right]^{h \Gamma}
\end{align}
is an equivalence of groupoids.
\end{Thm}
It follows from the discussion in Section \ref{Sec:Tori} that the assumption that $\Sha^1(\mathbb{Q}, \mathsf{G}) \to \Sha^1(\f, \g)$ is injective is necessary. It follows from Section \ref{Sec:Tori} that it is necessary that $\rH^1(\Gamma,K) \to \rH^1(\Gamma,\haf)$ has trivial kernel.
\begin{Cor} \label{Cor:FixedPointsShimuraSV5}
Suppose that the $\mathbb{R}$-split rank of the center of $\g$ is zero. If $K \subset \haf$ is a neat and good compact open subgroup and $\Sha^1(\mathbb{Q}, \mathsf{G}) \to \Sha^1(\f, \g)$ is injective, then the natural morphism of $\mathsf{E}$-varieties
\begin{align}
    \mathbf{Sh}_{K^{\Gamma}}\gx \to \mathbf{Sh}_{K}\hy^{\Gamma}
\end{align}
is an isomorphism.
\end{Cor}
\begin{proof}
It suffices to prove this result after basechanging to $\mathbb{C}$. By \cite[Proposition 4.2]{Edixhoven}, the target is smooth because $\mathbf{Sh}_{K}\hy$ is smooth. Therefore, it is enough to prove that the map induces a bijection on $\mathbb{C}$-points. Now recall from Lemma \ref{Lem:SV5} that both $\gx$ and $\hy$ satisfy SV5. Since $K$ is neat it follows that $K^{\Gamma}=K \cap \gaf$ is neat and so by Lemma \ref{Lem:FreeActionShimura}, the groupoid quotients in the statement of Theorem \ref{Thm:HomotopyFixedPointsShimura} are equivalent to the set theoretic quotients. The corollary now follows from Theorem \ref{Thm:HomotopyFixedPointsShimura}.
\end{proof}

To prove Theorem \ref{Thm:HomotopyFixedPointsShimura}, it is more convenient to work with a different presentation of the morphisms of groupoids in equation \eqref{Eq:ShimuraGroupoid}. Informally, we want to swap around the order in which we are taking the quotient.\footnote{What follows is presumably a tautology for the readers well versed in $2$-category theory. We have opted to spell out the details for our own benefit.}

A compact open subgroup $K \subset \haf$ acts on $\mathsf{Y} \times \haf$ by $k \cdot (x,g)=(x,g \cdot k)$. This action commutes with the action of $\h(\mathbb{Q})$ and therefore descends to an action of $K$ on
\begin{align}
    \h(\mathbb{Q}) \backslash (\mathsf{Y} \times \haf).
\end{align}
Moreover, if $K$ is $\Gamma$-stable then this action is $\Gamma$-semilinear and thus induces an action of $\Gamma$ on the quotient stack of $\h(\mathbb{Q}) \backslash (\mathsf{Y} \times \haf)$ by $K$. There is moreover an action of $K^{\Gamma}$ on $\g(\mathbb{Q}) \backslash (\mathsf{X} \times \gaf)$ and the natural map
\begin{align}
    \g(\mathbb{Q}) \backslash (\mathsf{X} \times \gaf) \to \h(\mathbb{Q}) \backslash (\mathsf{Y} \times \haf) 
\end{align}
induces a $\Gamma$-equivariant map of quotient stacks, see Section \ref{Sec:AppendixQuotientStackFixedPoints}. The following lemma tells us that we can indeed swap around the order in which the quotient is taken.
\begin{Lem} \label{Lem:DifferentPresentations}
The natural functor
\begin{align}
    \left[\g(\mathbb{Q}) \backslash \left(\mathsf{X} \times \gaf/K^{\Gamma}\right)\right] \to \left[\h(\mathbb{Q}) \backslash \left(\mathsf{Y} \times \haf/K\right)\right]^{h \Gamma}
\end{align}
is an equivalence if and only if 
\begin{align}
\left[\left(\g(\mathbb{Q}) \backslash \left(\mathsf{X} \times \gaf\right)\right)/K^{\Gamma}\right] \to \left[\left(\h(\mathbb{Q}) \backslash \left(\mathsf{Y} \times \haf\right)\right)/K\right]^{h \Gamma}
\end{align}
is an equivalence.
\end{Lem}
\begin{proof}
This is a straightforward consequence of Lemma \ref{Lem:FunctorialityQuotientStack} and equation \ref{Eq:KeyDiagram}. 
\end{proof}

\begin{proof}[Proof of Theorem \ref{Thm:HomotopyFixedPointsShimura}]
It follows from the assumptions of the theorem and Lemma \ref{Lem:QuotientExactAdelicSpace} that the natural map
\begin{align}
    (\g(\mathbb{Q}) \backslash \left(\mathsf{X} \times \gaf\right) \to \left(\h(\mathbb{Q}) \backslash \left(\mathsf{Y} \times \haf\right)\right)^{\Gamma}
\end{align}
is a bijection. To show that the natural map 
\begin{align} \label{Eq:AlternativePresentationMap}
\left[\left(\g(\mathbb{Q}) \backslash \left(\mathsf{X} \times \gaf\right)\right)/K^{\Gamma}\right] \to \left[\left(\h(\mathbb{Q}) \backslash \left(\mathsf{Y} \times \haf\right)\right)/K\right]^{h \Gamma}
\end{align}
is an equivalence, it suffices by Corollary \ref{Cor:NoCohomologyIsomorphism} to show that $\rH^1(\Gamma,K)=\{1\}$, which is true by assumption. Therefore, the map in equation \eqref{Eq:AlternativePresentationMap} is an equivalence and the theorem now follows from Lemma \ref{Lem:DifferentPresentations}.
\end{proof}
The following corollary is a direct consequence of the proof of Theorem \ref{Thm:HomotopyFixedPointsShimura}.
\begin{Cor} \label{Cor:DiscreteGroupoid}
If $K \subset \haf$ is a $\Gamma$-stable subgroup (not necessarily compact open!) with $\rH^1(\Gamma,K)=\{1\}$, then the natural map
\begin{align}
    \left[\g(\mathbb{Q}) \backslash \left(\mathsf{X} \times \gaf/K^{\Gamma}\right)\right] \to \left[\h(\mathbb{Q}) \backslash \left(\mathsf{Y} \times \haf/K\right)\right]^{h \Gamma}
\end{align}
is an equivalence. In particular, the right-hand side is equivalent to a discrete groupoid if $\gx$ satisfies SV5 and $K$ is neat.
\end{Cor}

\subsection{Shimura varieties of Hodge type} \label{Sec:FixedPointsHodge}  Let the notation be as in the beginning of Section \ref{Sec:Shimura}. For a symplectic space $(V, \psi)$ over $\mathbb{Q}$ we write $\mathsf{G}_V:=\operatorname{GSp}(V, \psi)$ for the group of symplectic similitudes of $V$ over $\mathbb{Q}$. It admits a Shimura datum $\mathsf{H}_V$ consisting of the union of the Siegel upper and lower half spaces. Assume furthermore that $\gx$ is of Hodge type and let $\iota:\gx \to \gvx$ be a closed immersion of Shimura data. Let $\f$ be a Galois totally real field with Galois group $\Gamma$. 

Recall the following construction from \cite[Section 7.1.6]{KisinZhou}\footnote{They use $\mathbf{H}'$ for what we call $\h$ and they use $\mathbf{H}$ for what we call $\ho$.}. Let $W$ be the symplectic space $V \otimes_{\mathbb{Q}} \f$ considered as a vector space over $\mathbb{Q}$, equipped with the symplectic form $\psi_W$ given by
\begin{equation}
    \begin{tikzcd}
         W \times W \arrow{r}{\psi \otimes_{\mathbb{Q}} \f} & \f \arrow{r}{\operatorname{Tr}_{\f/\mathbb{Q}}} & \mathbb{Q}.
    \end{tikzcd}
\end{equation}
Let $c_{\g}:\g \to \mathbb{G}_m$ be the restriction of the symplectic similitude character of $\gv$ to $\g$ and let $c_{\g, \f}:\h \to \operatorname{Res}_{\f   /\mathbb{Q}} \mathbb{G}_m$ be the induced map. Form the fiber product
\begin{equation}
    \begin{tikzcd}
        \h_3 \arrow{r} \arrow{d} & \mathbb{G}_m \arrow{d}{} \arrow{d} \\
        \h \arrow{r}{c_{\mathsf{G},\f}} & \operatorname{Res}_{\f/\mathbb{Q}} \mathbb{G}_m
    \end{tikzcd}
\end{equation}
and let $\ho$ be the neutral component of $\h_3$. Then $\ho$ is a connected reductive group over $\mathbb{Q}$ and the natural map $\g \to \h$ factors over $\ho$. For $h:\mathbb{S} \to \g_{\mathbb{R}}$ in $\mathsf{X}$ we let $\mathsf{Y}_1$ be the $\ho(\mathbb{R})$-conjugacy class of the composition $h:\mathbb{S} \to \g_{\mathbb{R}} \to \h_{1,\mathbb{R}}$, this does not depend on the choice of $h$. We will write $\gvf$ for the inverse image of $\mathbb{G}_m \subset \operatorname{Res}_{\f/\mathbb{Q}} \mathbb{G}_m$ under the natural map
\begin{align}
    \operatorname{Res}_{\f/\mathbb{Q}} G_V \to \operatorname{Res}_{\f/\mathbb{Q}} \mathbb{G}_m.
\end{align}
As explained above, this also comes equipped with a Shimura datum $\mathsf{H}_{V,\f}$. There is moreover a commutative diagram 
\begin{equation} \label{Eq:HodgeModifiedCenterDiagram}
    \begin{tikzcd}
        \g \arrow{r} \arrow{d}{\iota} & \ho \arrow{d} \arrow{dr}{\iota_{\ho}} \\
        \gv \arrow{r} & \gvf \arrow{r} & \gw,
    \end{tikzcd}
\end{equation}
where each map is a closed immersion compatible with the natural Shimura data.

\begin{Rem}
    The group $\ho$ does not depend on the choice of $\iota$, since the map $c_{\g}$ does not depend on the choice of $\iota$, see \cite[Lemma 7.1.1]{DvHKZIgusaStacks}.
\end{Rem}

\subsubsection{} Shimura varieties of Hodge type automatically satisfy SV5 and in fact their centers have $\mathbb{Q}$-split and $\mathbb{R}$-split ranks equal to $1$. Recall from Lemma \ref{Lem:SV5} that if $[\f:\mathbb{Q}]>1$ then $\hy$ does not satisfy SV5. Since $\hyo$ is of Hodge type, it does satisfy SV5. The point of the construction of $\ho$ is both to remedy the failure of SV5 and to build something that is again of Hodge type. \smallskip

We would like to compare the $\Gamma$-fixed points of Shimura varieties for $\hyo$ to the Shimura varieties for $\gx$. It seems complicated to do this directly, so we will instead compare the $\mathbb{C}$-points of the $\Gamma$-fixed points of the Shimura varieties for $\hyo$, to the $\Gamma$-homotopy fixed points of the Shimura stacks for $\hy$. We will need the following modification of \cite[Lemma 2.1.2]{KisinModels}, cf. \cite[Lemma 2.4.3]{Lovering}. 
\begin{Prop} \label{Prop:Injective}
Let $\ell$ be a prime number coprime to the order of $\Gamma$. Let $K^{\ell} \subset \h(\mathbb{A}_f^{\ell})$ be a neat good compact open subgroup and let $K_{\ell} \subset \g(\ql)$ be a $\Gamma$-stable pro-$\ell$ group. Define $K=K^{\ell} K_{\ell}$ and define $K_{1,\ell}=K_{\ell} \cap \ho(\ql)$ and $K^{1,\ell}=K^{\ell} \cap \ho(\af^{\ell})$. Then the natural map
\begin{align}
    \left(\ho(\mathbb{Q}) \backslash \mathsf{Y}_1 \times \h_1(\af)/K_1^\ell K_{1,\ell}\right)^{\Gamma} \to \left[\h(\mathbb{Q}) \backslash \mathsf{Y} \times \haf/K^{\ell} K_{1,\ell}\right]^{h \Gamma}
\end{align}
is fully faithful.\footnote{The groupoid on the right hand side is equivalent to a discrete groupoid by Corollary \ref{Cor:DiscreteGroupoid}, and thus the lemma is equivalent to saying that the map on isomorphism classes of objects is injective.}
\end{Prop}

\begin{proof}[Proof of Proposition \ref{Prop:Injective}]
We first prove that the natural functor
\begin{align} \label{Eq:MapStrangeInverseLimit}
    \left(\ho(\mathbb{Q}) \backslash \mathsf{Y}_1 \times \h_1(\af)/K_1^\ell \right)^{\Gamma} \to \left[\h(\mathbb{Q}) \backslash \mathsf{Y} \times \haf/K^{\ell} \right]^{h \Gamma}
\end{align}
is fully faithful. For this we consider the commutative diagram
\begin{equation}
    \begin{tikzcd}
    \ho(\mathbb{Q}) \backslash \mathsf{Y}_1 \times \h_1(\af)/K_1^\ell  \arrow{r} \arrow{dr} & \left[\h(\mathbb{Q}) \backslash \mathsf{Y} \times \haf/K^{\ell} \right] \arrow{d} \\
    & \h(\mathbb{Q}) \backslash \mathsf{Y} \times \haf/K^{\ell}.
    \end{tikzcd}
\end{equation}
The proof of \cite[Lemma 2.4.3]{Lovering} establishes that the diagonal arrow is injective: Indeed, we consider the commutative diagram
\begin{equation}
\begin{tikzcd}
        \ho(\mathbb{Q}) \backslash \mathsf{Y}_1 \times \h_1(\af)/K_1^\ell \arrow{d} \arrow{r} & \h(\mathbb{Q}) \backslash \mathsf{Y} \times \haf/K^{\ell} \arrow{d} \\
    \ho(\mathbb{Q}) \backslash H_1(\ql) \arrow{r} & \h(\mathbb{Q}) \backslash H(\ql).
\end{tikzcd}
\end{equation}
The bottom arrow is injective since $\h(\mathbb{Q}) \cap H_1(\ql) = \ho(\mathbb{Q})$. The fibers of the vertical maps can be identified with $\mathsf{Y}_1 \times \ho(\mathbb{A}_f^{\ell})/K_1^\ell$ and $\mathsf{Y} \times \h(\mathbb{A}_f^{\ell})/K^{\ell}$, respectively, and the natural map between them is injective since $K_1^{\ell}=K^{\ell} \cap \ho(\mathbb{A}_f^{\ell})$. We conclude that the top arrow is injective, and it follows that the diagonal arrow in the following diagram is injective
\begin{equation}
    \begin{tikzcd}
    \left(\ho(\mathbb{Q}) \backslash \mathsf{Y}_1 \times \h_1(\af)/K_1^\ell \right)^{\Gamma} \arrow{r} \arrow[dr, hook] & \left[\h(\mathbb{Q}) \backslash \mathsf{Y} \times \haf/K^{\ell} \right]^{h \Gamma} \arrow{d} \\
    & \left(\h(\mathbb{Q}) \backslash \mathsf{Y} \times \haf/K^{\ell}\right)^{\Gamma}.
    \end{tikzcd}
\end{equation}
Since $\rH^1(\Gamma, K^{\ell})=\{1\}$ by assumption, it follows from Corollary \ref{Cor:DiscreteGroupoid} that $$\left[\h(\mathbb{Q}) \backslash \mathsf{Y} \times \haf/K^{\ell} \right]^{h \Gamma}$$ is a discrete groupoid. Therefore $\left(\ho(\mathbb{Q}) \backslash \mathsf{Y}_1 \times \h_1(\af)/K_1^\ell \right)^{\Gamma} \to \left[\h(\mathbb{Q}) \backslash \mathsf{Y} \times \haf/K^{\ell} \right]^{h \Gamma}$ must be a fully faithful map of discrete groupoids. To continue the proof we consider the $2$-commutative diagram
\begin{equation}
    \begin{tikzcd}
    \ho(\mathbb{Q}) \backslash \mathsf{Y}_1 \times \h_1(\af)/K_1^\ell  \arrow{d} \arrow{r} & \left[\h(\mathbb{Q}) \backslash \mathsf{Y} \times \haf/K^{\ell} \right]\arrow{d} \\
    \ho(\mathbb{Q}) \backslash \mathsf{Y}_1 \times \h_1(\af)/K_1^\ell K_{1,\ell} \arrow{r} & \left[\h(\mathbb{Q}) \backslash \mathsf{Y} \times \haf/K^{\ell} K_{1,\ell}\right],
    \end{tikzcd}
\end{equation}
which we note is $2$-Cartesian since the vertical arrows are essentially surjective with fibers given by the discrete groupoid associated to $K_{1,\ell}$. By Lemma \ref{Lem:FiberProducts}, it follows that the diagram stays Cartesian when applying $\Gamma$-homotopy fixed points. Since $K_{\ell}$ is pro-$\ell$ by assumption, it follows that $K_{1,\ell}$ is also pro-$\ell$ and thus $\rH^1(\Gamma,K_{1,\ell})=\{1\}$ by Lemma \ref{Lem:CohomologyTrivialPrimeToP}. Therefore, it follows from Corollary \ref{Cor:NoCohomologyIsomorphism} that the left vertical map in the diagram
\begin{equation}
    \begin{tikzcd}
    \left(\ho(\mathbb{Q}) \backslash \mathsf{Y}_1 \times \h_1(\af)/K_1^\ell\right)^{\Gamma}  \arrow{d} \arrow{r} & \left[\h(\mathbb{Q}) \backslash \mathsf{Y} \times \haf/K^{\ell} \right]^{h\Gamma}\arrow{d} \\
    \left(\ho(\mathbb{Q}) \backslash \mathsf{Y}_1 \times \h_1(\af)/K_1^\ell K_{1,\ell}\right)^{\Gamma} \arrow{r} & \left[\h(\mathbb{Q}) \backslash \mathsf{Y} \times \haf/K^{\ell} K_{1,\ell}\right],
    \end{tikzcd}
\end{equation}
is (essentially) surjective, and fully faithfullness of the top row therefore implies the fully faithfulness of the bottom row. 
\end{proof}
\begin{Thm} \label{Thm:FixedPointsHodge}
Let $K \subset \haf$ be a neat and good compact open subgroup. If there is a prime number $\ell$ coprime to the order of $\Gamma$ such that $K=K^{\ell} K_{\ell}$ with $K_{\ell}$ a pro-$\ell$ group, and if $\Sha^1(\mathbb{Q}, \mathsf{G}) \to \Sha^1(\f, \g)$ is injective, then the natural map
\begin{align}
    \mathbf{Sh}_{K^{\Gamma}}\gx \to \mathbf{Sh}_{K_1}\hyo^{\Gamma}
\end{align}
is an isomorphism.
\end{Thm}
\begin{proof}
As in the proof of Corollary \ref{Cor:FixedPointsShimuraSV5}, we reduce to showing that the map is a bijection on $\mathbb{C}$-points. This map has the form
\begin{align}
    \g(\mathbb{Q}) \backslash \mathsf{X} \times \gaf/K^{\Gamma} \to \left(\ho(\mathbb{Q}) \backslash (\mathsf{Y}_1 \times \h_1(\af)/K_1 \right)^{\Gamma}
\end{align}
and we consider the composition
\begin{align} \label{Eq:FurtherComposite}
    \g(\mathbb{Q}) \backslash \mathsf{X} \times \gaf/K^\Gamma&\to \left(\ho(\mathbb{Q}) \backslash (\mathsf{Y}_1 \times \h_1(\af)/K_1)\right)^{\Gamma} \\
    &\to \left[\h(\mathbb{Q}) \backslash (\mathsf{X} \times \haf/K^{\ell} K_{1,\ell})\right]^{h \Gamma}.
\end{align}
Let $\ell$ and $K_{\ell}$ be as in the statement of the theorem. Since $K_{\ell}$ is pro-$\ell$ by assumption, it follows that $K_{1,\ell}$ is also pro-$\ell$ and thus $\rH^1(\Gamma,K_{1,\ell})=\{1\}$ by Lemma \ref{Lem:CohomologyTrivialPrimeToP}. Therefore it follows from Corollary \ref{Cor:DiscreteGroupoid} that this composition is an equivalence of categories. By Proposition \ref{Prop:Injective}, the map
\begin{align}
    \left(\ho(\mathbb{Q}) \backslash \mathsf{Y}_1 \times \h_1(\af)/K_1^\ell K_{1,\ell}\right)^{\Gamma} \to \left[\h(\mathbb{Q}) \backslash \mathsf{Y} \times \haf/K^{\ell} K_{1,\ell}\right]^{h \Gamma}
\end{align}
is fully faithful and thus it follows that 
\begin{align}
    \g(\mathbb{Q}) \backslash \mathsf{X} \times \gaf/K^\Gamma&\to \left(\ho(\mathbb{Q}) \backslash (\mathsf{Y}_1 \times \h_1(\af)/K_1)\right)^{\Gamma} \\
\end{align}
is an equivalence, since its composition with a fully faithful map is an equivalence. Since an equivalence of categories between discrete categories gives a bijection of the underlying sets, we are done.
\end{proof}
\begin{Rem}
Our original approach to proving this theorem was to directly study the $\Gamma$-fixed points of the Shimura variety for $\hyo$ using non-abelian cohomology; this calculation proved quite difficult. This is why it is useful to consider homotopy fixed points of Shimura stacks, even if one is just interested in studying Shimura varieties of Hodge type.
\end{Rem}
\subsubsection{} Finally, we collect a result for use in Section \ref{Sec:IgusaStacks}. For $K \subset \haf$ a $\Gamma$-stable compact open subgroup we will consider $K^{\Gamma} \subset \gaf$ and $K_1 \subset \h_1(\af)$. 
\begin{Lem} \label{Lem:FixedPointsInfiniteLevelHodge}
    If $\Sha^1(\mathbb{Q},\g) \to \Sha^1(\mathbb{Q},\h)$ is injective, then the natural map of $\mathsf{E}$-schemes
    \begin{align}
        \kappa:\varprojlim_{K \subset \haf} \mathbf{Sh}_{K^{\Gamma}}\gx \to \left(\varprojlim_{K \subset \haf} \mathbf{Sh}_{K_1}\hyo\right)^{\Gamma},
    \end{align}
    where the limit runs over $\Gamma$-stable compact open subgroups, is an isomorphism.
\end{Lem}
\begin{proof}
It suffices to prove this after basechanging to $\mathbb{C}$. To prove it is an isomorphism over $\mathbb{C}$, we first show it induces a bijection on $\mathbb{C}$-points. There we are looking at the map (injective by \cite[Variante 1.15.1]{DeligneTravaux}) 
\begin{align}
    \g(\mathbb{Q}) \backslash \mathsf{X} \times \gaf \to \left(\h_1(\mathbb{Q}) \backslash \mathsf{Y}_1 \times \h_1(\af) \right)^{\Gamma}.
\end{align}
As before we consider the composition
\begin{align}
    \g(\mathbb{Q}) \backslash \mathsf{X} \times \gaf \to \left(\h_1(\mathbb{Q}) \backslash \mathsf{Y}_1 \times \h_1(\af) \right)^{\Gamma} \to \left(\h(\mathbb{Q}) \backslash \mathsf{Y} \times \h(\af) \right)^{\Gamma},
\end{align}
which is an isomorphism by Lemma \ref{Lem:QuotientExactAdelicSpace}. It now suffices to show that 
\begin{align}
    \left(\h_1(\mathbb{Q}) \backslash \mathsf{Y}_1 \times \h_1(\af) \right)^{\Gamma} \to \left(\h(\mathbb{Q}) \backslash \mathsf{Y} \times \h(\af) \right)^{\Gamma}
\end{align}
is injective, which can be checked before taking $\Gamma$-fixed points. But then the result is well known, see \cite[Variante 1.15.1]{DeligneTravaux}. \smallskip 

To deduce that $\kappa$ is an isomorphism, we argue as follows: It follows from the proof of \cite[Proposition 1.15]{DeligneTravaux} that $\kappa$ is a closed immersion, thus it suffices to show the map is surjective on topological spaces. Both the source and target are Jacobson schemes since they are integral over the bottom object of the inverse limit, which is Jacobson; thus $\mathbb{C}$-points are dense in the source and target. Furthermore, the map $\kappa$ is integral since it is an inverse limit of integral maps of schemes. Therefore $\kappa$ is universally closed, and thus surjective since the image is closed and contains the dense set of $\mathbb{C}$-points.
\end{proof}

\section{Integral models of Shimura varieties} \label{Sec:IntegralModels} 
In this section we prove some results about integral models of Shimura varieties of Hodge type that we will need, in particular some $\Gamma$-equivariance properties for the integral models for $\hyo$. In Section \ref{Sec:ShtukasGeneric} we prove that the association of the universal local shtuka to a Shimura variety is functorial in $\gxg$ in the two-categorical sense, which we use to deduce $\Gamma$-equivariance in the two-categorical sense. This is transferred to shtukas on integral models in Section \ref{Sec:IntegralModels}. In Section \ref{Sec:LocalModelDiagrams}, we prove that certain $\Gamma$-equivariant morphisms of Shimura varieties can be upgraded to $\Gamma$-equivariant morphisms of the corresponding local model diagrams. 
\subsection{Shtukas on Shimura varieties} \label{Sec:ShtukasGeneric}
Let $\gx$ be a Shimura datum with reflex field $\mathsf{E}$ satisfying Milne's axiom SV5 from \cite[p. 64]{Milne}. Fix a prime $p$ and write $G=\mathsf{G} \otimes \qp$. Fix a prime $v$ of $\mathsf{E}$ above $p$ and let $E$ be the $v$-adic completion of $\mathsf{E}$ with ring of integers $\mathcal{O}_E$ and residue field $k_E$. Let $\mu$ be the $G(\qpbar)$-conjugacy class of cocharacters of $G$ coming from the Hodge cocharacter of some $x \in \mathsf{X}$ and the choice of place $v$. 

Let $\mathcal{G}$ be a quasi-parahoric model of $G$ over $\zp$ and set $\mathcal{G}(\zp):=K_p \subset G(\qp)$. For $K^p \subset \gaf$ a neat compact open subgroup we write $K=K^pK_p$ and consider the Shimura variety $\mathbf{Sh}_{K}\gx$ over $E$. The group $\gafp$ acts on the inverse limit
\begin{align}
    \mathbf{Sh}_{K_p}\gx:=\varprojlim_{K^p} \mathbf{Sh}_{K^pK_p}\gx.
\end{align}
If it is clear from context, we will omit $\gx$ from the notation. By \cite[Proposition 4.1.2]{PappasRapoportShtukas} and \cite[Section 4.1.1]{DanielsVHKimZhangII}, there are morphisms
\begin{align}
    \mathbf{Sh}_{K}\gx^{\lozenge} \to \shtgmu \otimes_{\spd \mathcal{O}_{E}} \spd E
\end{align}
that are compatible with changing $K^p$. The goal of this section is to investigate the functoriality of this construction in the triple $\gxg$.

\subsubsection{} Consider the category $\shtrp$ whose objects are triples $(\g, \mathsf{X}, \mathcal{G})$, where $\gx$ is a Shimura datum satisfying SV5 and where $\mathcal{G}$ is a quasi-parahoric model of $G$. Morphisms in $\shtrp$ are morphisms $\gx \to \gxp$ that extend (necessarily uniquely) to $\mathcal{G} \to \mathcal{G}'$. For $\gxg \in \shtrp$, we will write $K_p=\mathcal{G}(\zp)$. If we fix an isomorphism $\mathbb{C} \to \qpbar$, then the reflex field $\mathsf{E}$ of each Shimura datum has a natural embedding $\mathsf{E} \to \mathbb{C} \to \qpbar$ with completion $E$. For $L \subset \qpbar$ a finite extension of $\qp$, we will write $\shtrp_L \subset \shtrp$ for the full subcategory of triples such that the reflex field is contained in $L$. There is a strict functor
\begin{align}
    \mathbf{Sh}:\shtrp_L &\to \perfpairsrat_L \\
    (\g, \mathsf{X}, \mathcal{G}) &\mapsto \mathbf{Sh}_{K_p}\gx_{L}^{\lozenge},
\end{align}
and similarly a strict functor sending
\begin{align}
    \mathbf{Sh}_{\infty}:(\g, \mathsf{X}, \mathcal{G}) &\mapsto \varprojlim_{U_p} \mathbf{Sh}_{U_p}\gx_{L}^{\lozenge}.
\end{align}
We also get a strict functor sending $(\g, \mathsf{X}, \mathcal{G}) \mapsto \underline{K_p}$ and thus a weak functor $\mathbb{B}$ sending $(\g, \mathsf{X}, \mathcal{G}) \mapsto \mathbb{B} \underline{K_p}$, see Section \ref{Sec:AppendixQuotientStackFixedPoints}. The natural map 
\begin{align}
    \varprojlim_{U_p} \mathbf{Sh}_{U_p}\gx^{\lozenge} \to  \mathbf{Sh}_{K_p}\gx^{\lozenge}
\end{align}
is a $\underline{K}_p$-torsor and corresponds to a weak natural transformation (see Lemma \ref{Lem:FunctorialityQuotientStack})
\begin{align}
    \mathbf{Sh} \to \mathbb{B}.
\end{align}
\subsubsection{} There is a natural functor $\shtrp_L \to \shtpr_L$ sending $\gxg \mapsto (\mathcal{G}, \mu)$, where $\mu$ is the $G(\qpbar)$-conjugacy class of cocharacters of $G$ corresponding to our fixed isomorphism $\mathbb{C} \to \qpbar$. Thus we can think of $\operatorname{Sht}:\shtpr_L \to \perfpairsrat_L$ as a weak functor on $\shtrp_L$. We let $(\mathbb{B} \underline{K_p})^{\mathrm{\mathrm{dR}}}$ denote the sub (pre-)stack of $(\mathbb{B} \underline{K_p})$ consisting of $K_p$-torsors that are de-Rham in the sense of \cite[Definition 2.6.1]{PappasRapoportShtukas}.
\begin{Lem} \label{Lem:DeRham}
The morphism $\mathbf{Sh}_{K_p}\gx_{L}^{\lozenge} \to \mathbb{B} \underline{K_p}$ factors through $(\mathbb{B} \underline{K_p})^{\mathrm{\mathrm{dR}}} \subset (\mathbb{B} \underline{K_p})$.
\end{Lem}
\begin{proof}
This follows from the main results of \cite{LiuZhu}, as explained in the proof of \cite[Proposition 4.1.2]{PappasRapoportShtukas}.
\end{proof}

\begin{Prop} \label{Prop:FunctorialityShimuraVarietiesToShtukas}
There is a weak natural transformation $\mathbf{Sh} \to \operatorname{Sht}$ factoring on objects through $\operatorname{Sht}_{\mathcal{G},\mu,L}$, which recovers the construction of \cite[Section 2.6]{PappasRapoportShtukas} for each triple $\gxg$.
\end{Prop}
\begin{proof}
We are going to chain together a number of weak functors below. In the rest of this proof we let $S=\mathbf{Sh}_{K_p}\gx_{L}^{\lozenge}$. \smallskip 

There is a functor from $(\mathbb{B} \underline{K_p})(S)$ to the groupoid of exact tensor functors $\operatorname{Rep}_{\zp} \mathcal{G} \to \{\zp \text{ local systems on } S\}$, see \cite[Section 2.6.2]{PappasRapoportShtukas}. It is given by sending a $\underline{K_p}$-torsor $\mathcal{P} \to S$ to the tensor functor sending a representation $\rho:\mathcal{G} \to \operatorname{GL}(\Lambda)$, where $\Lambda$ is a finite free $\zp$-module, to the $\zp$-local system
    \begin{align}
        \mathcal{P} \times^{\rho} \underline{\Lambda}.
    \end{align}
    This is weakly functorial in $f:\mathcal{G} \to \mathcal{G}'$ by using the identifications
    \begin{align}
        \left(\mathcal{P} \times^{f} \mathcal{G'}\right) \times^{\rho'} \underline{\Lambda} \to \mathcal{P} \times^{\rho} \underline{\Lambda}  
    \end{align}
    of Section \ref{Sec:AppendixQuotientStacks}. Here $\rho':\mathcal{G}' \to \operatorname{GL}(\Lambda)$ is a representation whose composition with $f$ gives the representation $\rho$ of $\mathcal{G}$. Per definition of de-Rham local systems, this equivalence identifies $(\mathbb{B} \underline{K_p})^{\mathrm{\mathrm{dR}}}(S)$ with the groupoid of exact tensor functors $\operatorname{Rep}_{\zp} \mathcal{G} \to \{\zp \text{ de-Rham local systems on } S\}$. \smallskip 

Pappas and Rapoport construct in \cite[Proposition 2.6.3, Definition 2.6.4]{PappasRapoportShtukas} an equivalence of categories
\begin{align}
    \{\zp \text{ de-Rham local systems on } S\} \to \{\text{Shtukas on } S\}
\end{align}
compatible with tensor products. Composing with this equivalence gives us a functor from the groupoid of exact tensor functors $$\operatorname{Rep}_{\zp} \mathcal{G} \to \{\zp \text{ de-Rham local systems on } S\}$$ to the groupoid of exact tensor functors $\operatorname{Rep}_{\zp} \mathcal{G} \to \{\text{Shtukas on } S\}$, which is weakly functorial in $f:\mathcal{G} \to \mathcal{G}'$. \smallskip

Last, we claim that there is an equivalence of categories between the groupoid of exact tensor functors $\operatorname{Rep}_{\zp} \mathcal{G} \to \{\text{shtukas over } S\}$ to $\shtg(S)$, which is weakly functorial in $\mathcal{G} \to \mathcal{G}'$. The equivalence is explained in \cite[Section 5.1.1]{ImaiKatoYoucis}, and the weak functoriality can be proved in the same way as in the first paragraph of this proof.
\end{proof}

\subsection{The Pappas--Rapoport axioms} Before we can state the Pappas--Rapoport axioms we introduce some notation. Given a formal scheme $\mathcal{X} \to \spf \zpbr$ that is formally locally of finite type, and a closed point $x \in \mathcal{X}(\ovfp)$, we denote by $\widehat{\mathcal{X}}_{/x}$, the \emph{formal neighborhood} of $\mathcal{X}$ at $x$. There is a similar construction for certain v-sheaves over $\spd \zpbr$, see \cite[Section 3.3.1]{PappasRapoportShtukas}. We will only use this for the tautological point $x_0 \in \mintgb(\spd \ovfp)$. \smallskip 

Pappas and Rapoport conjecture, see \cite[Conjecture 4.2.2]{PappasRapoportShtukas} (and \cite[Conjecture 4.1.4]{DanielsVHKimZhangII} in the quasi-parahoric case), that there are flat normal integral models $\scrs_{K}\gx \to \spec \mathcal{O}_{E}$ of $\mathbf{Sh}_{K}\gx$ for sufficiently small $K^p$, together with forgetful maps $\scrs_{K'^pK_p}\gx \to \scrs_{K}$ over $\spec \mathcal{O}_{E}$ extending the ones on the generic fiber, such that:
\begin{enumerate}[label=\alph*)]
    \item The forgetful morphisms are finite \'etale and $\gafp$ acts on the inverse limit $$\scrs_{K_p}\gx:=\varprojlim_{K^p} \scrs_{K^pK_p}\gx.$$ Moreover, for every discrete valuation ring $R$ over $\mathcal{O}_{E}$ with characteristic $(0,p)$ the natural map
    \begin{align}
        \scrs_{K_p}\gx(R) \to \mathbf{Sh}_{K_p}\gx(R[1/p])
    \end{align}
    is a bijection. 
    
    \item For all sufficiently small $K^p$ the map $\mathbf{Sh}_{K}\gx^{\lozenge} \to \shtgmu \otimes_{\spd \mathcal{O}_{E}} \spd E$ extends to a map 
    \begin{align} \label{Eq:ShtukaFunctor}
        \scrs_{K}\gx^{\lozenge/} \to \shtgmu.
    \end{align}
    
    \item For all sufficiently small $K^p$ and every $x \in \scrs_{K}\gx(\ovfp)$ with induced $b_x:\spd \ovfp \to \shtgmu(\ovfp)$, there is an isomorphism 
    \begin{align}
        \Theta_x:\mintgbxnaught \to \left(\widehat{\scrs_{K,/x}}\right)^{\lozenge}, 
    \end{align}
    such that: The pullback under $\Theta_x$ of the $\mathcal{G}$-shtuka over $\scrs_{K}^{\lozenge/}$ coming from \eqref{Eq:ShtukaFunctor}, is isomorphic to the tautological $\mathcal{G}$-shtuka over $\mintgbxnaught$.
\end{enumerate}
By \cite[Remark 4.1.5]{DanielsVHKimZhangII}, the morphism in b) automatically factors through the open and closed substack $\shtgmuone \subset \shtgmu$.
\subsubsection{} \label{Sec:AxiomsDiscussion}
We will refer to integral models satisfying their conjecture as integral models satisfying the Pappas--Rapoport axioms or as \emph{integral canonical models}. Pappas and Rapoport prove, see \cite[Theorem 4.2.4]{PappasRapoportShtukas}, that such integral models are unique if they exist (see \cite[Theorem 4.1.8]{DanielsVHKimZhangII} for the quasi-parahoric case). These integral models are moreover functorial in the triple $\gxg$, see \cite[Corollary 4.1.10]{DanielsVHKimZhangII}. By \cite[Theorem I]{DanielsVHKimZhangII}, integral models satisfying their conjecture exist if $\gx$ is of Hodge type.

Let us denote by $\shtrp_{L,\operatorname{SV5},\mathrm{PR}}$ the category of triples $\gxg$ such that $\mathsf{E} \to  \mathbb{C} \to \qpbar$ factors through $L$, such that $\gx$ satisfies SV5 and such that \cite[Conjecture 4.1.4]{DanielsVHKimZhangII} holds for $\gxg$. Then as explained above, there is a strict functor $\mathscr{S}^{\lozenge/}:\shtrp_{L,\operatorname{SV5},\mathrm{PR}} \to \perfpairs_{\mathcal{O}_L}$ sending $\gxg$ to $\scrs_{K_p}\gx^{\lozenge/}_{\mathcal{O}_L}$. We can now state an important corollary to Proposition \ref{Prop:FunctorialityShimuraVarietiesToShtukas}. 
\begin{Cor} \label{Cor:FunctorialityShimuraVarietiesToShtukasII}
There is a weak natural transformation $\mathscr{S} \to \operatorname{Sht}$ of weak functors $\shtrp_{L,\operatorname{SV5},\mathrm{PR}} \to \perfpairs_{\mathcal{O}_L}$, with underlying $1$-morphisms the morphisms
\begin{align} 
        \scrs_{K_p}^{\lozenge/}\gx \to \shtgmu
    \end{align}
guaranteed to exist by axiom b).
\end{Cor}
\begin{proof}
Restricting to the generic fiber gives a fully faithful functor
\begin{align}
    \operatorname{Hom}(\scrs_{K_p}\gx^{\lozenge/}, \shtgmu) \to \operatorname{Hom}(\mathbf{Sh}_{K_p}\gx^{\lozenge}, \shtgmu),
\end{align}
by \cite[Corollary 2.7.10]{PappasRapoportShtukas}. To turn the collection of $1$-morphisms specified in the lemma into a weak functor, we have to specify certain coherence isomorphisms satisfying certain properties, see Definition \ref{Def:WeakFunctor}. By the fully faithfulness, it suffices to do this on the generic fiber, where the result is Proposition \ref{Prop:FunctorialityShimuraVarietiesToShtukas}.
\end{proof}

\subsubsection{} Now we prove some uniqueness results for the maps in axiom c). 

\begin{Lem} \label{Lem:UniquenessShtukas}
Let $R$ be a complete Noetherian local $\zpbr$-algebra equipped with a morphism of $\zpbr$-algebras $x:R \to \ovfp$ and let $(\spf R)^{\lozenge} \to \shtgmu$ be a morphism. Let $b_x:\spd \ovfp \to \shtgmu$ be the $\mathcal{G}$-shtuka over $\spd \ovfp$ induced by $x$. Then there is a unique morphism
\begin{align}
    \Psi_{b_{x}}:(\spf R)^{\lozenge} \to \mintgbxnaught
\end{align}
such that $\spd \ovfp \to (\spf R)^{\lozenge} \xrightarrow{\beta} \to \mintgbxnaught \shtgmu$ is equal to $b_x$. 
\end{Lem}
\begin{proof}
The uniqueness is a direct consequence of \cite[Proposition 4.2.5]{PappasRapoportShtukas}, and the existence is \cite[Proposition 4.7.1]{PappasRapoportShtukas}. [Here we note that the proof of the latter also works when $\mathcal{G}$ is quasi-parahoric by using \cite[Proposition 4.2.1]{PappasRapoportRZ} to transport the statement about formal neighborhoods to the quasi-parahoric case.]
\end{proof}
\begin{Cor} \label{Cor:UniquenessShtukasII}
Let $\scrs_{K}=\scrs_K\gx$ be an integral model satisfying axioms a),b),c) above and let $x \in \scrs_{K}(\ovfp)$. Let $b_x: \spd \ovfp \to \shtgmu$ denote the induced $\mathcal{G}$-shtuka over $\spd \ovfp$. Then the map
\begin{align}
  \left(\widehat{\scrs_{K,/x}}\right)^{\lozenge} \to \mintgbpxnaught  
\end{align}
of Lemma \ref{Lem:UniquenessShtukas} is an isomorphism 
\end{Cor}
\begin{proof}
This is a straightforward consequence of axiom c) and the uniqueness proved in Lemma \ref{Lem:UniquenessShtukas}.
\end{proof}
\subsection{Local model diagrams} \label{Sec:LocalModelDiagrams} The purpose of this section is to discuss local model diagrams for the integral models of Shimura varieties constructed by Kisin--Zhou. In particular, we want to show that these local model diagrams are $\Gamma$-equivariant. 

\subsubsection{} Let $\gx$ be a Shimura datum of Hodge type with reflex field $\mathsf{E}$, let $p > 2$ be a prime such that $p$ is coprime to the order of $\pi_1(\gder)$ and let $G = \mathsf{G} \otimes \qp$ as usual. 
\begin{Hyp} \label{Hyp:InducedRamification}
    Either $G$ splits over a tamely ramified extension, or there is a Hodge embedding $\iota:\gx \to \gvx$ such that the identity component $\mathsf{G}^{\mathrm{Sp}}$ of $\gx \cap \mathsf{Sp}_V$ is isomorphic to
\begin{align}
  \prod_{i=1}^{s} \operatorname{Res}_{\mathsf{K}_i/\mathbb{Q}} \mathsf{H}_i^{\mathrm{Sp}},  
\end{align}
where $\mathsf{K}_1, \cdots, \mathsf{K}_s$ are totally real fields and where $\mathsf{H}_i^{\mathrm{Sp}}$ is a connected reductive group over $\mathsf{K}_i$ whose base change to each $p$-adic place of $\mathsf{K}_i$ is tamely ramified. 
\end{Hyp}
\subsubsection{} Let $\mathcal{G} := \mathcal{G}_x$ be the stabilizer Bruhat--Tits  group scheme (also called stabilizer quasi-parahoric, see \cite[Section 2.2]{PappasRapoportRZ}) associated to a point $x$ in the Bruhat--Tits building of $G$. Let $v$ be a place of $\mathsf{E}$ above $p$ and let $E$ be the $v$-adic completion of $\mathsf{E}$. 

Assume that Hypothesis \ref{Hyp:InducedRamification} holds and let $\iota$ be the Hodge embedding guaranteed to exist by that hypothesis. After possibly replacing $\iota$ by another Hodge embedding and $(V,\psi)$ by another symplectic space over $\mathbb{Q}$, we may assume that: The Hodge embedding $\iota$ is good for $(\mathsf{G}, \mathsf{X}, \mathcal{G})$ with respect to some lattice $\Lambda \subset V \otimes \qp$ in the sense of \cite[Section 5.1.2]{KisinZhou}, and $\gx$ satisfies Hypothesis \ref{Hyp:InducedRamification}. Indeed, this follows from \cite[Lemma 5.1.3]{KisinZhou}, noting that the new Hodge embedding $\iota'$ constructed from $\iota$ in the proof of \cite[Proposition 3.3.18]{KisinZhou} still satisfies the condition above in terms of $\mathsf{G}^{\mathrm{SP}}$. Then $\mathcal{G}$ embeds into the parahoric group scheme $\mathcal{G}_V$ of $G_{V}$ that is the stabilizer of the lattice $\Lambda$.

\subsubsection{} Let $\f$ be a Galois totally real field extension of $\mathbb{Q}$ unramified at $p$, let $\Gamma$ be its Galois group and let $F=\f \otimes \qp$. Let $\gx \to \hyo \to \hy$ and $\iota_{H_1}:\hyo \to \gwx$ be as in Section \ref{Sec:FixedPointsHodge}, where we recall that $W= V \otimes_{\mathbb{Q}} \f$. Let $\mathcal{H}_1$ denote the Bruhat--Tits stabilizer group scheme associated to the image of $x$ under the morphism of buildings $B(G, \mathbb{Q}_p) \to B(H_1, \qp)$, and let $\mathcal{H}=\operatorname{Res}_{\mathcal{O}_F/\zp} \mathcal{G}_{\mathcal{O}_F}$ be the Bruhat--Tits stabilizer scheme associated to the image of $x$ in $B(H,\qp)$. 

Recall from \cite[Section 2.4.8]{KisinZhou} that there is a natural map $\beta:\mathcal{H}_1 \to \mathcal{H}$.
\begin{Lem} \label{Lem:RSmoothClosedImmersion}
If Hypothesis \ref{Hyp:InducedRamification} holds, then the natural map $\mathcal{H}_1 \to \mathcal{H}$ is a closed immersion.
\end{Lem}
\begin{proof}
    By \cite[Proposition 2.4.9]{KisinZhou}, it suffices to show that the centraliser $T_1$ of a maximal $\qpbr$-split torus in $H_1$ is an $R$-smooth torus (in the sense of \cite[Section 5.1.2]{KisinZhou}). This is automatic when $G$ is tamely ramified, see \cite[Proposition 2.4.6]{KisinZhou}. Otherwise, we observe that
\begin{align}
    H_1^{\mathrm{Sp}}=\operatorname{Res}_{\f/\mathbb{Q}} \left(\operatorname{Res}_{\mathsf{K}_i/\mathbb{Q}} \mathsf{H}_i^{\mathrm{Sp}}\right)_{\f}
\end{align}
and this implies as in the proof of \cite[Proposition 5.2.6]{KisinZhou} that $T_1$ is $R$-smooth.
\end{proof}
\subsubsection{} We recall the following commutative diagram of closed embeddings of Shimura data from Section \ref{Sec:FixedPointsHodge}
\begin{equation}
    \begin{tikzcd}
        \gx \arrow{r} \arrow{d}{\iota} & (\ho, \mathsf{Y}_1) \arrow{d} \arrow{dr}{\iota_{H_1}} \\
        \gvx \arrow{r} & \gvfx\arrow{r} & \gwx.
    \end{tikzcd}
\end{equation}
Let $\mathcal{G}_{W}$ be the parahoric group scheme of $G_W=\g_{W} \otimes \qp$ that is the stabilizer of the lattice $\Lambda_W=\Lambda \otimes_{\zp} \mathcal{O}_F$ and let $\mathcal{GL}_{W}$ be the reductive integral model of $\operatorname{GL}_W$ corresponding to $\Lambda_W$. We will also consider the parahoric model $\mathcal{G}_{V,F}$ of $\gvf$ corresponding to $\Lambda \otimes \mathcal{O}_F=:\Lambda_{F}$ (thought of as an $\mathcal{O}_F$-module). We will write $M_{p,F}=\mathcal{G}_{V,F}(\zp)$ and $M_{p,W}=\mathcal{G}_{W}(\zp)$.
\begin{Lem} \label{Lem:GoodHodgeEmbedding}
    The morphism $\iota_{H_1}$ is a good Hodge embedding with respect to the lattice $\Lambda_W = \Lambda \otimes_{\zp} \mathcal{O}_F \subset W \otimes \qp$. 
\end{Lem}
\begin{proof}
    We need to check that the morphism $\iota_{H_1}$ satisfies the following three properties
    \begin{enumerate}[label=(\roman*)]
        \item $\iota_{H_1}(H_1)$ contains the scalars.
        \item $\iota_{H_1}$ extends to a closed immersion $\iota_{H_1}:\mathcal{H}_1 \to \mathcal{GL}_{W}$.
        \item The natural map of local models $\locmodhomu \to \mathbb{M}_{\mathcal{GL}_{W},\mu}$ induced by $\mathcal{H}_1 \to \mathcal{GL}_W$ is a closed immersion.\footnote{This condition is equivalent to condition (3) of \cite[Definition 3.3.15]{KisinZhou} because there is a unique map extending the natural map on the generic fiber. Moreover, the local models used by \cite{KisinZhou} agree with ours because theirs also satisfy the Scholze--Weinstein conjecture (which means that they have the correct associated v-sheaf), see \cite[Proposition 3.3.10]{KisinZhou}.}
    \end{enumerate}
    We know that $\iota_{H_1}(H_1)$ contains the scalars because $\iota(\g)$ contains the scalars of $G_V$ and the image of the natural map $G_{V, F} \to G_W$ contains the scalars. Property (ii) is a consequence of Lemma \ref{Lem:RSmoothClosedImmersion} together with the fact that $\mathcal{H}\to \mathcal{GL}_W$ is a closed immersion because it is the restriction of scalars along $\spec \mathcal{O}_F \to \spec \zp$ of (the base change to $\mathcal{O}_F$ of) the closed immersion $\mathcal{G} \to \mathcal{GL}_V$. Property (iii) follows because under our assumptions we have functorial isomorphism (see Lemma \ref{Lem:IdentificationLocalModelsI})
    \begin{align}
        \mathbb{M}_{\mathcal{GL}_{W},\mu} \otimes \mathcal{O}_F&\simeq \prod_{i=1}^d   \mathbb{M}_{\mathcal{GL}_{V},\mu} \otimes \mathcal{O}_F \\
        \locmodhomu \otimes \mathcal{O}_F \simeq \locmodhmu \otimes \mathcal{O}_F &\simeq \prod_{i=1}^d   \mathbb{M}_{\mathcal{G},\mu} \otimes \mathcal{O}_F
    \end{align}
    and property (iii) holds for $\iota:\mathcal{G} \to \mathcal{G}_V$ by assumption. 
\end{proof}

\subsubsection{} We let $K_p=\mathcal{H}(\zp)$ which gives $K_{p,1}=\mathcal{H}_1(\zp)$ and $K_{p}^{\Gamma}=\mathcal{G}(\zp)$. We recall the construction of $\scrs_{K_{p,1}}\hyo$ and $\scrs_{K_{p}^{\Gamma}}\gx$ from \cite[Section 5.1.4]{KisinZhou}, and note that they come with natural maps
\begin{align}
    \iota:\scrs_{K_{p}^{\Gamma}}\gx &\to \scrs_{M_p}\gvx \\
    \iota_{H_1}:\scrs_{K_{p,1}}\hyo &\to \scrs_{M_{p,W}}\gwx,
\end{align}
where $M_p=\mathcal{G}_V(\zp)$ and $M_{p,W}=\mathcal{G}_W(\zp)$. The targets of these morphisms have a moduli interpretation in terms of (weakly) polarized abelian schemes $(A,\lambda)$ up to prime-to-$p$ isogeny with prime-to-$p$ level structures. By \cite[Proposition 5.2.2]{KisinZhou}, the $\mathcal{O}_{E}$ scheme $\scrs_{K_{p}^{\Gamma}}\gx$ fits in a local model diagram
\[\xymatrix{ &\widetilde{\scrs}_{K_{p}^{\Gamma}}\gx\ar[dr]^\pi\ar[dl]^q&\\
\scrs_{K_{p}^{\Gamma}}\gx & &\locmodgmu,}\]
where $q$ is a $\mathcal{G}$-torsor and $\pi$ is pro-smooth of relative dimension $\dim G$. We have a similar local model diagram for $\scrs_{K_{p,1}}\hyo$ given by
\[\xymatrix{ &\widetilde{\scrs}_{K_{p,1}}\hyo\ar[dr]^{\pi_1}\ar[dl]^{q_1} & \\
\scrs_{K_{p,1}}\hyo & &\locmodhmu,}\] where $q_1$ is an $\mathcal{H}_1$-torsor and $\pi_1$ is pro-smooth of relative dimension $\dim H_1$.
\begin{Rem} \label{Rem:EntirelyClassical}
    The map $\pi$ has an entirely classical description over the generic fiber, see \cite[Proposition 4.2.26]{KisinPappas}. The torsor of trivializations of the first de-Rham cohomology of the universal abelian variety over $\scrs_{M_p}\gvx$ has a natural reduction to a $\mathcal{G}$-bundle $\mathcal{V}_{\mathrm{dR}}$. The morphism $\pi$ is given by assigning to a trivialization $\phi$ of the first de-Rham cohomology, the pre-image $\phi^{-1}(\mathcal{F}^1)$ of the nontrivial step of the Hodge filtration.
\end{Rem}

\subsubsection{} We require the following results on the local model diagram.

\begin{Prop} \label{Prop:GammaEquivarianceLocalModelDiagramHodge}
\begin{enumerate}

\item There exists a $\Gamma$-action on $\widetilde{\scrs}_{K_{p,1}}\hyo$ extending that on $\scrs_{K_{p,1}}\hyo$, which fits in a commutative diagram
\begin{equation}
    \begin{tikzcd}
        & \widetilde{\scrs}_{K_{p,1}}\hyo  \arrow[dl,"q_1"] \arrow{r}{\widetilde{\iota}_{H_1}} & \widetilde{\scrs}_{M_{p,F}}\gvfx \arrow{dr}{\pi_{F, V}} \arrow[bend left=20, dl, "q_{F,V}", densely dotted] \\
         \scrs_{K_{p,1}}\hyo \arrow{r}{\iota_{H_1}} & \scrs_{M_{p,F}}\gvfx & \locmodhmu \arrow[from = ul,crossing over ,"\pi_1"] \arrow{r}{\iota_{H_1, p}} & \locmodgvfmu,
    \end{tikzcd}
\end{equation}
and a $\Gamma$-action on $ \widetilde{\scrs}_{M_{p,F}}\gvfx$ such that $\pi_1, \pi_{F, V}, q_1, q_{F, V}, \iota_{H_1, p}, \iota_{H_1}$ are $\Gamma$-equivariant.  
\end{enumerate}
\end{Prop}
\begin{proof}
    To prove the Proposition it suffices to consider the case $\gx = \gvx$, whence $\hyo = \gvfx$. Indeed, we already know that the $\Gamma$-action on $\mathbf{Sh}_{K_1'}\hyo$ extends to $\mathscr{S}_{K_1'}\hyo$ by \cite[Corollary 4.1.10]{DanielsVHKimZhangII}. Thus we are only trying to show that the action extends to the $\mathcal{H}_1$-torsor, and this admits a closed immersion into the pullback of $\widetilde{\scrs}_{M_{p,F}}\gvfx \to \scrs_{M_{p,F}}\gvfx$ along the $\Gamma$-equivariant natural map $\iota_{H_1}$. Moreover, the $\Gamma$-equivariance of the induced maps can be checked on the generic fiber and then over $\mathbb{C}$, where it is clear from uniformization. 
    
    We may also assume that $\Lambda \subset V$ is a self-dual $\mathbb{Z}_{p}$-lattice whence we may assume the same for $\Lambda_W \subset V_F$. To define the action of $\Gamma$ on the torsor $\widetilde{\scrs}_{M_{p,F}}\gvfx$ we may make use of the moduli problem associated to this torsor. Recall that $\scrs_{M_{p,F}}\gvfx$ classifies, over $S$ a $\mathbb{Z}_p$-scheme, tuples $(A, \lambda, i, \overline{\eta})$ where:
    \begin{enumerate}
        \item $h: A \to S$ is an abelian scheme of dimension $[\f:\mathbb{Q}] \cdot \text{dim}(V)/2$ considered up to prime-to-p isogeny. 
        \item $\lambda$ is a $\mathbb{Z}_{(p)}^{\times}$-multiple of a principal polarization.
        \item $i: \mathcal{O}_{\f,(p)} \to \operatorname{End}_S(A)$ is a map of rings such that the Rosati involution associated to $\lambda$ induces the identity on $\mathcal{O}_F$ and such that $i$ satisfies a determinant condition as in \cite[Section 5]{KottwitzPoints}.
        \item $\eta$ is a trivialization $\eta: V^p(A) \xrightarrow{\sim} V_F \otimes \mathbb{A}_f^p$ compatible with the $\mathcal{O}_{\f, (p)}$-action and with the polarization up to a scalar in $\mathbb{A}_{f}^{p, \times}$.
    \end{enumerate}
    The scheme $\widetilde{\scrs}_{M_{p,F}}\gvfx$ classifies quintuples $(A, \lambda, i, \eta, \tau)$ where $(A, \lambda, i, \eta)$ are as above and $\tau: \Lambda^{\vee}_{F} \xrightarrow{\sim} R^1h_* \Omega_{A/S}^*$ is compatible with the symplectic pairing and $\mathcal{O}_F$ actions on the source and target. The $\Gamma$-action on the moduli problem $\widetilde{\scrs_{M_{p,F}}}\gvfx$ sends a tuple $(A, \lambda, i, \eta, \tau)$ to the tuple $(A, \lambda, i \otimes_{\mathcal{O}_{\f,(p)}, \gamma^{-1}} \mathcal{O}_{\f,(p)}, \eta, \gamma(\tau))$, here $\tau: \Lambda^\vee_{F}\otimes \mathcal{O}_S \xrightarrow{\sim} R^1h_*\Omega^*_{A/S}$ and we define $\gamma(\tau)$ to be the composition $\tau \circ (\iota_{\gamma} \otimes \mathcal{O}_S)$ where $\iota_{\gamma}: \Lambda^{\vee}_{F} \otimes_{\mathcal{O}_{\f,(p)}, \gamma^{-1}} \mathcal{O}_{F} \xrightarrow{\sim} \Lambda^{\vee}_{F}$ is the natural isomorphism arising from the trivial descent datum on $\Lambda_{V, F} \xrightarrow{\sim} \Lambda \otimes \mathcal{O}_F$. By definition, this lifts the $\Gamma$ action on the moduli problem $\scrs_{M_{p,F}}\gvfx$. One is left to check that the map $\pi_{F, V}$ is $\Gamma$-equivariant, and that this moduli interpretation for the $\Gamma$ action agrees with the naive description of the $\Gamma$-action on the Hilbert--Siegel modular varieties we consider by uniformization. We check the former, the latter is left to the reader. By flatness and normality of $\widetilde{\scrs}_{M_{p,F}}\gvfx$ and its local model it suffices to check this over the generic fiber, but here we see that 
    \[
        \pi_{F, V, \mathbb{Q}_p}(\gamma(A, \lambda, i, \eta, \tau)) = \iota_{\gamma}^{-1}\tau^{-1}(\mathcal{F}^1(R^1h_{*}\Omega^*_{A/S}))
    \]
    which was exactly our definition of the $\Gamma$-action on the local model $\mathbb{M}_{\mathcal{G}_{V,F}, \mu, \mathbb{Q}_p}$, see Remark \ref{Rem:EntirelyClassical}. \end{proof}

 \subsubsection{} Finally, we state a Proposition which will be used in the proof of part (3) of Theorem \ref{Thm:FixedpointIntegralIntroduction}. The assumption that $\mathcal{H}$ is hyperspecial can probably be relaxed with some effort; this assumption is however present for different reasons in part (3) of Theorem \ref{Thm:FixedpointIntegralIntroduction}. 
\begin{Prop} \label{Prop:EquivariantTangentSpaces}
Assume that $\mathcal{H}$ is the parahoric group scheme associated to a hyperspecial vertex in the Bruhat-Tits building of $H$. Let $x: \spec \ovfp \to \scrs_{K_1}\hyo$ be a point which in the image of the inclusion $\scrs\gx \to \scrs \hyo$. Then there exists a $\Gamma$-fixed point $y: \spec \ovfp \to \locmodhmu$ and a $\Gamma$-equivariant isomorphism of tangent spaces \[\Theta: T_x(\scrs_{K_1}\hyo) \to T_y(\locmodhmu).\]
\end{Prop}
\begin{proof}
We actually prove slightly more. We will borrow freely the notation of Proposition \ref{Prop:GammaEquivarianceLocalModelDiagramHodge}. Let $x$ be as above and let $R_x$ denote the complete local ring of $\scrs\hyo$ at $x$. We let $S_x$ denote the versal deformation ring of the p-divisible group with $\mathcal{O}_F$-action living over the point $x$. By the argument of \cite[Proposition 2.3.5]{KisinModels}, the natural morphism $S_x \to R_x$ factors through a surjection onto the complete local ring of some point $y \in \locmodhomu$, we denote this ring by $S_y$. Let $x'$ denote $\iota_{H_1}(x)$, let $y' = \iota_{H_1, p}(y)$ and let $R'_{x'}, S'_{y'}$ the complete local rings of $\scrs_{M_{p, F}}\gvfx$, $\locmodgvfmu$ at these points. By construction of the map $S_x \to R_x$ and by the results of Proposition \ref{Prop:GammaEquivarianceLocalModelDiagramHodge}, one has a commutative diagram of morphisms of local rings
    \[
        \begin{tikzcd}
            S_y \arrow{r} &R_x \\
            S'_{y'}\arrow{u}\arrow{r} &R'_{x'} \arrow{u}
        \end{tikzcd}
    \]
    where the vertical arrows are surjections (for the right hand vertical arrow, this follows from on \cite[Theorem 1.1.1]{Xunormalization}), and all maps are $\Gamma$-equivariant. Since the maps on Zariski tangent spaces induced by the vertical arrows are injective, it thus suffices to deduce $\Gamma$-equivariance of the lower horizontal arrow, so we are reduced to the Hilbert--Siegel case.  
    
    Let $\bar{R'}_{x'}, \bar{S'}_{y'}$ denote the special fiber of the deformation rings $R'_{x'}, S'_{y'}$ respectively. The induced map $\bar{S'}_{y'}/\mathfrak{m}_{y'}^p \to \bar{R'}_{x'}/\mathfrak{m}_{x'}^p$ has an explicit moduli theoretic interpretation in terms of Grothendieck--Messing's crystalline deformation theory. We let $\mathcal{A}$ denote the universal polarized abelian scheme with $\mathcal{O}_F$-action over $\spec(\bar{R'}_{x'})$, and let $\mathcal{A}_0$ denote its fiber over the point $\spec(\fpbar) = \spec(\bar{R'}_{x'}/\mathfrak{m}_{x'})$ in what follows. The ring map $\bar{R}_{x'}/\mathfrak{m}_{x'}^p \to \ovfp = \bar{R'}_{x'}/\mathfrak{m}_{x'}$ admits a set of trivial nilpotent divided powers, whence one has a canonical identification of polarized $\mathcal{O}_F$-modules (via the Gauss--Manin connection) $\mathrm{H}^1_{\mathrm{dR}}(\mathcal{A}/(\bar{R'}_{x'}/\mathfrak{m}_{x'}^p)) \xrightarrow{\sim} \mathrm{H}^1_{\mathrm{dR}}(\mathcal{A}_0/\ovfp) \otimes_{\fpbar} \bar{R'}_{x'}/\mathfrak{m}_{x'}^p$, and one has a canonical identification $\mathrm{H}_{\mathrm{cris}}^1(\mathcal{A}_0/(\bar{R'}_{x'}/\mathfrak{m}_{x'}^p)) \xrightarrow{\sim} \mathrm{H}^1_{\mathrm{dR}}(\mathcal{A}/(\bar{R'}_{x'}/\mathfrak{m}_{x'}^p))$, and finally one has the identification $\mathrm{H}^1_{\mathrm{cris}}(\mathcal{A}_0/(\bar{R'}_{x'}/\mathfrak{m}_{x'}^p)) \otimes \ovfp = \mathrm{H}^1_{\mathrm{dR}}(\mathcal{A}/\ovfp)$. Because $x'$ is in the image of $\scrs\gx$, we may fix an $\mathcal{O}_F$-equivariant trivialization of $\mathrm{H}^1_{\mathrm{dR}}(\mathcal{A}_0/\fpbar) \xrightarrow{\sim} \Lambda^{\vee}_{F} \otimes_{\zp} \fpbar$ compatible with the standard symplectic pairing on $\Lambda^{\vee}_{F}$ and the polarization on $\mathrm{H}^1_{\mathrm{dR}}(\mathcal{A}_0/\fpbar)$, the $\Gamma$ action, and such that the image of the Hodge filtration under this trivialization defines the point $y'$ on $\locmodhomu$. Now by the result \cite[Theorem V.1.10]{Messing} the image of the Hodge filtration $F^0\mathrm{H}^1_{\mathrm{dR}}(\mathcal{A}/(\bar{R'}_{x'}/\mathfrak{m}_{x'}^p))$ inside of $\mathrm{H}^1_{\mathrm{cris}}(\mathcal{A}_0/(\bar{R'}_{x'}/\mathfrak{m}_{x'}^p)) \simeq (\Lambda^{\vee}_{F} \otimes_{\zp} \fpbar) \otimes_{\fpbar} (\bar{R'}_{x'}/\mathfrak{m}_{x'}^p)$ gives the isomorphism $\bar{S'}_{y'}/\mathfrak{m}_{y'}^p \to \bar{R'}_{x'}/\mathfrak{m}_{x'}^p$.

    Now the action of $\Gamma$ on the source takes the abelian scheme $\mathcal{A}$ to $\gamma^*\mathcal{A}$. On the other hand as $\mathcal{A}_0$ arises from an abelian variety $\mathcal{A}'_0$ such that $\mathcal{A}_0 = \mathcal{A}'_0 \otimes \mathcal{O}_F$ and thus there exists a canonical isomorphism $\mathrm{H}^1_{\mathrm{cris}}(\mathcal{A}_0/(\bar{R'}_{x'}/\mathfrak{m}_{x'}^p)) \xrightarrow{\sim} \mathrm{H}^1_{\mathrm{cris}}(\mathcal{A}_0'/(\bar{R'}_{x'}/\mathfrak{m}_{x'}^p)) \otimes \mathcal{O}_F$ compatible with the $\mathcal{O}_F$ action and polarization, as well as the $\Gamma$ action. There is a canonical isomorphism $\gamma^*\mathrm{H}^1_{\mathrm{dR}}(\mathcal{A}/(\bar{R'}_{x'}/\mathfrak{m}_{x'}^p)) \xrightarrow{\sim} \mathrm{H}^1_{\mathrm{dR}}(\gamma^*\mathcal{A}/(\bar{R'}_{x'}/\mathfrak{m}_{x'}^p))$ of polarized $\mathcal{O}_F$-modules. Finally the desired $\Gamma$-equivariance follows from scrutinizing the commutative diagram
    \[
        \begin{tikzcd}
            \gamma^*\mathrm{H}^1_{\mathrm{dR}}(\mathcal{A}/(\bar{R'}_{x'}/\mathfrak{m}_{x'}^p))\arrow{d} \arrow{r} &\mathrm{H}^1_{\mathrm{dR}}(\gamma^*\mathcal{A}/(\bar{R'}_{x'}/\mathfrak{m}_{x'}^p)) \arrow{d} &&\mathrm{H}^1_{\mathrm{dR}}(\mathcal{A}/(\bar{R'}_{x'}/\mathfrak{m}_{x'}^p)) \arrow{d} \\
            \gamma^*\mathrm{H}^1_{\mathrm{cris}}(\mathcal{A}_0/(\bar{R'}_{x'}/\mathfrak{m}_{x'}^p)) \arrow{d} \arrow{r}& \mathrm{H}^1_{\mathrm{cris}}(\gamma^*\mathcal{A}_0/(\bar{R'}_{x'}/\mathfrak{m}_{x'}^p)) \arrow{rr}{\mathrm{H}^1_{\mathrm{cris}}(\psi_{\gamma}^{-1})} && \mathrm{H}^1_{\mathrm{cris}}(\mathcal{A}_0/(\bar{R'}_{x'}/\mathfrak{m}_{x'}^p)) \arrow{d} \\
            \gamma^*(\Lambda^{\vee}_{F} \otimes_{\zp} \fpbar) \otimes_{\fpbar} (\bar{R'}_{x'}/\mathfrak{m}_{x'}^p) \arrow{rrr}&&& (\Lambda^{\vee}_{F} \otimes_{\zp} \fpbar \otimes_{\fpbar} (\bar{R'}_{x'}/\mathfrak{m}_{x'}^p).
        \end{tikzcd}
    \]
\end{proof}

\begin{Rem}
    We had originally hoped to prove the $\Gamma$-equivariance in Proposition \ref{Prop:GammaEquivarianceLocalModelDiagramHodge} using the $\Gamma$-equivariance of shtukas and the interpretation of the local model diagram in terms of shtukas, see \cite[Section 4.9.1]{PappasRapoportShtukas}. This does not seem to be possible however, since the scheme-theoretic local model diagram is not uniquely pinned down by the v-sheaf theoretic local model diagram, as discussed in \cite[Section 4.9.1]{PappasRapoportShtukas}. 
\end{Rem}    
\section{Fixed points of integral models of Shimura varieties} \label{Sec:MainTheoremI}
In this section we prove Theorem \ref{Thm:FixedpointIntegralIntroduction} from the introduction.

\subsection{Fixed points of integral models of Shimura varieties of Hodge type} \label{Sec:FixedPointsHodgeIntegral}
Let $\gx$ be a Shimura datum of Hodge type, let $\f$ be a totally real Galois extension of $\mathbb{Q}$ with Galois group $\Gamma$. Let $\gx \to \hyo \to \hy$ be as in Section \ref{Sec:FixedPointsHodge}. We let $\calgcirc$ be a parahoric model of $G$ that is the identity component of a Bruhat--Tits stabilizer group scheme $\mathcal{G}$ corresponding to some point $x$ in the Bruhat--Tits building of $G(\qp)$. We let $\calhcirc \subset \mathcal{H}$ be the corresponding objects for $H$, and we let $K_p=\mathcal{H}(\zp)$ and $K_p^{\circ}=\calhcirc(\zp)$. We let $\mathcal{H}_1$ be the Bruhat--Tits stabilizer group scheme for $H_1$ and we let $\mathcal{H}_1' \subset \mathcal{H}_1$ be the inverse image of $\calhcirc \subset \mathcal{H}$. We let $\mathcal{G}^{\mathrm{ad}}$ be the Bruhat--Tits stabilizer group scheme corresponding to the image of $x$ in the building of $G^{\mathrm{ad}}$, and similarly we define $\mathcal{H}^{\mathrm{ad}}$. We also need the parahoric group schemes $\calgadcirc$ and $\calhadcirc$, and we observe that
\begin{align}
    \calhad \simeq \operatorname{Res}_{\mathcal{O}_F/ \zp} \calgad_{\mathcal{O}_F}, \qquad  \calhadcirc \simeq \operatorname{Res}_{\mathcal{O}_F/ \zp} \calgadcirc_{\mathcal{O}_F}.
\end{align}
We set $K_{p,1}=\mathcal{H}_1(\zp)=K_p \cap H_1(\qp)$ and $K_{p,1}'=\mathcal{H}_1(\zp)=K_p' \cap H_1(\qp)$. We also have $K_p^{\Gamma}=\mathcal{G}(\zp)$ and $K_p^{\circ, \Gamma}=\mathcal{G}(\zp)$. 
\begin{Rem}
    The notation just introduced conflicts slightly with that of the introduction. The objects denoted by $\mathcal{G}, \mathcal{H}_1$ in the introduction correspond to $\calgcirc, \mathcal{H}_1'$ in this section. 
\end{Rem}

In what follows we let $K^p \subset \hafp$ denote a neat good compact open subgroup such that there is a prime number $\ell \not=p$ coprime to the order of $\Gamma$ such that $K^p=K^{p,\ell} K_{\ell}$ with $K_{\ell}$ a pro-$\ell$ group. For such $K^p$ we have $K_1, K_1', K^{\Gamma}, K^{\circ, \Gamma}$ defined in the obvious way using $K_1^p$ and $K^{p,\Gamma}$. Note that if $\f$ is tamely ramified over $\mathbb{Q}$, then such $K^p$ form a cofinal collection of compact open subgroups by Proposition \ref{Prop:ExistenceCofinalGoodCompactOpens}. Note that there are integral models of levels $K_1, K_1', K^{\Gamma}, K^{\circ, \Gamma}$ satisfying the Pappas--Rapoport axioms. These agree, by construction, with the ones of \cite{KisinZhou} that we used in Section \ref{Sec:LocalModelDiagrams}. The following theorem generalizes Theorem \ref{Thm:FixedpointIntegralIntroduction}.
\begin{Thm} \label{Thm:FixedPointsCanonicalIntegralModelsHodge}
Assume that $p$ is unramified in $\f$, that $p>2$ and that $\Sha^1(\mathbb{Q}, \mathsf{G}) \to \Sha^1(\f, \g)$. 
\begin{enumerate}
    \item The natural map
\begin{align} \label{Eq:NaturalMapMainTheorem}
    \scrs_{K^{\circ,\Gamma}}\gx \to \scrs_{K_{1}'}\hyo^{\Gamma}.
\end{align}
is a universal homeomorphism. 

    \item If $p$ is coprime to $|\Gamma| \cdot |\pi_1(\gder)|$ and $\gx$ satisfies Hypothesis \ref{Hyp:InducedRamification}, then the natural map of \eqref{Eq:NaturalMapMainTheorem} is an isomorphism.
    
    \item If $K_p$ is hyperspecial, then the natural map of \eqref{Eq:NaturalMapMainTheorem} is an isomorphism.
\end{enumerate}
\end{Thm} 
\begin{proof}
\textbf{Part (1):} We first observe that the natural map $\mathcal{H}_1' \to \calhad$ factors through $\calhadcirc$ because $\mathcal{H}_1' \to \calhad$ factors through $\mathcal{H}$, and $\mathcal{H}_1'$ maps to $\mathcal{H}^{\circ}$. The corresponding map $\scrs_{K_1'}\hyo^{\lozenge/} \to \shthopmu$ is $\Gamma$-equivariant in the $2$-categorical sense by Corollary \ref{Cor:FunctorialityShimuraVarietiesToShtukasII}, and so is its composition with $\shthopmu \to \shthadcircmu$. 

By Lemma \ref{Lem:HomotopyFixedPointsShtukas}, using the fact that $p$ is unramified in $\f$, there is an isomorphism
\begin{align}
    \shtgadcircmu \to \left(\shthadcircmu\right)^{h \Gamma}.
\end{align}
This induces a map
\begin{align}
    \scrs_{K_1'}\hyo^{\lozenge/, \Gamma} \to \shtgadcircmu,
\end{align}
compatible with the maps $\scrs_{K^{\circ,\Gamma}}\gx^{\lozenge/} \to \shtgcircmu \to \shtgadcircmu$ and $\scrs_{K^{\circ,\Gamma}}\gx^{\lozenge/} \to \scrs_{K_1'}\hyo^{\lozenge/, \Gamma}$. For $x \in \scrs_{K_1'}\hyo^{\Gamma}(\ovfp)$ we let $b_x:\spd \ovfp \to \shtgadcircmu$ be the corresponding $\calgadcirc$-shtuka. Then by Lemma \ref{Lem:UniquenessShtukas}, there is a morphism
\begin{align}
    \Psi_{b_x}: \left(\widehat{\scrs_{K_1'}}\hyo^{\Gamma} \right)_{/x}^{\lozenge} \to \mintgadcircbxnaught,
\end{align}
and an isomorphism from the pullback under $\Psi_{b_x}$ of the tautological $\calgadcirc$-shtuka on $\mintgadcircbxnaught$ to the universal $\calgadcirc$-shtuka pulled back from $\scrs_{K_1'}\hyo^{\Gamma, \lozenge/}$. If we let $d_x$ denote the induced map $\spd \ovfp \to \shthadcircmu$, then by the uniqueness proved in Lemma \ref{Lem:UniquenessShtukas}, there is a commutative diagram
\begin{equation}
    \begin{tikzcd}
     \left(\widehat{\scrs_{K_1'}}\hyo^{\Gamma} \right)_{/x}^{\lozenge} \arrow{d} \arrow{r}{\Psi_{b_x}} & \mintgadcircbxnaught \arrow{d} \\
   \left(\widehat{\scrs_{K_1'}}\hyo \right)_{/x}^{\lozenge} \arrow{r}{\Psi_{d_x}} & \minthadcircbxnaught, 
    \end{tikzcd}
\end{equation}
where the vertical arrows are the natural closed immersions. We note that the map $\Psi_{d_x} $ is an isomorphism by Corollary \ref{Cor:UniquenessShtukasII} and the fact that (for a lift $b_x': \spd \ovfp \to \shthopmu$ of $b_x$)
\begin{align}
    \mathcal{M}^{\mathrm{int}}_{\mathcal{H}_1', b_x', \mu, /x_0} \to \minthadcircbxnaught
\end{align}
is an isomorphism, see \cite[Theorem 5.1.2]{PappasRapoportRZ}. It is moreover a direct consequence of the uniqueness proved in Lemma \ref{Lem:UniquenessShtukas} that the morphism $\Psi_{d_x}$ is $\Gamma$-equivariant. Since taking $\Gamma$-invariants commutes with taking diamond functors and formal completions, we see using Proposition \ref{Prop:InvariantsRZ} that $\Psi_{b_x}$ is the $\Gamma$-fixed points of $\Psi_{d_x}$. In particular, $\Psi_{b_x}$ is an isomorphism. \smallskip

By \cite[Theorem 2.5.4, Theorem 2.5.5]{PappasRapoportRZ} in combination with \cite[Corollary 1.4]{GleasonLourenco}, using our assumption that $p>2$, there is a flat normal formal scheme $\mathcal{N}$ such that 
\begin{align}
     \mathcal{N}^{\lozenge} \simeq \mintgadcircbxnaught. 
\end{align}
In particular, the unique $\ovfp$-point of $\mintgadcircbxnaught$ (and thus of $\left(\widehat{\scrs_{K_1'}}\hyo^{\Gamma} \right)_{/x}$ ) lifts to a $\spf L$ point for a finite extension $L$ of $\zpbr$. Since this is true for all $x$, this implies that $\scrs_{K_1'}\hyo^{\Gamma}$ is topologically flat over $\mathcal{O}_E$, i.e., that its generic fiber is dense. Let $\scrs_{K_1'}\hyo^{\Gamma,\mathrm{awn}} \to \scrs_{K_1'}\hyo^{\Gamma}$ be the absolute weak normalization, see \cite[Lemma 0EUS]{stacks-project}; this is a universal homeomorphism and so $\scrs_{K_1'}\hyo^{\Gamma,\mathrm{awn}}$ is again topologically flat over $\mathcal{O}_E$. Then by \cite[Lemma 2.13]{AGLR}, the natural map
\begin{align}
     \left(\widehat{\scrs_{K_1'}}\hyo^{\Gamma, \mathrm{awn}} \right)_{/x}^{\lozenge} \to  \left(\widehat{\scrs_{K_1'}}\hyo^{\Gamma} \right)_{/x}^{\lozenge}
\end{align}
is an isomorphism. By the fully faithfulness of the $\lozenge$ functor on absolute weakly normal formal schemes flat separated and topologically of finite type over $\zp$, see \cite[Theorem 2.16]{AGLR}, we find that 
\begin{align}
    \left(\widehat{\scrs_{K_1'}}\hyo^{\Gamma, \mathrm{awn}} \right)_{/x} \simeq \mathcal{N}.
\end{align}
This implies that the complete local rings of $\scrs_{K_1'}\hyo^{\Gamma, \mathrm{awn}}$ at $\ovfp$ points are normal, and thus that the local rings are normal by \cite[Lemma 0FIZ]{stacks-project}. This shows that $\scrs_{K_1'}\hyo^{\Gamma, \mathrm{awn}}$ is normal, because normality (of a quasicompact scheme) can be checked at closed points. By the universal property of the absolute weak normalization, the natural map $\scrs_{K^{\circ, \Gamma}}\gx \to \scrs_{K_1'}\hyo^{\Gamma}$ lifts to a finite map
\begin{align}
    \scrs_{K^{\circ,\Gamma}}\gx \to \scrs_{K_1'}\hyo^{\Gamma,\mathrm{awn}}.
\end{align}
This map is an isomorphism on the generic fiber by Theorem \ref{Thm:FixedPointsHodge} and thus an isomorphism since the target is normal and source and target are flat over $\zp$, see \cite[Lemma 0AB1]{stacks-project}. The theorem now follows from the fact that $\scrs_{K_1'}\hyo^{\Gamma,\mathrm{awn}} \to \scrs_{K_1}\hyo^{\Gamma}$ is a universal homeomorphism. \medskip

\textbf{Part (2):} We assume from now on that $p$ is coprime to the order of $\Gamma$ and to the order of $\pi_1(\gder)$, and that $\gxg$ satisfies Assumption \ref{Hyp:InducedRamification}. It then follows from Proposition \ref{Prop:GammaEquivarianceLocalModelDiagramHodge} that there is a $\Gamma$-equivariant and smooth map
\begin{align}
    \scrs_{K_{1}}\hyo \to \locmodhadmuStack.
\end{align}
It follows from \cite[Proposition 4.2]{Edixhoven}, using the fact that $\Gamma$ is of order prime-to-$p$, that the induced map
\begin{align}
    \scrs_{K_1}\hyo^{\Gamma} \to \locmodhadmuStackGamma.
\end{align}
is also smooth. Recall that Corollary \ref{Cor:HomotopyFixedPointsLocalModelQuotients} tells us that the natural map 
\begin{align}
    \locmodgadmuStack \to \locmodhadmuStackGamma
\end{align}
is an isomorphism. Thus $\scrs_{K_1}\hyo^{\Gamma}$ has a smooth local model diagram to $\locmodgadmu$ and is thus flat over $\zp$ and normal since $\locmodgadmu$ is, see \cite[Corollary 1.4]{GleasonLourenco}. The natural map $\scrs_{K_{1}'}\hyo \to \scrs_{K_{1}}\hyo$ is $\Gamma$-equivariant (by \cite[Corollary 4.1.10]{DanielsVHKimZhangII}) and finite \'etale (by the proof of \cite[Theorem 4.1.12]{DanielsVHKimZhangII}) and thus the induced map on fixed points 
\begin{align}
    \scrs_{K_{1}'}\hyo^{\Gamma} \to \scrs_{K_{1}}\hyo^{\Gamma}
\end{align}
is also finite \'etale. It follows that $\scrs_{K_{1}'}\hyo^{\Gamma}$ is flat over $\zp$ and normal since $ \scrs_{K_{1}}\hyo^{\Gamma}$ is. Now the natural map of the theorem is a finite morphism that induces an isomorphism on generic fibers by Theorem \ref{Thm:FixedPointsHodge}. Since the target is flat over $\zp$ and normal, it follows from \cite[Lemma 0AB1]{stacks-project} that the map is an isomorphism. \medskip

\textbf{Part (3):} We have a diagram
\[
    \scrs_{K^{\circ, \Gamma}}\gx \to \scrs_{K_1'}\hyo^{\Gamma} \to \scrs_{K_1'}\hyo. 
\]
We first observe that the composite is a closed immersion by \cite[Theorem 1.1.1]{Xunormalization} and the second arrow is a closed immersion because $\scrs_{K_1'}\hyo$ is separated; it follows that the first arrow is a closed immersion. From part (1), it moreover follows that the first arrow is a universal homeomorphism. To show that the first map is an isomorphism, it is thus enough to show it induces isomorphisms on complete local rings of $\ovfp$-points, see e.g. \cite[Lemma 0819]{stacks-project}. By Lemma \ref{Lem:CommutativeAlgebra} below, using the fact that $K_p^{\circ, \Gamma}$ is hyperspecial, it suffices to show that the induced map on tangent spaces at $\ovfp$-points is a bijection. Since the map is a closed immersion, it suffices to show that for $x \in \scrs_{K^{\circ, \Gamma}}\gx$ the tangent space of $\scrs_{K^{\circ, \Gamma}}\gx \to \scrs_{K_1'}\hyo^{\Gamma}$ at $x$ has the same dimension as the tangent space of $\scrs_{K_1'}\hyo^{\Gamma}$ at $x$. \smallskip 

We note that $\scrs_{K^{\circ, \Gamma}}\gx$ is smooth, and that the tangent space at $x$ has dimension equal to the dimension of the tangent space of any point in $\locmodgmu=\locmodgadmu$. By Proposition \ref{Prop:EquivariantTangentSpaces}, we may identify the tangent space of $\scrs_{K_1'}\hyo^{\Gamma}$ at $x$ with the $\Gamma$-fixed point of the tangent space of $\locmodhmu$ at a $\Gamma$-fixed point $y$ of $\locmodhmu=\locmodhadmu$. By Proposition \ref{Prop:FixedPointsLocalModels}, we have $y \in \locmodgadmu$, and moreover the $\Gamma$-fixed points of the tangent space gives the tangent space of $\locmodgadmu$ at $y$. This proves the claim about the dimensions of tangent spaces, and so we are done by the argument in the previous paragraph.
\end{proof}

\begin{Lem} \label{Lem:CommutativeAlgebra}
    Let $f:A \to B$ be a surjective local homomorphism of complete local Noetherian $\zpbr$-algebras with $B$ flat, formally smooth, and unramified over $\zpbr$. Suppose that $f$ induces a universal homeomorphism on spectra. If $f$ induces an isomorphism on residue fields and tangent spaces, then it is an isomorphism.
\end{Lem}
\begin{proof}
    The proof is no doubt well known to experts, but we were unable to find a suitable reference. Because $B$ is formally smooth over $\zpbr$ we have that it is a complete local ring of the form $\zpbr[[X_1, \dots, X_n]]$. The map $f: A \to B$ admits a section $g: B \to A$, which we may construct by taking $g(X_i)$ to be any $Y \in A$ such that $f(Y) = X_i$. The morphism $g$ is finite, as both $A, B$ are complete and $g$ induces an isomorphism on residue fields and tangent spaces. Thus $g \circ f$ makes $A$ a finite $A$ algebra, and so by Nakayama's lemma we see that $g \circ f$ is a surjection, as $A$ is Noetherian; it is thus an isomorphism. 
\end{proof}

\begin{Rem}
One can prove a version of Part (1) of Theorem \ref{Thm:FixedPointsCanonicalIntegralModelsHodge} for integral models of Shimura varieties of abelian type $\gx$, using Theorem \ref{Thm:HomotopyFixedPointsShimura} as input on the generic fiber (thus we require the assumption that the $\mathbb{R}$-split rank of $Z_{\g}$ is zero). To do this, one needs integral models satisfying the Pappas--Rapoport axioms, see \cite{DanielsYoucis} for this.\footnote{In an earlier version of this paper, we proved a weaker version of these axioms for the same integral models of Shimura varieties in order to prove a version of Theorem \ref{Thm:FixedPointsCanonicalIntegralModelsHodge}. Our weaker version involved constructing shtukas for $\calgad$, which is easier than constructing $\mathcal{G}$-shtukas, see \cite[Remark 1.1]{DanielsYoucis}.} In fact, the proof of such a theorem is easier than the proof of Theorem \ref{Thm:FixedPointsCanonicalIntegralModelsHodge}, because one can directly work with $\hy$ instead of $\hyo$. 
\end{Rem}
\begin{Rem}
    If $p$ is completely split in $\f$, then the proof of part (1) of Theorem \ref{Thm:FixedPointsCanonicalIntegralModelsHodge} can be adapted to show that the natural map is actually an isomorphism. The key point is that if $p$ is completely split, then one can prove a version of Lemma \ref{Lem:IdentificationLocalModelsI} for the integral local Shimura varieties for $\calhadcirc$, which then implies that Proposition \ref{Prop:InvariantsRZ} holds on the level of formal schemes. We leave the details to the interested reader.
\end{Rem}

\section{Fixed points of Igusa stacks} \label{Sec:IgusaStacks} In this section we discuss the fixed points of the Igusa stacks of \cite{DvHKZIgusaStacks}. We use this to prove Theorem \ref{Thm:FixedPointsVsheavesIntroduction}, see Theorem \ref{Thm:FixedPointsIntegralHodgeRamified}.

\subsection{Igusa stacks} Fix an isomorphism $\mathbb{C} \to \qpbar$. For $\qp \subset L \subset\qpbar$ let us denote by $\shtrp_{L,\operatorname{Hdg}}$ the category whose objects are triples $\gxg$ (with $\mathcal{G}$ quasi-parahoric) such that $\mathsf{E} \to  \mathbb{C} \to \qpbar$ factors through $L$ and such that $\gxg$ is of Hodge type, and whose morphisms are morphisms $f:\gx \to \gxp$ such that $f$ extends (necessarily uniquely) to a morphism $\mathcal{G} \to \mathcal{G}'$. By \cite[Theorem I, Corollary 4.1.10]{DanielsVHKimZhangII}, there is a functor $\scrs^{\diamond}:\shtrp_{L,\operatorname{Hdg}} \to \perfpairs_{\mathcal{O}_L}$ sending $\gxg$ to $\scrs_{K_p}\gx^{\diamond}$, see Section \ref{Sec:AxiomsDiscussion}. Here $K_p=\mathcal{G}(\zp)$ and $\scrs_{K_p}\gx^{\diamond}$ is the diamond associated to the formal scheme $\widehat{\scrs}_{K_p}\gx$, which is the $p$-adic completion of the integral model $\scrs_{K_p}\gx$. 

The following result is a consequence of \cite[Theorem I, Theorem VII]{DvHKZIgusaStacks} as we will explain below. It was originally conjectured by Scholze. 
\begin{Thm} \label{Thm:IgusaStack}
    There is a functor $\igs:\shtrp_{L,\operatorname{Hdg}}\to \perfpairs_{\mathcal{O}_L}$ together with a natural transformation $\scrs^{\diamond} \to \igs$ and a weak natural transformations $\igs \to \bun$, and a strictly commutative diagram of weak natural transformations
\begin{equation}
    \begin{tikzcd}
    \scrs^{\diamond} \arrow{r} \arrow{d} & \operatorname{Sht} \arrow{d} \\
    \igs \arrow{r} & \bun,
    \end{tikzcd}
\end{equation}
such that the resulting $2$-commutative diagrams of stacks on $\perf$
\begin{equation}
    \begin{tikzcd}
    \scrs_{K_p}\gx^{\diamond}\arrow{r} \arrow{d} & \shtgmuone \arrow{d} \\
    \igs\gx \arrow{r} & \operatorname{Bun}_G,
    \end{tikzcd}
\end{equation}
are $2$-Cartesian for all $\gxg \in \shtrp_{L,\operatorname{Hdg}}$, and such that $ \igs\gx \to \operatorname{Bun}_G$ factors through the open substack $\bungmu \to \operatorname{Bun}_G$. 
\end{Thm}
\begin{proof}
    The existence of $ \igs\gx$ is \cite[Theorem VII, Theorem 6.0.1, Remark 6.0.2]{DvHKZIgusaStacks}, the existence of the strict functor $\igs$ and the strict natural transformation $\scrs^{\diamond} \to \igs$ is \cite[Theorem I]{DvHKZIgusaStacks}. The weak natural transformation $\igs \to \operatorname{Bun}$ is constructed at the end of the proof of \cite[Proposition 7.1.6]{DvHKZIgusaStacks}, see the commutative cube in \cite[Equation (7.1.4)]{DvHKZIgusaStacks}. By construction of the Igusa stack, the diagram 
   \begin{equation}
    \begin{tikzcd}
    \scrs_{K_p}\gx^{\diamond}\arrow{r} \arrow{d} & \shtgmu \arrow{d} \\
    \igs\gx \arrow{r} & \operatorname{Bun}_G,
    \end{tikzcd}
\end{equation}
is strictly commutative for all $\gxg$, and thus we see that the result follows.
\end{proof}

\subsection{Fixed points of Igusa stacks} \label{Sec:IgusaStacksFix}
Let the notation be as in Section \ref{Sec:FixedPointsHodgeIntegral}. 
\begin{Thm} \label{Thm:FixedPointsIgusa}
    If $\Sha^1(\mathbb{Q}, \mathsf{G}) \to \Sha^1(\f, \g)$ is injective, then the natural map
    \begin{align}
         \igs \gx \to  \igs\hyo^{\Gamma} \times_{\bun_{H_1}^{h \Gamma}} \bungmu
    \end{align}
    is an isomorphism.
\end{Thm}
To prove the theorem, we need to introduce potentially crystalline loci. 
\subsubsection{The potentially crystalline locus in Hodge type Shimura varieties} 
Recall that there is an open immersion $\mathbf{Sh}_K\gx^{\circ, \mathrm{an}} \subset \mathbf{Sh}_K\gx^{\mathrm{an}}$ of rigid spaces over $E$ called the \emph{potentially crystalline locus}, constructed in \cite{ImaiMieda}, see \cite[Theorem 5.17]{ImaiMieda}. The formation of $\mathbf{Sh}_K\gx^{\circ, \lozenge}$ is compatible with changing $K$, see \cite[Corollary 5.29]{ImaiMieda}, and we will also consider
\begin{align}
    \mathbf{Sh}_{K_p}\gx^{\circ, \lozenge}&=\varprojlim_{K^p \subset \gafp} \mathbf{Sh}_{K^pK_p}\gx^{\circ, \lozenge} \\
    \mathbf{Sh}_{}\gx^{\circ, \lozenge}&=\varprojlim_{K \subset \gafp} \mathbf{Sh}_{K}\gx^{\circ, \lozenge}.
\end{align}
\begin{Lem} \label{Lem:FunctorialityPotentiallyCrystalline}
If $\gx \to \gxp$ is a closed immersion of Shimura data of Hodge type, then for $K_p' \subset G'(\qp)$ containing $K_{p} \subset G(\qp)$ there are equalities of open subdiamonds
\begin{align}
    \mathbf{Sh}_{K_p}\gx^{\circ, \lozenge} &= \mathbf{Sh}_{K_p'}\gxp^{\circ, \lozenge} \times_{\mathbf{Sh}_{K_p'}\gxp^{\lozenge}}\mathbf{Sh}_{K_p}\gx^{\lozenge} \\
    \mathbf{Sh}_{}\gx^{\circ, \lozenge} &= \mathbf{Sh}_{}\gxp^{\circ, \lozenge} \times_{\mathbf{Sh}_{}\gxp^{\lozenge}}\mathbf{Sh}_{}\gx^{\lozenge}.
\end{align}
\end{Lem}
\begin{proof}
This is a direct consequence of \cite[Lemma 2.2]{ImaiMieda}, cf. \cite[Lemma 5.1.6, Lemma 5.1.7]{DvHKZIgusaStacks}.
\end{proof}
The following lemma is \cite[Lemma 5.1.6]{DvHKZIgusaStacks}.
\begin{Lem} \label{Lem:PotentiallyCrystalline}
There is an equality
\begin{align}
    \mathbf{Sh}_K\gx^{\circ, \lozenge} = \scrs_{K_p}\gx^{\diamond} \times_{\spd \mathcal{O}_{E}} \spd E
\end{align}
of open subdiamonds of $\mathbf{Sh}_K\gx^{\lozenge}$.
\end{Lem}
Let $K_{p,1}^{\circ} \subset K_p \cap H_1(\qp)$ be the unique parahoric subgroup, corresponding to the identity component $\mathcal{H}_1^{\circ}$ of $\mathcal{H}_1$. 
\begin{Lem} \label{Lem:FixedPointsGenericRamifiedHodge}
If $\Sha^1(\mathbb{Q}, \mathsf{G}) \to \Sha^1(\f, \g)$ is injective, then the following diagram is $2$-Cartesian 
\begin{equation}
    \begin{tikzcd}
    \mathbf{Sh}_{K_p^{\circ, \Gamma}}\gx^{\circ, \lozenge} \arrow{d} \arrow{r} & \mathbf{Sh}_{K_{p,1}^{\circ}}\hyo^{\circ, \Gamma, \lozenge} \arrow{d} \\
    \operatorname{Sht}_{\mathcal{G},\mu, E} \arrow{r} & \left(\operatorname{Sht}_{\calhocirc,\mu, E}\right)^{h \Gamma}.
    \end{tikzcd}
\end{equation}
\end{Lem}
\begin{proof}
It follows from Lemma \ref{Lem:FixedPointsInfiniteLevelHodge} that the natural map (where the limit runs over $\Gamma$-stable compact open subgroups of $\hafp$)
\begin{align}
    \varprojlim_{K \subset \hafp} \mathbf{Sh}_{K^{\Gamma}}\gx^{\lozenge} \to \varprojlim_{K \subset \hafp} \mathbf{Sh}_{K_1}\hyo^{\lozenge, \Gamma}
\end{align}
is an isomorphism. We then invoke Lemma \ref{Lem:PotentiallyCrystalline} to get an isomorphism
\begin{align}
    \mathbf{Sh}_{}\gx^{\circ, \lozenge} \to \mathbf{Sh}_{}\hyo^{\circ, \lozenge}.
\end{align}
The map $\mathbf{Sh}_{}\hyo^{\circ, \lozenge} \to \mathbf{Sh}_{K_{p,1}^{\circ}}\hyo^{\circ, \lozenge}$ is a torsor for the sheaf of groups $\underline{K_{p,1}^{\circ}}$ associated to the topological group $K_{p,1}^{\circ}$. Now we observe that there is a natural (in particular $\Gamma$-equivariant) isomorphism
\begin{align}
    \operatorname{Sht}_{\mathcal{H}_1^{\circ},\mu, E} \simeq \left[\operatorname{Gr}_{H, \mu^{-1}}/\underline{K_{p,1}^{\circ}} \right],
\end{align}
see \cite[Proposition 11.17]{ZhangThesis}. Moreover, by construction, the $\underline{K_{p,1}^{\circ}}$-torsor over $\mathbf{Sh}_{K_{p,1}^{\circ}}\hyo^{\circ, \lozenge}$ coming from the map
\begin{align}
    \mathbf{Sh}_{K_{p,1}}\hyo^{\circ, \lozenge} \to \operatorname{Sht}_{\calhocirc,\mu, E},
\end{align}
is given by $\mathbf{Sh}_{}\hyo^{\circ, \lozenge} \to \mathbf{Sh}_{K_{p,1}^{\circ}}\hyo^{\circ, \lozenge}$. Thus we have a $\Gamma$-equivariant $2$-Cartesian diagram
\begin{equation}
    \begin{tikzcd}
        \mathbf{Sh}_{}\hyo^{\circ, \lozenge} \arrow{d} \arrow{r} & \operatorname{Gr}_{H, \mu^{-1}} \arrow{d} \\
        \mathbf{Sh}_{}\hyo^{\circ, \lozenge}/\underline{K_{p,1}^{\circ}} \arrow{r} & \left[\operatorname{Gr}_{H, \mu^{-1}}/\underline{K_{p,1}^{\circ}} \right],
    \end{tikzcd}
\end{equation}
whose homotopy fixed points are again $2$-Cartesian by Lemma \ref{Lem:FiberProducts}. The homotopy fixed points diagram looks like (see the proof of Proposition \ref{Prop:FixedPointsLocalModels})
\begin{equation}
    \begin{tikzcd}
        \mathbf{Sh}_{}\gx^{\circ, \lozenge} \arrow{d} \arrow{r} & \operatorname{Gr}_{G, \mu^{-1}} \arrow{d} \\
        \left(\mathbf{Sh}_{}\hyo^{\circ, \lozenge}/\underline{K_{p,1}^{\circ}}\right)^{\Gamma} \arrow{r} & \left[\operatorname{Gr}_{H, \mu^{-1}}/\underline{K_{p,1}^{\circ}} \right]^{h \Gamma}.
    \end{tikzcd}
\end{equation}
If we base change the bottom row via $\mathbb{B} \underline{K_{p}^{\circ, \Gamma}} \to \mathbb{B} \underline{K_{p,1}^{\circ}}$, then by Proposition \ref{Prop:KeyPropositionAppendix} we get the Cartesian diagram
\begin{equation}
    \begin{tikzcd}
        \mathbf{Sh}_{}\gx^{\circ, \lozenge} \arrow{d} \arrow{r} & \operatorname{Gr}_{G, \mu^{-1}} \arrow{d} \\
        \mathbf{Sh}_{}\gx^{\circ, \lozenge}/\underline{K_{p}^{\circ, \Gamma}} \arrow{r} & \left[\operatorname{Gr}_{G, \mu^{-1}}/\underline{K_{p}^{\circ, \Gamma}}\right].
    \end{tikzcd}
\end{equation}
By Proposition \ref{Prop:KeyPropositionAppendix}, this proves that the following diagram is Cartesian
\begin{equation}
    \begin{tikzcd}
    \mathbf{Sh}_{K_p^{\circ, \Gamma}}\gx^{\circ, \lozenge} \arrow{d} \arrow{r} & \mathbf{Sh}_{K_{p,1}^{\circ}}\hyo^{\circ, \Gamma, \lozenge} \arrow{d} \\
    \left[\operatorname{Gr}_{G, \mu^{-1}}/\underline{K_{p}^{\circ, \Gamma}}\right] \arrow{r} & \left[\operatorname{Gr}_{H, \mu^{-1}}/\underline{K_{p,1}^{\circ}} \right]^{h \Gamma},
    \end{tikzcd}
\end{equation}
proving the lemma.
\end{proof}

\begin{proof}[Proof of Theorem \ref{Thm:FixedPointsIgusa}]
It follows from Theorem \ref{Thm:IgusaStack} that there is a $2$-commutative square
\begin{equation}
    \begin{tikzcd}
    \igs \gx \arrow{r} \arrow{d} & \bungmu \arrow{d} \\
     \igs\hyo^{h \Gamma} \arrow{r} & \bun_{H_1}^{h \Gamma}.
    \end{tikzcd}
\end{equation}
This induces a map 
\begin{align}
      \igs\gx \to  \igs\hyo^{h \Gamma} \times_{\bun_{H_1}^{h \Gamma}} \bungmu
\end{align}
which we will show is an isomorphism. By v-descent we can do this after basechanging via the v-cover
\begin{align}
    \operatorname{Sht}_{\calgcirc,\mu,E} \to \bungmu
\end{align}
from \cite[Corollary 6.4.2]{DvHKZIgusaStacks}. Using Theorem \ref{Thm:IgusaStack}, Lemma \ref{Lem:PotentiallyCrystalline} and Lemma \ref{Lem:FiberProducts}, we can identify the basechanged map with the natural map
\begin{align}
    \mathbf{Sh}_{K_p^{\circ, \Gamma}}\gx^{\circ, \lozenge} \to \mathbf{Sh}_{K_{p,1}^{\circ}}\hyo^{\circ, \lozenge, \Gamma} \times_{\operatorname{Sht}_{\calhocirc,\mu,E}^{h \Gamma}} \operatorname{Sht}_{\calgcirc,\mu,E}. 
\end{align}
This natural map is an isomorphism by Lemma \ref{Lem:FixedPointsGenericRamifiedHodge}.
\end{proof}

\subsection{Fixed points of integral models of Shimura varieties of Hodge type II} \label{sub:ProofIgusa} We now prove Theorem \ref{Thm:FixedPointsVsheavesIntroduction}, which we will restate for the convenience of the reader.  Let $\gx$ be a Shimura datum of Hodge type with reflex field $\mathsf{E}$. Fix a prime $p$ and a prime $v$ of $\mathsf{E}$ above $p$, let $E,G,\mathcal{O}_E, \f,\Gamma,F,\mathcal{O}_F$ be as above. Let $\mathcal{G}$ be a parahoric model of $G$ over $\zp$, let $\mathcal{H}=\operatorname{Res}_{\mathcal{O}_F/\zp} \mathcal{G}_{\mathcal{O}_F}$ and let $K_p=\mathcal{H}(\zp)$. Let $\mathcal{H}_1$ be the group smoothening of the Zariski closure of $\mathcal{H}(\zpbr) \cap H_1(\qpbr)$ in $\mathcal{H}$; this is a quasi-parahoric model of $H_1$ with $\mathcal{H}_1(\zp)=K_{p,1}=K_{p} \cap H_1(\qp)$. 
\begin{Thm} \label{Thm:FixedPointsIntegralHodgeRamified}
If $\Sha^1(\mathbb{Q}, \mathsf{G}) \to \Sha^1(\f, \g)$ is injective, then the following diagram of v-stacks is $2$-Cartesian
\begin{equation}
    \begin{tikzcd}
    \scrs_{K_p^{\circ, \Gamma}}\gx^{\lozenge/} \arrow{r} \arrow{d} & \scrs_{K_{p,1}}\hyo^{ \Gamma, \lozenge/} \arrow{d} \\
    \shtgcircmu \arrow{r} & \shthomu^{h \Gamma}.
    \end{tikzcd}
\end{equation}    
\end{Thm}
\begin{proof} If we basechange the isomorphism
\begin{align}
     \igs\gx \to  \igs\hyo^{h \Gamma} \times_{\bun_{H_1}^{h \Gamma}} \bungmu
\end{align}
of Theorem \ref{Thm:FixedPointsIgusa} via $\shtgcircmu \to \bungmu$, then by Theorem \ref{Thm:IgusaStack} we get the natural map 
\begin{align}
     \scrs_{K_p^{\circ,\Gamma}}\gx^{\diamond} \to \scrs_{K_{p,1}^{\circ}}\hyo^{ \Gamma, \diamond} \times_{\shthomu^{h \Gamma}} \shtgcircmu,
\end{align}
which is therefore an isomorphism of v-sheaves. If we combine this isomorphism with Lemma \ref{Lem:FixedPointsGenericRamifiedHodge} (or rather its generalization to quasi-parahoric subgroups, using $\shthomuone$ and \cite[Corollary 3.3.9]{DanielsVHKimZhangII}), then we see that the natural map
\begin{align}
     \scrs_{K_p^{\circ,\Gamma}}\gx^{\lozenge/} \to \scrs_{K_{p,1}^{\circ}}\hyo^{ \Gamma, \lozenge/} \times_{\shthomu^{h \Gamma}} \shtgcircmu,
\end{align}
is an isomorphism (this follows from the definition of $\lozenge/$, see Section \ref{Sec:DiamondFunctors}). This concludes the proof. 
\end{proof}

\appendix \section{Some \texorpdfstring{$(2,1)$}{2,1}-category theory}  \label{Appendix:HFP}
\newcommand{\Ob}{\operatorname{Ob}}
\newcommand{\Mor}{\operatorname{Mor}}
\subsection{Strict \texorpdfstring{$(2,1)$}{(2,1)}-categories and weak functors}
Recall from \cite[Section 02X8]{stacks-project} the definition of a strict $(2, 1)$-category. We will use $\Mor^1(x, y)$ to refer to the category of $1$-morphisms between $x$ and $y$ in such a category, and sometimes abusively to refer to the class of objects of this category using the same notation. We will sometimes refer to isomorphisms in the category $\Mor^1(x, y)$ as natural transformations.

\begin{Def} \label{Def:WeakFunctor}
    Let $\mathcal{C}$ be a 1-category and $\mathcal{D}$ a strict $(2, 1)$-category. We define a \emph{weak functor} $\mathcal{F}: \mathcal{C} \to \mathcal{D}$ to be a pair $(\mathcal{F}, \eta_{\mathcal{F}})$ consisting of:
    \begin{enumerate}
        \item An assignment
    \[
        \mathcal{F}: \Ob(\mathcal{C}) \to \Ob(\mathcal{D})
    \]
     and a map $\mathcal{F}: \Mor(x, y) \to \Ob(\Mor^1(\mathcal{F}(x), \mathcal{F}(y)))$.
    \item For every $x \in \Ob(\mathcal{C})$ a 2-morphism $\eta_{\mathcal{F}, x}: \text{Id}_{\mathcal{F}(x)} \to \mathcal{F}(\text{Id}_{x})$.
    \item For every composable pair $f: x \to y$, $g: y \to z$ in $\mathcal{C}$ a 2-morphism $\eta_{\mathcal{F}, f, g}: \mathcal{F}(g \circ f) \to \mathcal{F}(g) \circ \mathcal{F}(f)$. 
    \end{enumerate}
    Such that:
    \begin{enumerate}
        \item For any morphism $f: x \to y$ in $\mathcal{C}$ we have that 
        \[
        \eta_{\mathcal{F}, f, \text{Id}_y} = \eta_{\mathcal{F}, y} \circ \text{Id}_{\mathcal{F}(f)}
        \]
        and 
        \[\eta_{\mathcal{F}, \text{Id}_x, f} = \text{Id}_{\mathcal{F}(f)} \circ \eta_{\mathcal{F}, x}.
        \]
        \item For any composable triple $f: x \to y, g: y \to z, h: z \to t$ we have 
        \[
            (\text{Id}_{\mathcal{F}(h)} \circ \eta_{\mathcal{F}, f, g}) \circ \eta_{\mathcal{F}, g \circ f, h} = (\eta_{\mathcal{F}, g, h} \circ \text{Id}_{\mathcal{F}(f)}) \circ \eta_{\mathcal{F}, f, h \circ g}.
        \]
    \end{enumerate}
\end{Def}

\begin{Def} \label{Def:WeakNaturalTransformation}
    A weak natural transformation $\epsilon: \mathcal{F} \to \mathcal{G}$ of weak functors $\mathcal{F}, \mathcal{G}: \mathcal{C} \to \mathcal{D}$ is:
    \begin{enumerate}
        \item A collection for $x \in \Ob(\mathcal{C})$ of $\epsilon_x: \mathcal{F}(x) \to \mathcal{G}(x)$ of 1-morphisms in $\mathcal{D}$.
        \item For each morphism $f: x \to y$ in $\mathcal{C}$, a natural transformation $\epsilon(f):\epsilon_y \circ \mathcal{F}(f) \to \mathcal{G}(f) \circ \epsilon_x$ in the category $\Mor^1(\mathcal{F}(x), \mathcal{G}(y))$, we will denote this in diagram form by
        \[
        \begin{tikzcd}
            & \mathcal{F}(x) \arrow{r}{\mathcal{F}(f)}\arrow{d}{\epsilon_x} & \mathcal{F}(y) \arrow{d}{\epsilon_y} \arrow[Rightarrow, swap]{dl}{\epsilon(f)}\\
            & \mathcal{G}(x) \arrow{r}{\mathcal{G}(f)} & \mathcal{G}(y),
        \end{tikzcd}
        \]
    \end{enumerate}
    satisfying the following conditions:
    \begin{enumerate}
        \item We have an equality of natural transformations for every pair of morphisms $f: x \to y, g: y \to z$:
        \[
            \epsilon(g \circ f) \circ \eta_{\mathcal{F}, f, g} = \eta_{\mathcal{G}, f, g} \circ \mathcal{G}(g)_* (\epsilon(f)) \circ \mathcal{F}(f)^*(\epsilon(g)).
        \]
        \item For every object $x \in \Ob(\mathcal{C})$ we ask for commutativity of the diagram
        \[
            \begin{tikzcd}
                &&&\epsilon_x&&\\
                & \epsilon_x \circ \text{Id}_{\mathcal{F}(x)} \arrow{urr}{=}\arrow{dr}{\eta_{\mathcal{F}, x}}&&&& \text{Id}_{\mathcal{G}(x)} \circ \epsilon_x \arrow[swap]{ull}{=}\arrow[swap]{dl}{\eta_{\mathcal{G}, x}}\\
                &&\epsilon_x \circ \mathcal{F}(\text{Id}_x)  \arrow{rr}{\epsilon(\text{Id}_x)}&& \mathcal{G}(\text{Id}_x) \circ \epsilon_x&
            \end{tikzcd}
        \]
        inside of the category $\Mor(\mathcal{F}(x), \mathcal{G}(x))$. We note that because we work with strict $(2, 1)$-categories the commutativity of this diagram is well-posed.
    \end{enumerate}
\end{Def}

Let $\Gamma$ be an abstract group. We define the category $B\Gamma$ to be the classifying category of the abstract group $\Gamma$ (that is, the category with one object $\ast$ with automorphism group $\Gamma$). Suppose $\mathcal{D}$ is a strict $(2, 1)$-category. 

\begin{Def} \label{Def:GammaObject}
    A $\Gamma$-object in $\mathcal{D}$ is a weak functor $\mathcal{F}: B\Gamma \to \mathcal{D}$. A $\Gamma$-equivariant morphism is a weak natural transformation of such functors.
\end{Def}
Let $\mathcal{D}$ be a strict $(2,1)$-category and let $x$ be an object of $\mathcal{D}$. Then a weak functor $\mathcal{F}: B\Gamma \to \mathcal{D}$ sending $\ast$ to $x$ is called a \emph{weak $\Gamma$-action on $x$}.
\begin{eg} \label{eg:TrivialAction}
Let $\mathcal{D}$ be a strict $(2,1)$-category and let $x$ be an object of $\mathcal{D}$. Then the trivial $\Gamma$-action on $x$ is the following weak functor $\mathcal{F}: B\Gamma \to \mathcal{D}$: On objects it sends $\ast$ to $x$ and on morphisms it sends all morphisms to the identity $x \to x$. Moreover the $2$-morphisms $\eta_{\mathcal{F},x}$ and $\eta_{\mathcal{F},f,g}$ are all taken to be the identity $2$-morphism.
\end{eg}
\begin{eg} \label{eg:StrongAction}
Let $\mathcal{D}$ be a strict $(2,1)$-category and let $x$ be an object of $\mathcal{D}$ such that $\Mor^1(x,x)$ is a discrete category.\footnote{A discrete category is a category where the only morphisms are the identity morphisms.} Then $\Ob(\Mor^1(x,x))$ is a group under composition of morphisms and a weak $\Gamma$-action on $x$ is the same as a group homomorphism $\Gamma \to \Ob(\Mor^1(x,x))$.
\end{eg}

\subsubsection{} Recall \cite[Tag 003R]{stacks-project} that fiber products in the $(2, 1)$-category of categories have the following description. Let $\mathcal{A}, \mathcal{B}, \mathcal{C}$ be $(2, 1)$-categories and let $\mathcal{F}: \mathcal{A} \to \mathcal{C}, \mathcal{G}: \mathcal{B} \to \mathcal{C}$ be functors. Then there is the following canonical presentation of the fiber product:

\begin{Def} \label{Def:2FiberProduct}
    The fiber product $\mathcal{A} \times_{\mathcal{C}} \mathcal{B}$ identifies with the strict $(2,1)$-category whose objects consist of triples $(a, b, f)$ where $a \in \Ob(\mathcal{A}), b \in \Ob(\mathcal{B})$ and $f: \mathcal{F}(a) \to \mathcal{G}(b)$ is an isomorphism in $\mathcal{C}$. The morphisms $\phi: (a, b, f) \to (c, d, g)$ consist of pairs $(X, Y)$ where $X: a \to c, Y: b \to d$ are such that the diagram
    \[
        \begin{tikzcd}
            & \mathcal{F}(a)\arrow{r}{f}\arrow{d}{X} & \mathcal{G}(b) \arrow{d}{Y}\\
            & \mathcal{F}(c) \arrow{r}{g}& \mathcal{G}(d)
        \end{tikzcd}
    \]
    commutes. 
\end{Def}

\begin{Def} \label{Def:HomotopyFixedPoints}
    Let $X$ be a $\Gamma$-object in the strict $(2, 1)$-category of categories, defined a weak functor $\mathcal{F}$. Then we define the $\Gamma$-\emph{homotopy fixed points} $X^{h\Gamma}$ of $X$ to be the following category: The objects of $X^{h\Gamma}$ are tuples $(x, \{\tau_{\gamma}\}_{\gamma \in \Gamma})$, where $x \in \Ob(X)$ is an object and $\tau_{\gamma}: x \to \mathcal{F}(\gamma)(x)$ for each $\gamma \in \Gamma$ is an isomorphism such that for all $\gamma, \gamma'$ the diagram
    \[
       \begin{tikzcd} 
            &x \arrow{r}{\tau_{\gamma'}} \arrow[swap]{d}{\tau_{\gamma'\gamma}} &\mathcal{F}(\gamma')(x) \arrow{d}{\mathcal{F}(\gamma')(\tau_{\gamma})} \\
            &\mathcal{F}(\gamma' \gamma)(x)\arrow{r}{\eta_{\mathcal{F}, \gamma, \gamma'}} & \mathcal{F}(\gamma')(\mathcal{F}(\gamma)(x)) 
        \end{tikzcd}
    \]
    is commutative. A morphism $f:(x, \tau) \to (x', \tau')$ of such objects is a morphism $f:x \to x'$ such that the following diagram commutes for all $\gamma \in \Gamma$
    \begin{equation}
    \begin{tikzcd}
        x \arrow{r}{f} \arrow{d}{\tau_{\gamma}} & x' \arrow{d}{\tau'_{\gamma}} \\
        \mathcal{F}(\gamma)(x) \arrow{r}{\mathcal{F}(f)} & \mathcal{F}(\gamma)(x').
    \end{tikzcd}
    \end{equation}
\end{Def}
\begin{Lem} \label{Lem:FunctorialityHomotopyFixedPoints}
A $\Gamma$-equivariant morphism $\alpha:X \to Y$ of $\Gamma$-objects in the strict $(2, 1)$-category of categories defines a natural functor 
\begin{align}
    \alpha^{h \Gamma}:X^{h \Gamma} \to Y^{h \Gamma}.
\end{align}
Furthermore, this functor is an equivalence if $\alpha$ is an equivalence.
\end{Lem}
\begin{proof}
Let $\mathcal{CAT}$ be the strict $(2,1)$-category of categories and suppose that $X$ and $Y$ are given by weak functors $\mathcal{F}, \mathcal{G}:B\Gamma \to \mathcal{CAT}$. Then the morphism $\alpha$ is precisely a weak natural transformation $\epsilon:\mathcal{F} \to \mathcal{G}$. In particular, for each $\gamma \in \Gamma$ we have a natural transformation
\begin{equation}
    \begin{tikzcd}
            & X \arrow{r}{\mathcal{F}(\gamma)}\arrow[d, swap, "\alpha"] & X \arrow{d}{\alpha} \arrow[Rightarrow, swap]{dl}{\epsilon(f)}\\
            & Y \arrow[r,swap,"\mathcal{G}(\gamma)"] & Y.
        \end{tikzcd}
\end{equation}
Given a tuple $(x, \{\tau_{\gamma}\}_{\gamma \in \Gamma}) \in X^{h \Gamma}$ we would like to define its image under $\alpha^{h \Gamma}$ to be the tuple $(\alpha(x), \{\epsilon_{\gamma} \circ \alpha(\tau_{\gamma})\}_{\gamma \in \Gamma})$. To check that this tuple defines an object of $Y^{h \Gamma}$, we need to check that for $\gamma, \gamma' \in \Gamma$ the diagram
\begin{equation}
    \begin{tikzcd}
        \alpha(x) \arrow{r}{\epsilon_{\gamma'} \circ \alpha(\tau_{\gamma'})} \arrow{d}{\epsilon_{\gamma'\gamma} \circ \alpha(\tau_{\gamma'\gamma})} & \mathcal{G}(\gamma')(\alpha(x)) \arrow{d}{\mathcal{G}(\gamma')(\epsilon_{\gamma} \circ \alpha(\tau_{\gamma}))} \\
        \mathcal{G}(\gamma'\gamma)(\alpha(x)) \arrow{r}{\eta_{\mathcal{G}, \gamma, \gamma'}} & \mathcal{G}(\gamma')(\mathcal{G}(\gamma)(\alpha(x)) 
    \end{tikzcd}
\end{equation}
commutes. But this is a direct consequence of the fact that $\epsilon$ is a weak natural transformation. Thus we can define $\alpha^{h \Gamma}$ on the level of objects by sending 
$(x, \{\tau_{\gamma}\}_{\gamma \in \Gamma}) \in X^{h \Gamma}$ to the tuple $(\alpha(x), \{\epsilon_{\gamma} \circ \alpha(\tau_{\gamma})\}_{\gamma \in \Gamma}) \in Y^{h \Gamma}$. 

Given a morphism $f:(x, \{\tau_{\gamma}\}_{\gamma \in \Gamma}) \to (x', \{\tau'_{\gamma}\}_{\gamma \in \Gamma})$ in $X^{h \Gamma}$, one can check that
\begin{align}
    \alpha(f):\alpha(x) \to \alpha(x')
\end{align}
is a morphism in $Y^{h \Gamma}$. The association $f \mapsto \alpha(f)$ is compatible with compositions, since $\alpha$ is a functor from $X$ to $Y$. Thus we have constructed the desired functor $\alpha^{h\Gamma}$.

If $\alpha$ is fully faithful, then it is immediate that $\alpha^{h \Gamma}$ is also fully faithful. If $\alpha$ is essentially surjective and fully faithful, then $\alpha^{h \Gamma}$ is also essentially surjective. Indeed, given $(y, \{\kappa_{\gamma}\}_{\gamma \in \Gamma}) \in Y^{h \Gamma}$ we may choose an object $x \in X$ and an isomorphism $\xi:\alpha(x) \to y$. For each $\gamma \in \Gamma$ there is an induced isomorphism
\begin{align}
    \alpha(\mathcal{F}(\gamma)) \to \mathcal{G}(\gamma)(y)
\end{align}
given by $\mathcal{G}(\gamma)(\xi) \circ \epsilon_{\gamma}$. Using the fully faithfulness, the isomorphism $\kappa_{\gamma}:y \to \mathcal{G}(\gamma)(y)$ induces a unique isomorphism $\tau_{\gamma}:x \to \alpha(x)$. Using this uniqueness, it is not hard to check that the tuple $(x, \{\tau_{\gamma}\}_{\gamma \in \Gamma})$ defines an object in $X^{h \Gamma}$. Its image under $\alpha^{h \Gamma}$ is moreover isomorphic to $(y, \{\kappa_{\gamma}\}_{\gamma \in \Gamma})$ by construction.
\end{proof}
\begin{Lem} \label{Lem:FiberProductsOfFunctors}
    Let $\mathcal{F}, \mathcal{G}, \mathcal{H}: \mathcal{C} \to \mathcal{CAT}$ be weak functors, where $\mathcal{C}$ is a 1-category and $\mathcal{CAT}$ is the strict $(2,1)$-category of categories. Let $\delta: \mathcal{F} \to \mathcal{H}, \epsilon: \mathcal{G} \to \mathcal{H}$ be weak natural transformations. Then there exists a canonical weak functor $\mathcal{F} \times_{\mathcal{H}} \mathcal{G}: \mathcal{C} \to \mathcal{CAT}$ such that $(\mathcal{F} \times_{\mathcal{H}} \mathcal{G})(x) = \mathcal{F}(x) \times_{\mathcal{H}(x)} \mathcal{G}(x)$ for all $x \in \Ob(\mathcal{C})$. 
\end{Lem}
\begin{proof}
\textbf{Step 1: Objects.} We define $\mathcal{P}:=(\mathcal{F} \times_{\mathcal{H}} \mathcal{G})$ on objects by sending $x \mapsto \mathcal{F}(x) \times_{\mathcal{H}(x)} \mathcal{G}(x)$. For $x,y \in \mathcal{C}$ we define
\begin{align}
    \mathcal{P}:\operatorname{Mor}_{\mathcal{C}}(x,y) \to \operatorname{Ob}\left(\operatorname{Mor}(\mathcal{F}(x) \times_{\mathcal{H}(x)} \mathcal{G}(x), \mathcal{F}(y) \times_{\mathcal{H}(y)} \mathcal{G}(y)\right)
\end{align}
to be the assignment taking $g:x \to y$ to the functor
\begin{align}
    \mathcal{F}(g) \times_{\mathcal{H}(g)} \mathcal{G}(g): \mathcal{F}(x) \times_{\mathcal{H}(x)} \mathcal{G}(x) \to \mathcal{F}(y) \times_{\mathcal{H}(y)} \mathcal{G}(y)
\end{align}
taking $(a,b,f) \in \mathcal{F}(x) \times_{\mathcal{H}(x)} \mathcal{G}(x)$ to the triple $(\mathcal{F}(g)(a), \mathcal{G}(g)(b), \mathcal{P}(g)(f))$, where $\mathcal{P}(g)(f))$ is defined as
\begin{equation}
    \begin{tikzcd}
        \delta_y(\mathcal{F}(g)(a)) \arrow[drr, "\mathcal{P}(g)(f))"] \arrow{d}{\delta(g)} \\
        \mathcal{H}(g)(\delta_x(a)) \arrow{r}{\mathcal{H}(g)(f)} & \mathcal{H}(g)(\delta_x(b)) \arrow[r, swap, "\epsilon(g)^{-1}"] & \epsilon_y(\mathcal{G}(g)(b)).
    \end{tikzcd}
\end{equation}
\textbf{Step 2: Morphisms.} A morphism $p:(a,b,f) \to (a',b',f')$ is sent to the morphism
\begin{align}
    (\mathcal{F}(g)(p), \mathcal{G}(g)(p))
\end{align}
which one checks is a morphism in $\mathcal{F}(y) \times_{\mathcal{H}(y)} \mathcal{G}(y)$. To check that $\mathcal{F}(g) \times_{\mathcal{H}(g)} \mathcal{G}(g)$ is a functor, one uses the fact that $\delta$ and $\epsilon$ are weak natural transformations (and thus have compatibility with composition). \smallskip

\textbf{Step 3: Identity natural transformations.} Next, for each $x \in \Ob(\mathcal{C})$ we define a natural transformation $\eta_x: \text{Id}_{\mathcal{P}(x)} \to (\mathcal{P})(\text{Id}_x)$ given by the morphism
\begin{align}
    \eta_{x,(a,b,f)}: (a,b,f) \to (\mathcal{F}(\operatorname{Id}_x)(a), \mathcal{G}(\operatorname{Id}_x)(b), \mathcal{P}(\operatorname{Id}_x)(f))
\end{align}
described by $(\eta_{\mathcal{F},x,a}, \eta_{\mathcal{G},x,a})$. Naturality follows from the naturality of $\eta_{\mathcal{F}}$ and $\eta_{\mathcal{G}}$, and we must simply check that the morphisms $(\eta_{\mathcal{F}, x,a}, \eta_{\mathcal{G}, x,b})$ are morphisms in the category $\mathcal{F}(x) \times_{\mathcal{H}(x)} \mathcal{G}(x)$. We recall that this comes down to showing that the top square in the following diagram is commutative for every $(a, b, f) \in \Ob(\mathcal{F}(x) \times_{\mathcal{H}(x)} \mathcal{G}(x))$
    \begin{equation} \label{Eq:IdentityNaturalTransformationFiberProduct}
        \begin{tikzcd}
            &\delta_x(a) \arrow{r}{f}\arrow{d}{\delta_x(\eta_{\mathcal{F}, x,a})} & \epsilon_x(b) \arrow{d}{\epsilon_x(\eta_{\mathcal{G}, x,b})}\\
            &\delta_x(\mathcal{F}(\text{Id}_x)(a)) \arrow{d}{\delta(\text{Id}_x)} \arrow{r}{\mathcal{P}(\operatorname{Id}_x)(f)} & \epsilon_x(\mathcal{G}(\text{Id}_x)(b)) \arrow{d}{\epsilon(\text{Id}_x)} \\
            &\mathcal{H}(\text{Id}_x)(\delta_x(a)) \arrow{r}{\mathcal{H}(\text{Id}_x)(f)} &\mathcal{H}(\text{Id}_x)(\epsilon_x(b)).
        \end{tikzcd}
    \end{equation}
But we know that the bottom square is commutative by the definition of $\mathcal{P}(\operatorname{Id}_x)(f)$. Moreover, because $\delta$ is a weak natural transformation from $\mathcal{F}$ to $\mathcal{H}$, we know that $\delta(\text{Id}_x) \circ \delta_x(\eta_{\mathcal{F}, x,a}) = \eta_{\mathcal{H}, x} \circ \delta_x$ as morphisms between $\delta_x(a)$ and $\mathcal{H}(\text{Id}_x)(\delta_x(a))$ in the category $\mathcal{H}(x)$. In addition $\epsilon(\text{Id}_x) \circ \epsilon_x(\eta_{\mathcal{G}, x,b}) = \eta_{\mathcal{H}, x,b} \circ \epsilon_x$ as morphisms from $\epsilon_x(b)$ to $\mathcal{H}(\text{Id}_x)(\epsilon_x(b))$. The commutativity of the outer square of \eqref{Eq:IdentityNaturalTransformationFiberProduct} follows from these last two observations. Since we also have commutativity of the bottom diagram, the commutativity of the top diagram follows since the vertical arrows are all isomorphisms. \smallskip

\textbf{Step 4: Natural transformations for compositions of morphisms.} Finally, for a pair of morphisms $\gamma':x \to y$ and $\gamma: y \to z \in \mathcal{C}$ we define a natural transformation $\eta_{\mathcal{P}, \gamma', \gamma}:\mathcal{P}(\gamma' \circ \gamma) \to \mathcal{P}(\gamma') \circ \mathcal{P}(\gamma)$ of functors $\mathcal{P}(x) \to \mathcal{P}(z)$ which is defined on objects $(a,b,f) \in \mathcal{P}(x)$ by 
    \[
        (\eta_{\mathcal{F}, \gamma, \gamma'}(a) , \eta_{\mathcal{G}, \gamma, \gamma'}(b)).
    \]
Once again, the naturality is straightforward if we can check that this is an isomorphism in $\mathcal{P}(z)$. For this, we consider the following diagram of isomorphisms in $\mathcal{H}(z)$
\begin{equation}
    \begin{tikzcd}
        \mathcal{H}(\gamma' \circ \gamma)(\delta_x(a)) \arrow{rr}{\mathcal{H}(\gamma' \circ \gamma)(f)}  & & \mathcal{H}(\gamma' \circ \gamma)(\epsilon_x(a)) \arrow{d}{\epsilon(\gamma' \gamma)^{-1})}\\
        \delta_z(\mathcal{F}(\gamma' \circ \gamma)(a)) \arrow{u}{\delta(\gamma' \gamma)} \arrow[d,"\delta_z(\eta_{\mathcal{F}, \gamma, \gamma'}(a))", swap]  \arrow{rr}{\mathcal{P}(\gamma' \circ \gamma)(f)}  & & \epsilon_z(\mathcal{G}(\gamma' \circ \gamma)(b)) \arrow[d, "\epsilon_z(\eta_{\mathcal{G}, \gamma, \gamma'}(b))"]  \\
        \delta_z(\mathcal{F}(\gamma')(\mathcal{F}(\gamma)(a))) \arrow{d}{\delta(\gamma')} \arrow{rr}{\mathcal{P}(\gamma')(\mathcal{P}(\gamma)(f))} & & \epsilon_z(\mathcal{G}(\gamma')(\mathcal{G}(\gamma)(b)))   \\
        \mathcal{H}(\gamma')(\delta_y(\mathcal{F}(\gamma)(a))) \arrow{rr}{\mathcal{H}(\gamma')(\mathcal{P}(\gamma)(f))} \arrow{d}{\mathcal{H}(\gamma')(\delta(\gamma))} & & \mathcal{H}(\gamma')(\epsilon_y(\mathcal{G}(\gamma)(a))) \arrow{u}{\epsilon(\gamma')^{-1}} \\
        \mathcal{H}(\gamma')(\mathcal{H}(\gamma)(\delta_x(a))) \arrow{rr}{\mathcal{H}(\gamma')(\mathcal{H}(\gamma)(f))} & &\mathcal{H}(\gamma')(\mathcal{H}(\gamma)(\epsilon_x(b))) \arrow{u}{(\mathcal{H}(\gamma')(\epsilon(\gamma))^{-1}}.
    \end{tikzcd}
\end{equation}
We are asked to show that the second square from the top commutes, and we see that this reduces to checking commutativity of the outer square in the diagram. But this again follows from the fact that $\delta$ and $\epsilon$ are weak natural transformations.
\smallskip 

\textbf{Step 5: End of the proof.} We have now specified the data required to define a weak functor $\mathcal{C} \to \mathcal{CAT}$. To check that this is a weak functor, we need to check certain identities of natural transformations hold, see Definition \ref{Def:WeakFunctor}. Since all natural transformations $\eta_{\mathcal{P}}$ are defined as pairs $(\eta_{\mathcal{F}}, \eta_{\mathcal{G}})$, the identities for $\eta_{\mathcal{P}}$ follow from those for $\eta_{\mathcal{F}}$ and $\eta_{\mathcal{G}}$.
\end{proof}
We have the following key lemma. Let $\mathcal{F}, \mathcal{G}, \mathcal{H}:B \Gamma \to \mathcal{CAT}$ be weak functors and let $\delta: \mathcal{F} \to \mathcal{H}, \epsilon: \mathcal{G} \to \mathcal{H}$ be weak natural transformations. Recall the fiber product weak functor $\mathcal{F} \times_{\mathcal{H}} \mathcal{G}$ from Lemma \ref{Lem:FiberProductsOfFunctors}. 
\begin{Lem}\label{Lem:FiberProducts}
There is an equivalence of categories
    \[
        \beta:\mathcal{F}^{h\Gamma}(\ast) \times_{\mathcal{H}^{h\Gamma}(\ast)} \mathcal{G}^{h\Gamma}(\ast) \to (\mathcal{F} \times_{\mathcal{H}} \mathcal{G})(\ast)^{h \Gamma}.
    \]
\end{Lem}
\begin{proof} Let us write $\mathcal{P}=\mathcal{F} \times_{\mathcal{H}} \mathcal{G}$ for simplicity. 

\textbf{Step 1: Defining the functor on objects.} The functor $\beta$ has the following construction on objects: Let $(a,\{\tau_{ \gamma}\}_{\gamma \in \Gamma})$ be an object in $\mathcal{F}^{h\Gamma}(\ast)$, let $(b,\{\sigma_{ \gamma}\}_{\gamma \in \Gamma})$ be an object in $\mathcal{G}^{h\Gamma}(\ast)$ and let $f: \delta^{h \Gamma}(a,\{\tau_{\gamma}\}_{\gamma \in \Gamma}) \to \epsilon^{h \Gamma}(b,\{\sigma_{\gamma}\}_{\gamma \in \Gamma})$ be an isomorphism in $\mathcal{H}^{h \Gamma}(\ast)$. Observe that $f$ is per definition an automorphism $\delta(a) \to \epsilon(b)$ in $\mathcal{H}$ satisfying certain extra properties.
    
    We define $\beta((a,\{\tau_{\gamma}\}_{\gamma \in \Gamma}),(b,\{\sigma_{\gamma}\}_{\gamma \in \Gamma}),f)$ to be the triple $(a,b,f, \{\rho_{\gamma}\}_{\gamma \in \Gamma})$, where $\rho_{\gamma}$ is the pair $(\tau_{\gamma}, \rho_{\gamma})$. To check that our triple defines an element of $\mathcal{P}(\ast)^{h\Gamma}$, we need to check that for all $\gamma, \gamma'$ the diagram
    \begin{equation} \label{Eq:CommutativityGammaAction}
        \begin{tikzcd}
            (a,b,f) \arrow{r}{\rho_{\gamma}} \arrow{d}{\tau_{\gamma'\gamma}} & \mathcal{P}(\gamma')(a,b,f) \arrow{d}{\mathcal{P}(\gamma)(\rho_{\gamma})} \\
            \mathcal{P}(\gamma' \gamma)(a,b,f) \arrow{r}{\eta_{\mathcal{P}, \gamma, \gamma'}} & \mathcal{P}(\gamma')(\mathcal{P}(\gamma)(a,b,f))
        \end{tikzcd}
    \end{equation}
    commutes. But morphisms in $\mathcal{P}(\ast)^{h \Gamma}$ are determined by the corresponding morphisms in $\mathcal{P}(\ast)$, which are in turn determined by a pair consisting of a morphism in $\mathcal{F}(\ast)$ and a morphism in $\mathcal{G}(\ast)$. The commutativity of the diagram \eqref{Eq:CommutativityGammaAction} then follows from the definition of $\eta_{\mathcal{P}, \gamma, \gamma'}$ and $\mathcal{P}$, in combination with the fact that $(a,\{\tau_{ \gamma}\}_{\gamma \in \Gamma})$ is an object in $\mathcal{F}^{h\Gamma}(\ast)$ and the fact that $(b,\{\sigma_{ \gamma}\}_{\gamma \in \Gamma})$ is an object of $\mathcal{G}^{h\Gamma}(\ast)$. \smallskip
    
\textbf{Step 2: Defining the functor on morphisms.} A morphism $$g:((a,\{\tau_{\gamma}\}_{\gamma \in \Gamma}),(b,\{\sigma_{\gamma}\}_{\gamma \in \Gamma}),f) \to ((a',\{\tau'_{\gamma}\}_{\gamma \in \Gamma}),(b',\{\sigma'_{\gamma}\}_{\gamma \in \Gamma}),f')$$ in $\mathcal{F}^{h\Gamma}(\ast) \times_{\mathcal{H}^{h\Gamma}(\ast)} \mathcal{G}^{h\Gamma}(\ast)$ corresponds to a pair of morphisms $(g_{1}, g_{2})$. Concretely, we have $g_1:a \to a'$ and $g_2:b \to b'$ such that the following diagrams commute for all $\gamma \in \Gamma$:
\begin{equation} \label{Eq:ThreeDiagrams}
    \begin{tikzcd}
        a \arrow{r}{g_1} \arrow{d}{\tau_{\gamma}} & a' \arrow{d}{\tau'_{\gamma}} \\
        \mathcal{F}(\gamma)(a) \arrow{r}{\mathcal{F}(g_1)} & \mathcal{F}(\gamma)(a')
    \end{tikzcd}
       \begin{tikzcd}
        b \arrow{r}{g_2} \arrow{d}{\sigma_{\gamma}} & b' \arrow{d}{\sigma'_{\gamma}} \\
        \mathcal{G}(\gamma)(b) \arrow{r}{\mathcal{F}(g_2)} & \mathcal{G}(\gamma)(b')
    \end{tikzcd}
    \begin{tikzcd}
        \delta(a) \arrow{r}{f} \arrow{d}{\delta(g_1)} & \epsilon(b) \arrow{d}{\epsilon(g_2)} \\
        \delta(a') \arrow{r}{f'} & \epsilon(b').
    \end{tikzcd}
\end{equation}
We will define $\beta(g)$ to be the morphism
\begin{align}
    (a,b,f, \{\rho_{\gamma}\}_{\gamma \in \Gamma}) \to (a',b',f', \{\rho'_{\gamma}\}_{\gamma \in \Gamma})
\end{align}
corresponding to $(g_1,g_2)$. This is a morphism in $\mathcal{P}(\ast)$ between $(a,b,f)$ and $(a',b',f)$ by the commutativity of the third diagram in \eqref{Eq:ThreeDiagrams}. To show that this is a morphism in $\mathcal{P}(\ast)^{h\Gamma}$, we have to show that for all $\gamma \in \Gamma$ the diagram
\begin{equation} \label{Eq:OneDiagram}
    \begin{tikzcd}
        (a,b,f) \arrow{r}{g} \arrow{d}{\rho_{\gamma}} & (a',b',f') \arrow{d}{\rho'_{\gamma}} \\
        P(\gamma)(a,b,f) \arrow{r}{\mathcal{P}(\gamma)(g)} & P(\gamma)(a',b',f')
    \end{tikzcd}
\end{equation}
commutes. But morphisms in $\mathcal{P}(\ast)$ are determined by pairs of morphisms in $\mathcal{F}(\ast)$ and $\mathcal{H}(\ast)$. The commutativity now follows from the commutativity of the first two diagrams of \eqref{Eq:ThreeDiagrams} and the definition of $P(\gamma)$ and $\rho_{\gamma}$. \smallskip
    
\textbf{Step 3: Showing $\beta$ is fully faithful.} Suppose we are given objects $$((a,\{\tau_{\gamma}\}_{\gamma \in \Gamma}),(b,\{\sigma_{\gamma}\}_{\gamma \in \Gamma}),f), ((a',\{\tau'_{\gamma}\}_{\gamma \in \Gamma}),(b',\{\sigma'_{\gamma}\}_{\gamma \in \Gamma}),f') \in \mathcal{F}^{h\Gamma}(\ast) \times_{\mathcal{H}^{h\Gamma}(\ast)} \mathcal{G}^{h\Gamma}(\ast).$$ We have seen in Step 2 that a morphism between these objects consists of a pair of morphisms $g_1:a \to a'$ and $g_2:b \to b'$ such that the diagrams in \eqref{Eq:ThreeDiagrams} commute. There are also conditions for the pair $(g_1,g_2)$ to define a morphism between
\begin{align}
    \beta((a,\{\tau_{\gamma}\}_{\gamma \in \Gamma}),(b,\{\sigma_{\gamma}\}_{\gamma \in \Gamma}),f)), \beta((a',\{\tau'_{\gamma}\}_{\gamma \in \Gamma}),(b',\{\sigma'_{\gamma}\}_{\gamma \in \Gamma}),f')),
\end{align}
namely the commutativity of the third diagram in \eqref{Eq:ThreeDiagrams} and the commutativity of \eqref{Eq:OneDiagram}. But it is immediate that the latter condition is equivalent to the commutativity of the first two diagrams in \eqref{Eq:ThreeDiagrams}, proving fully faithfulness. \smallskip

\textbf{Step 4: Essential surjectivity of $\beta$.} Let $(a,b,f,\{\rho_{\gamma}\}_{\gamma \in \Gamma})$ be an object of $\mathcal{P}(\ast)^{h \Gamma}$. Then for each $\gamma \in \Gamma$ the map $\rho_{\gamma}$ consists of a pair of an isomorphism $\tau_{\gamma}:a \to \mathcal{F}(\gamma)(a)$ and an isomorphism $\sigma_{\gamma}:b \to \mathcal{G}(\gamma)(b)$. It is straightforward to show that $(a,b,f,\{\rho_{\gamma}\}_{\gamma \in \Gamma})$ is equal to (!!) the image under $\beta$ of the triple $(a,\{\tau_{\gamma}\}_{\gamma \in \Gamma}),(b,\{\sigma_{\gamma}\}_{\gamma \in \Gamma}),f)$. 
\end{proof}

\subsection{Quotient stacks and fixed points} \label{Sec:AppendixQuotientStacks} In this section we will study the homotopy fixed points of quotients stacks and in particular the interaction between the formation of quotient stacks and homotopy fixed points. 

Let $S$ be a small category with all fiber products, equipped with a Grothendieck pretopology. We will consider the strict $(2,1)$-category $\operatorname{Cat}_{/S}$ of categories fibered in groupoids over $S$, see \cite[Definition 02XS]{stacks-project}. The $2$-fiber product of categories fibered in groupoids as described in \cite[Lemma 0040]{stacks-project} is the same as the one described in Definition \ref{Def:2FiberProduct}, thus Lemma \ref{Lem:FiberProducts} holds for categories fibered in groupoids. \smallskip

Let $G$ be a sheaf of groups on $S$ and consider the category fibered in groupoids $\mathbb{B} G \to S$: Its objects are pairs $(\mathcal{P},T)$ where $T$ is an object of $S$ and where $\mathcal{P} \to T$ is a left $G$-torsor. A morphism $(\mathcal{P},T) \to (\mathcal{P}',T')$ is a pair of morphisms $a:T \to T'$ and $b:\mathcal{P} \to \mathcal{P}'$ such that $b$ is $G$-equivariant and such that the following diagram commutes and is Cartesian
\begin{equation}
    \begin{tikzcd}
    \mathcal{P} \arrow{r}{b} \arrow{d} & \mathcal{P}' \arrow{d} \\
    T \arrow{r}{a} & T'.
    \end{tikzcd}
\end{equation}
The composition of morphisms is given by concatenating Cartesian diagrams.

Recall the notation of a stack in groupoids over $S$, see \cite[Definition 02ZI]{stacks-project}. We observe that the proof of \cite[Lemma 04UK]{stacks-project} can be repeated to show that $\mathbb{B} G \to S$ is a stack in groupoids over $S$.

Given a morphism $f:G \to G'$ of sheaves of groups, there is a functor $\mathbb{B}f:\mathbb{B} G \to \mathbb{B} G'$ sending
\begin{align}
    (T,\mathcal{P}) \mapsto (T, \mathcal{P} \times^{G} G')
\end{align}
and on morphisms sending $b:\mathcal{P} \to \mathcal{P}'$ to the morphism $b':\mathcal{P} \times^{G} G' \to \mathcal{P} \times^{G} G'$ defined by $b'(p,g')=(b(p),g')$. Here we consider $\mathcal{P} \times^{G} G'$ as the quotient (sheaf) of $P \times G'$ by the right action of $G$ given by $(p,h) \cdot g = (g^{-1} \cdot p, h f(g))$. We let $G'$ act on $P \times G'$ by $g' \cdot (p,h) = (p, g' h)$, which clearly descends to the quotient $\mathcal{P} \times^{G} G'$ turning that quotient into a left $G'$-torsor. 
\begin{Lem} \label{Lem:FunctorialityClassifyingStack}
Let $\operatorname{Grp}_{S}$ be the $1$-category of sheaves of groups in $S$. There is a weak functor 
\begin{align}
    \mathbb{B}:\operatorname{Grp}_{S} \to \operatorname{Cat}_{/S}
\end{align}
which on objects sends $G$ to $\mathbb{B}G$ and on morphisms sends $f:G \to G'$ to $\mathbb{B}f:\mathbb{B} G \to \mathbb{B} G'$.
\end{Lem}
\begin{proof}
Part (1) of Definition \ref{Def:WeakFunctor} has been specified by the lemma. For part (2) we have to specify a $2$-morphism
\begin{align}
    \eta_{\mathbb{B},G}:\operatorname{Id}_{\mathbb{B}G} \to \mathbb{B}(\operatorname{Id}_G).
\end{align}
We take the one which is given on objects by the map
\begin{align}
    \mathcal{P} &\to \mathcal{P}\times^{G} G \\
    p &\mapsto  (p,1).
\end{align}
Moreover, given $f_1:G_1 \to G_2$ and $f_2:G_2 \to G_3$ we have to specify a $2$-morphism
\begin{align}
\eta_{\mathbb{B},f_1,f_2}:\mathbb{B}(f_2 \circ f_1) \to \mathbb{B}(f_2) \circ \mathbb{B}(f_1).
\end{align}
We take the one given by the isomorphism
\begin{align}
    \mathcal{P} \times^{G_1} G_3 &\to \left(\mathcal{P}\times^{G_1} G_2 \right) \times^{G_2} G_3 \\
    (p,g_3) &\mapsto ((p,1),g_3).
\end{align}
It is straightforward, if somewhat tedious, to check that these coherence data satisfy Definition \ref{Def:WeakFunctor}
\end{proof}
Now let $X$ be a sheaf on $S$ equipped with a left action of a sheaf of groups $G$. Then we define a category fibered in groupoids $\left[ X / G \right] \to S$ whose objects are triples $(T,\mathcal{P},\omega)$, where $T$ is an object of $S$, where $\mathcal{P} \to T$ is a $G$-torsor and where $\omega:\mathcal{P} \to X$ is a $G$-equivariant map of sheaves. A morphism $(T,\mathcal{P},\omega) \to (T',\mathcal{P}',\omega')$ is a pair $(a,b)$, where $a:T \to T'$ is a morphism in $S$ and $b:\mathcal{P} \to \mathcal{P}'$ is a $G$-equivariant morphism such that the following two diagrams commute 
\begin{equation}
    \begin{tikzcd}
    \mathcal{P} \arrow{r}{b} \arrow{d} & \mathcal{P}' \arrow{d} \\
    T \arrow{r}{a} & T',
    \end{tikzcd} \qquad 
    \begin{tikzcd}
    \mathcal{P} \arrow{r}{b} \arrow{d}{\omega} & \mathcal{P}' \arrow{d}{\omega'} \\
    X \arrow{r}{\operatorname{Id}_X} & X,
    \end{tikzcd}
\end{equation}
and such that the first diagram is Cartesian. We observe that $\left[X / G \right]$ is a stack in groupoids over $S$; indeed, the proof of \cite[Lemma 0370]{stacks-project} goes through verbatim.

Note that there is a natural forgetful morphism $\left[ X / G \right] \to \mathbb{B} G$ of categories fibered in groupoids over $S$ which sends $(T,\mathcal{P},\omega) \mapsto (T,\mathcal{P})$. \medskip 

\subsubsection{} Let $\operatorname{GrpShv}_S$ be the category whose objects are triples $(G,X,A)$, where $G$ is a sheaf of groups on $S$, where $X$ is a sheaf on $S$, and where $A:G \times X \to X$ is a left action of $G$ on $X$. A morphism $(G,X,A) \to (G',X',A')$ consists of a homomorphism of groups $h:G \to G'$ and a morphism of sheaves $f:X \to X'$, such that $f$ is $G$-equivariant via $h$. There is an obvious forgetful functor $\operatorname{GrpShv}_S \to \operatorname{Grp}_S$ sending $(G,X,A)$ to $G$. We consider $\mathbb{B}$ as a weak functor on $\operatorname{GrpShv}_S$ via this forgetful functor.
\begin{Lem} \label{Lem:FunctorialityQuotientStack}
There is a weak functor $\mathscr{Q}:\operatorname{GrpShv}_S \to \operatorname{Cat}_{/S}$, which on objects sends $(G,X,A)$ to $\left[ X / G \right]$. There is moreover a weak natural transformation $\mathscr{Q} \to \mathbb{B}$, which on objects induces the natural forgetful functor
\begin{align}
    \left[ X / G \right] \to \mathbb{B} G.
\end{align}
\end{Lem}
\begin{proof}
Given a morphism $(h,f):(G,X,A) \to (G',X',A')$ we define a morphism
\begin{align}
    \mathscr{Q}(h,f):\left[ X / G \right] \to \left[ X' / G' \right]
\end{align}
by sending $(T,\mathcal{P}, \omega) \mapsto (T, \mathcal{P} \times^G G', \omega_h)$, where 
\begin{align}
    \omega_h:\mathcal{P} \times^G G' &\to X' \\
    (p,g') &\mapsto A'(g',f(\omega(p)).
\end{align}
For the identity natural transformation
\begin{align}
    \eta_{\mathscr{Q},(G,X,A)}:\operatorname{Id}_{\mathscr{Q}(G,X,A)} \to \mathscr{Q}(\operatorname{Id}_G, \operatorname{Id}_X)
\end{align}
we take the isomorphism $(T,\mathcal{P}, \omega) \mapsto (T, \mathcal{P}\times^G G, \omega_{\operatorname{Id}_G})$ given by
\begin{align}
    \mathcal{P} &\to \mathcal{P}\times^G G \\
    p &\mapsto (p,1),
\end{align}
which is compatible with $\omega$ and $\omega_{\operatorname{Id}_G}$ since
\begin{align}
    (\omega_{\operatorname{Id}_G})(p,1)=A(1,\operatorname{Id}_X(\omega(p))=\omega(p).
\end{align}
Given a pair of composable morphisms $(h,f):(G_1,X_1,A_1) \to (G_2,X_2,A_2)$ and $(h',f'):(G_2,X_{2},A_2) \to (G_3,X_3,A_3)$, we need to specify a natural transformation
\begin{align}
    \eta_{\mathscr{Q},(h,f),(h',f')}:\mathscr{Q}(h' \circ h, f' \circ f) \to \mathscr{Q}(h',f') \circ \mathscr{Q}(h,f). 
\end{align}
For this we take the isomorphism in $\left[ X_3 / G_3 \right]$ given by 
\begin{align}
    (S,\mathcal{P} \times^{G_1} G_3,\omega_{h' \circ h}) \to (S, \left(\mathcal{P} \times^{G_2} G_2 \right) \times^{G_2} G_3), \left(\omega_{h}\right)_{h'})
\end{align}
given by the isomorphism
\begin{align}
    \mathcal{P} \times^{G_1} G_3 &\to \left(\mathcal{P} \times^{G_1} G_2 \right) \times^{G_2} G_3 \\
    (p,g_3) &\mapsto ((p,1),g_3)
\end{align}
from the proof of Lemma \ref{Lem:FunctorialityClassifyingStack}. We have to check that this is an isomorphism in $\left[ X_3 / G_3 \right]$ which comes down to the equality
\begin{align}
    \omega_{h'h}(((p,1),g_3)) = \omega_{h' \circ h}(p,g_3).
\end{align}
To check this we compute that
\begin{align}
   \left(\omega_{h}\right)_{h'}(((p,1),g_3))&=A'(g_3, f'(\omega_h(p,1)))\\
    &=A_3(g_3, f'(A_2(1,f(p)))) \\
    &=A_3(g_3, f'(f(p))) \\
    &=\omega_{h' \circ h}(p,g_3).
\end{align}
As in the proof of Lemma \ref{Lem:FunctorialityClassifyingStack}, we omit the verification that these coherence data satisfy Definition \ref{Def:WeakFunctor}. \medskip 

The weak natural transformation $\mathscr{Q} \to \mathbb{B}$ is easy to specify because it follows by construction that for $(h,f):(G,X,A) \to (G',X',A')$ the diagram
\begin{equation}
    \begin{tikzcd}
    \left[X /G\right] \arrow{r}{\mathscr{Q}(h,f)} \arrow{d} & \left[X' /G' \right] \arrow{d} \\
    \mathbb{B} G \arrow{r}{\mathbb{B}(h)} & \mathbb{B} G'
    \end{tikzcd}
\end{equation}
is strictly commutative. Thus we can specify the identity natural transformation in (2) of Definition \ref{Def:WeakNaturalTransformation}, which clearly satisfies the properties outlined in that definition.
\end{proof}
We will now spell out a particular case of Lemma \ref{Lem:FunctorialityQuotientStack}: Let $\Gamma$ be
an abstract group, let $(G,X,A) \in \operatorname{GrpShv}_S$ and let $\Gamma \to \operatorname{Isom}_{\operatorname{GrpShv}_S}((G,X,A),(G,X,A))$ be a group homomorphism. This gives us a functor
\begin{align}
    B \Gamma \to \operatorname{GrpShv}_S,
\end{align}
which we can compose with the weak functor of Lemma \ref{Lem:FunctorialityQuotientStack} to get a $\Gamma$-object in categories fibered in groupoids over $S$, see Example \ref{eg:StrongAction}. Concretely, the group homomorphism $$\Gamma \to \operatorname{Isom}_{\operatorname{GrpShv}_S}((G,X,A),(G,X,A))$$ comes down to the following data: We have a left action of $\Gamma$ on $G$ over $S$; let us write $\alpha_\gamma:G \to G$ for the induced isomorphism of groups for each $\gamma \in \Gamma$. We have a left action of $\Gamma$ on $X$, which we denote by $\beta_\gamma:X \to X$ for each $\gamma \in \Gamma$. These morphisms are subject to the following commutative diagram for all $\gamma \in \Gamma$
\begin{equation}
    \begin{tikzcd}
    G \times X \arrow{r}{A} \arrow{d}{\alpha_{\gamma} \times \beta_{\gamma}}  & X \arrow{d}{\beta_{\gamma}} \\
    G \times X \arrow{r}{A} & X.
    \end{tikzcd}
\end{equation}
The induced action of $\Gamma$ on $\left[ X / G \right]$ on objects takes a tuple $(T,\mathcal{P}, \omega)$ and sends it to the tuple $(T, \mathcal{P} \times^{G, \gamma} G, \omega_{\gamma})$. Concretely, we have that
\begin{align}
    \omega_{\gamma}(p,g) &= A(g,\beta_{\gamma}(\omega(p)))
\end{align}
which means that $\omega_{\gamma}(p,1)=\beta_{\gamma}(\omega(p))$.

\subsubsection{} By Lemma \ref{Lem:FunctorialityQuotientStack} there is a $\Gamma$-equivariant morphism
\begin{align}
    \left[X /G\right] \to \mathbb{B} G
\end{align}
and thus by Lemma \ref{Lem:FunctorialityHomotopyFixedPoints} there is a functor
\begin{align}
    \left[X /G\right]^{h \Gamma} \to (\mathbb{B} G)^{h \Gamma}.
\end{align}
Here $(\mathbb{B} G)^{h \Gamma}$ is the $\Gamma$-homotopy fixed points of $(\mathbb{B} G)$ in the sense of definition \ref{Def:HomotopyFixedPoints}, which is a category fibered in groupoids over $S$.
\begin{Lem} \label{Lem:Descent}
    The category fibered in groupoids $(\mathbb{B} G)^{h \Gamma}$ is a stack in groupoids over $S$.
\end{Lem}
\begin{proof}
Descent for morphisms in $(\mathbb{B} G)^{h \Gamma}$ follows directly from descent of morphisms in $\mathbb{B} G$ since the forgetful map $(\mathbb{B} G)^{h \Gamma} \to \mathbb{B} G$ is faithful by construction. Descent for objects in $(\mathbb{B} G)^{h \Gamma}$ follows from descent for objects in $\mathbb{B} G$ in combination with descent for morphisms in $\mathbb{B} G$, and the weak functoriality of the $\Gamma$ action.
\end{proof}

\subsubsection{} \label{Sec:AppendixQuotientStackFixedPoints}  Now let $G^{\Gamma}$ be the fixed point sheaf of $G$ and let $X^{\Gamma}$ be the fixed point sheaf of $X$. Then there is a natural morphism
\begin{align}
    \mathbb{B} G^{\Gamma} \to (\mathbb{B} G)^{h \Gamma}
\end{align}
of categories fibered in groupoids over $S$ which sends $(T,\mathcal{P})$ to $(T, \mathcal{P} \times^{G^{\Gamma}} G)$ equipped with the isomorphisms
\begin{align}
    \tau_{\gamma,0}:\left(\mathcal{P} \times^{G^{\Gamma}} G \right) &\to \left(\mathcal{P} \times^{G^{\Gamma}} G \right) \times^{G, \gamma} G \\
    (p,g) &\mapsto ((p, 1), \alpha_{\gamma}(g)),
\end{align}
and which sends a morphism
\begin{equation}
    \begin{tikzcd}
    \mathcal{P} \arrow{r}{b} \arrow{d} & \mathcal{P}' \arrow{d} \\
    T \arrow{r}{a} & T'.
    \end{tikzcd}
\end{equation}
to the induced diagram
\begin{equation}
    \begin{tikzcd}
    \mathcal{P} \times^{G^{\Gamma}} G \arrow{r}{b_G} \arrow{d} & \mathcal{P}' \times^{G^{\Gamma}} G \arrow{d} \\
    T \arrow{r}{a} & T',
    \end{tikzcd}
\end{equation}
where $b_G(p,g)=(b(p),g)$, which is well defined by the $G^{\Gamma}$-equivariance of $b$. Similarly there is a natural morphism
\begin{align}
    \left[X^{\Gamma} / G^{\Gamma} \right] \to \left[X / G\right]^{h \Gamma}
\end{align}
of categories fibered in groupoids over $S$ which sends $(T,\mathcal{P},\omega)$ to $(T, \mathcal{P} \times^{G^{\Gamma}} G, \{\tau_{\gamma,0}\}_{\gamma \in \Gamma}, \omega_G)$, where
\begin{align}
    \omega_G: \mathcal{P} \times^{G^{\Gamma}} G \to X
\end{align}
is defined by $\omega_G(p,g)=A(g,\omega(p))$. By construction, we see that the following diagram of categories fibered in groupoids over $S$ is strictly commutative 
\begin{equation} \label{Eq:KeyDiagram}
    \begin{tikzcd}
    \left[X^{\Gamma} / G^{\Gamma} \right] \arrow{r} \arrow{d} & \left[X /G\right]^{h \Gamma} \arrow{d} \\
    (\mathbb{B} G^{\Gamma}) \arrow{r} & (\mathbb{B} G)^{h \Gamma}.
    \end{tikzcd}
\end{equation}
\begin{Prop} \label{Prop:KeyPropositionAppendix}
The diagram in equation \eqref{Eq:KeyDiagram} is $2$-Cartesian.
\end{Prop}
We start by proving a lemma.
\begin{Lem} \label{Lem:FullyFaithfulII}
The natural map $\mathbb{B}G^{\Gamma} \to \left(\mathbb{B}G\right)^{h\Gamma}$ is fully faithful.
\end{Lem}
\begin{proof}
Since this is a morphism of categories fibered in groupoids over $S$, it suffices to prove that the map on fibers is fully faithful. Let $T \in S$, let $\mathcal{P}, \mathcal{P}'$ be $G^{\Gamma}$-torsors over $T$ and let $f:(T, \mathcal{P} \times^{G^{\Gamma}} G, \{\tau_{\gamma,0}\}_{\gamma \in \Gamma}) \to (T, \mathcal{P}' \times^{G^{\Gamma}} G, \{\tau'_{\gamma,0}\}_{\gamma \in \Gamma})$ be an isomorphism in $\left(\mathbb{B}G\right)^{h\Gamma}(T)$.

We need to show that $f$ is induced from a unique isomorphism $\mathcal{P} \to \mathcal{P}'$. If we can show that $f_2(p,g)=g$, then $f$ is induced from a unique isomorphism $\mathcal{P} \to \mathcal{P}'$; namely, the unique $G^{\Gamma}$-equivariant morphism sending $p$ to $f_1(p,1)$. It follows from the definitions of morphisms in $\left(\mathbb{B}G\right)^{h\Gamma}$ that for $\gamma \in \Gamma$ the following diagram must commute
\begin{equation}
    \begin{tikzcd}
        \left(\mathcal{P} \times^{G^{\Gamma}} G \right) \arrow{r}{\tau_{\gamma,0}} \arrow{d}{f} & \left(\mathcal{P} \times^{G^{\Gamma}} G \right) \times^{G, \gamma} G \arrow{d}{(f,1)} \\
        \left(\mathcal{P}' \times^{G^{\Gamma}} G \right) \arrow{r}{\tau'_{\gamma,0}} & \left(\mathcal{P}' \times^{G^{\Gamma}} G \right) \times^{G, \gamma} G.
    \end{tikzcd}
\end{equation}
Using the definition $\tau_{\gamma,0}(p,g)=(p,1,\gamma(g))$ this gives us the equality
\begin{align}
    ((f_1(p,g),1),\gamma(f_2(p,g)) = ((f_1(p,g),1),\gamma(g)).
\end{align}
Thus we see that $f_2(p,g)=g$ and thus $f$ is induced from the unique $G^{\Gamma}$-equivariant map $\mathcal{P} \to \mathcal{P}'$ sending $p \mapsto f_2(p,1)$. 
\end{proof}
\begin{proof}[Proof of Proposition \ref{Prop:KeyPropositionAppendix}]
Both the top horizontal map and the bottom horizontal map are fully faithfull, this follows from the definition and Lemma \ref{Lem:FullyFaithfulII}. The proposition then comes down to the following claim: Given an object $(T,\mathcal{P},\{\tau_{\gamma}\}_{\gamma \in \Gamma},\omega)$ of $\left[ X / G \right]^{h \Gamma}$ such that there is a $G^{\Gamma}$-torsor $\mathcal{Q} \to T$ and a $G^{\Gamma}$-equivariant morphism $f:\mathcal{Q} \to \mathcal{P}$ inducing an isomorphism
\begin{align}
    \mathcal{P} \times^{G^{\Gamma}} G \to \mathcal{Q}
\end{align}
under which $\{\tau_{\gamma}\}_{\gamma \in \Gamma}$ corresponds to $\{\tau_{\gamma,0}\}_{\gamma \in \Gamma}$, the composition of $f$ with $\omega$ factors through $X^{\Gamma}$. Let us write $\omega'$ for the map $\mathcal{Q} \times^{G^{\Gamma}} G \to X$ induced by $f$. In this situation the diagrams 
\begin{equation}
    \begin{tikzcd}
    \mathcal{Q} \times^{G^{\Gamma}} G \arrow{d}{\omega'}  \arrow{r}{\tau_{\gamma},0} & \left(\mathcal{Q} \times^{G^{\Gamma}} 
G \right) \times^{G,\gamma} G \arrow{d}{\omega'_{\gamma}} \\
    X \arrow{r}{\operatorname{Id}_X} & X
    \end{tikzcd}
\end{equation}
commute for all $\gamma \in \Gamma$, which means that for $q \in \mathcal{Q}$ we have
\begin{align}
    \omega'(q,1)&=\omega'_{\gamma}((q,1),1)=\beta_{\gamma}(\omega'(q,1))
\end{align}
for all $\gamma \in \Gamma$. We deduce that $\omega'$ maps $\mathcal{Q} \subset \mathcal{Q} \times^{G^{\Gamma}} G$ to $X^{\Gamma}$, which concludes the proof.
\end{proof}
\subsubsection{} We now study under what conditions the map $\mathbb{B}G^{\Gamma} \to (\mathbb{B} G)^{h \Gamma}$ is an equivalence. For this purpose, we introduce cocycles and cohomology groups. Consider the sheaf $\underline{Z}^1(\Gamma,G)$ whose $T$-points are given by the set of cocycles $Z^1(\Gamma,G(T))$ with values in $G(T)$. In other words, these are the functions $\sigma:\Gamma \to G(T)$ satisfying
\begin{align}
    \sigma(\gamma \cdot \gamma')=\sigma(\gamma) \cdot \alpha_{\gamma}(\sigma(\gamma')).
\end{align}
There is an action of $G$ on $\underline{Z}^1(\Gamma,G)$, which takes an element $g \in G(T)$ and a cocycle $\sigma \in Z^1(\Gamma, G(T))$ and sends it to the cocycle
\begin{align}
    \sigma'(\gamma) = g \cdot \sigma(\gamma) \cdot \alpha_{\gamma}(g^{-1}).
\end{align}
The quotient set for this action is per definition the first (nonabelian) cohomology group of $\Gamma$ with coefficients in $G(T)$, in formulas, 
\begin{align}
    Z^1(\Gamma, G(T))/G(T) = \rH^1(\Gamma, G(T)).
\end{align}
Given a cocycle $\sigma:\Gamma \to G(T)$ and $\gamma \in \Gamma$ we get an isomorphism of trivial $G_T$-torsors
\begin{align}
    \epsilon_{\gamma,\sigma}:G_T &\to G_{T} \times^{G, \gamma} G \\ 
    g &\mapsto (\alpha_{\gamma^{-1}}(g),\sigma(\gamma)).
\end{align}
To check that $(T,G_T,\{\epsilon_{\gamma,\sigma}\}_{\gamma \in \Gamma}))$ defines an object of $(\mathbb{B}G)^{h \Gamma}$ we need to check that the following diagram commutes
\begin{equation}
    \begin{tikzcd}
    G_T \arrow{d}{\epsilon_{\gamma_2 \gamma_1, \sigma}} \arrow{r}{\epsilon_{\gamma_2}} & G_T \times^{G,\gamma_2} G \arrow{d}{\mathbb{B}(\gamma_2)(\epsilon_{\gamma_1, \sigma})} \\
    G_T \times^{G,\gamma_2 \gamma_1} G \arrow[r,"\eta_{\mathbb{B}, \gamma_1, \gamma_2}"]&   \left(G_T \times^{G,\gamma_1} G \right) \times^{G,\gamma_2} G.
    \end{tikzcd}
\end{equation}
This comes down to the equality
\begin{align}
    (\alpha_{\gamma_2^{-1} \gamma_1}(g), \sigma(\gamma_1 \gamma_2)),1) = ((\alpha_{\gamma_2^{-1}} \alpha_{\gamma_1^{-1}}(g), \sigma(\gamma_1)), \sigma(\gamma_1))
\end{align}
which follows from the cocycle condition
\begin{align}
    \sigma(\gamma \cdot \gamma')=\sigma(\gamma) \cdot \alpha_{\gamma}(\sigma(\gamma')).
\end{align}
If we now consider $\underline{Z}^1(\Gamma,G)$ as a fibered category over $S$, the construction above defines a natural map $\underline{Z}^1(\Gamma,G) \to (\mathbb{B} G)^{h \Gamma}$ which takes $(T,\sigma \in Z^1(\Gamma, G(T))$ to $(T,G_T,\{\epsilon_{\gamma,\sigma}\}_{\gamma \in \Gamma})$. 
\begin{Lem} \label{Lem:CocyclesTorsors}
This map induces an isomorphism
\begin{align}
    \left[\underline{Z}^1(\Gamma,G)/G \right] \to (\mathbb{B} G)^{h \Gamma}.
\end{align}
\end{Lem}
\begin{proof}
This comes down to showing that the sheaf $\underline{Z}^1(\Gamma,G)$ represents the functor on $S$ sending $T$ to the set of $\Gamma$-equivariant structures on the trivial $G$-torsor over $T$. We have seen above how to go from a cocycle $\sigma \in \underline{Z}^1(\Gamma,G)(T)$ to a $\Gamma$-equivariant structure on the trivial $G$-torsor over $T$. Conversely, let $(\mathcal{P}^0 \times^{G^{\Gamma}} G, \{\tau_{\gamma}\}_{\gamma \in \Gamma})$, be a $\Gamma$-equivariant structure on the trivial $G$-torsor $\mathcal{P}^0 \times^{G^{\Gamma}} G$ where $\mathcal{P}^0$ is the trivial $G^{\Gamma}$-torsor. Then we get a cocycle by considering the collection of elements $\sigma(\gamma) \in G(T)$ defined to be the element of $G(T)$ corresponding to the following automorphism of $\mathcal{P}^0 \times^{G^{\Gamma}} G$ over $T$
\begin{equation}
    \begin{tikzcd}
    \mathcal{P}^0 \times^{G^{\Gamma}} G \arrow{r}{\tau_{\gamma}} & \left(\mathcal{P}^0 \times^{G^{\Gamma}} G \right) \times^{G,\gamma} G \arrow{r}{\tau_{\gamma,0}^{-1}} & \mathcal{P}^0 \times^{G^{\Gamma}} G.
    \end{tikzcd}
\end{equation}
Here $\tau_{\gamma,0}$ is the isomorphism 
    \begin{align}
    \tau_{\gamma,0}:\left(\mathcal{P}^0 \times^{G^{\Gamma}} G \right) &\to \left(\mathcal{P}^0 \times^{G^{\Gamma}} G \right) \times^{G, \gamma} G \\
    (p,g) &\mapsto ((p, 1), \gamma(g))
\end{align}
from the beginning of Section \ref{Sec:AppendixQuotientStackFixedPoints}. It is a straightforward yet laborious check that this defines an inverse to the construction above, and the lemma is proved.
\end{proof}
We let $\underline{H}^1(\Gamma, G)$ be the sheafification of the presheaf $T \mapsto \rH^1(\Gamma, G(T))$ on $S$. It follows from Lemma \ref{Lem:CocyclesTorsors} that there is a natural map $(\mathbb{B} G)^{h \Gamma} \to \underline{H}^1(\Gamma, G)$. This fits in a $2$-commutative diagram
\begin{equation} \label{Eq:TwoCartesianCohomologyIII}
    \begin{tikzcd}
            \mathbb{B} G^{\Gamma} \arrow{r} \arrow{d} & (\mathbb{B} G)^{h \Gamma} \arrow{d} \\
            S \arrow{r} & \underline{H}^1(\Gamma, G),
    \end{tikzcd}
\end{equation}
where the top horizontal map is the natural map, and the bottom horizontal map is the natural inclusion of the cohomology class of the trivial cocycle. We have the following fundamental result. 
\begin{Prop} \label{Prop:HomotopyFixedPointsClassifyingStack}
    The diagram in \eqref{Eq:TwoCartesianCohomologyIII} is $2$-Cartesian.
\end{Prop}
\begin{proof}
The top horizontal arrow is fully faithful by Lemma \ref{Lem:FullyFaithfulII}. It suffices to show that its essential image consists of those elements of $(\mathbb{B} G)^{\Gamma}$ whose image in $\underline{H}^1(\Gamma,G)$ is trivial. This can be checked locally in the Grothendieck topology on $S$ since both source and target are stacks in groupoids over $S$, see Lemma \ref{Lem:Descent}. 

Given an object of $(\mathbb{B} G)^{h \Gamma}$ over an object $T$ of $S$ whose image in $\underline{H}^1(\Gamma,G)$ is trivial, we replace $T$ by a cover to assume that the underlying $G$-torsor is trivial. Then the $\Gamma$-equivariant structure corresponds to a cocycle $\sigma:\Gamma \to G(T)$, and by assumption we may pass to a further cover of $T$ to assume that this cocycle is trivial. We now observe that the cocycle being trivial tells us precisely that the $\Gamma$-equivariant structure on the trivial $G$-torsor over $T$ is isomorphic to the \emph{trivial} $\Gamma$-equivariant structure on the trivial $G$-torsor. In other words, that the corresponding object of $(\mathbb{B} G)^{h \Gamma}(T)$ is in the essential image of $(\mathbb{B} G^{h \Gamma})(T)$.
\end{proof}
Putting together Proposition \ref{Prop:HomotopyFixedPointsClassifyingStack} and Proposition \ref{Prop:KeyPropositionAppendix} we obtain the following corollaries. 
\begin{Cor} \label{Cor:DoubleCartesian}
    There is a natural $2$-Cartesian diagram
    \begin{equation}
        \begin{tikzcd}
            \left[X^{\Gamma}/G^{\Gamma} \right] \arrow{d} \arrow{r} & \left[ X / G \right]^{\Gamma} \arrow{d} \\
            S \arrow{r} & \underline{H}^1(\Gamma,G).
        \end{tikzcd}
    \end{equation}
\end{Cor}
\begin{Cor} \label{Cor:NoCohomologyIsomorphism}
If $\underline{H}^1(\Gamma, G)$ is trivial, then the natural map
\begin{align}
    \left[X^{\Gamma} / G^{\Gamma} \right] \to \left[X /G\right]^{h \Gamma}
\end{align}
is an equivalence.
\end{Cor}

\DeclareRobustCommand{\VAN}[3]{#3}
\bibliographystyle{amsalpha}
\bibliography{references}
\end{document}